\documentclass[psamsfont,a4paper,11pt]{article}
\usepackage[english]{babel} 
\usepackage{afterpage} 
\usepackage{amscd}
\usepackage{amsfonts}
\usepackage{amsmath}
\usepackage{amssymb}
\usepackage{amsthm}
\usepackage{blindtext}
\usepackage{booktabs}
\usepackage[usenames,dvipsnames]{color}
\usepackage{enumitem}
\setlist[enumerate]{label=(\alph*)} 
\usepackage{hyperref}
\usepackage[latin1]{inputenc} 
\usepackage{graphicx}
\usepackage{longtable}	
\usepackage{mathdots}
\usepackage{mathrsfs}
\usepackage[fixamsmath]{mathtools}
\usepackage{multirow}
\usepackage{multicol}
\usepackage{latexsym} 	
\usepackage{setspace}
\usepackage{stackengine}
\usepackage{stackrel}       
\usepackage{stmaryrd}  
\usepackage{tabu}
\usepackage{tabularx}	
\usepackage[flushleft]{threeparttable}
\usepackage{tikz}
\usepackage{tikz-cd}

\allowdisplaybreaks

\usetikzlibrary{shapes,arrows}

\def\a{\alpha} 
\def\b{\beta} 
\def\d{\delta} 
\def\e{\epsilon} 
\def\g{\gamma} 
\newcommand{\C}{\mathbb{C}} 
\newcommand{\Z}{\mathbb{Z}} 
\newcommand{\NN}{\mathbb{N}}

\newcommand{\ideal}[1]{\langle #1 \rangle}
\newcommand{\z}{\zeta}
\newcommand{\vare}{\varepsilon}

\newcommand{\End}{\text{End}}

\newcommand{\coch}[1]{\vare_{#1}^\vee}
\newcommand{\twid}[1]{\widetilde{#1}}

\newcommand{\zz}{\boldsymbol{z}}
\newcommand{\zn}{\boldsymbol{z}^{n_Q\a^\vee}}

\newcommand{\nequiv}{\not\equiv}
\newcommand{\Uv}[1]{U_{\sqrt{v}}(\widehat{\frak{gl}}({#1}))}
\newcommand{\Uq}[1]{U_q(\widehat{\frak{gl}}({#1}))}

\newcommand{\UqF}[1]{U_q^F(\widehat{\frak{gl}}({#1}))}

\newcommand\xrowht[2][0]{\addstackgap[.5\dimexpr#2\relax]{\vphantom{#1}}}

\newtheorem{Thm}{Theorem}[section]		
\newtheorem{Lemma}[Thm]{Lemma} 
 
\newtheorem{Prop}[Thm]{Proposition} 
\newtheorem{Cor}[Thm]{Corollary}

\newtheorem*{thm}{Main Result}	

\theoremstyle{definition}
\newtheorem{Def}[Thm]{Definition}
\newtheorem*{defn}{Definition}  

\theoremstyle{remark}
\newtheorem{Remark}[Thm]{Remark} 
\newtheorem{Example}[Thm]{Example}

\newcommand{\gammaice}[6]{
\begin{tikzpicture}[scale=0.77, every node/.style={scale=0.8}]
\coordinate (a) at (-.75, 0);
\coordinate (b) at (0, .75);
\coordinate (c) at (.75, 0);
\coordinate (d) at (0, -.75);
\coordinate (aa) at (-.75,.5);
\coordinate (cc) at (.75,.5);
\draw (a)--(c);
\draw (b)--(d);
\draw[fill=white] (a) circle (.25);
\draw[fill=white] (b) circle (.25);
\draw[fill=white] (c) circle (.25);
\draw[fill=white] (d) circle (.25);
\node at (0,1) { };
\node at (a) {$#1$};
\node at (b) {$#2$};
\node at (c) {$#3$};
\node at (d) {$#4$};
\node at (aa) {$#5$};
\node at (cc) {$#6$};
\end{tikzpicture}}

\newcommand{\gamgam}[8]{
\begin{tikzpicture}[scale=0.8, every node/.style={scale=0.8}]
\coordinate (a) at (-.75, -.75);
\coordinate (b) at (-.75, .75);
\coordinate (c) at (.75, .75);
\coordinate (d) at (.75, -.75);
\coordinate (aa) at (-.75, -1.25);
\coordinate (bb) at (-.75, 1.25);
\coordinate (cc) at (.75, 1.25);
\coordinate (dd) at (.75, -1.25);
\draw (a)--(c);
\draw (b)--(d);
\draw[fill=white] (a) circle (.25);
\draw[fill=white] (b) circle (.25);
\draw[fill=white] (c) circle (.25);
\draw[fill=white] (d) circle (.25);
\node at (0,1) { };
\node at (a) {$#1$};
\node at (b) {$#2$};
\node at (c) {$#3$};
\node at (d) {$#4$};
\node at (aa) {$#5$};
\node at (bb) {$#6$};
\node at (cc) {$#7$};
\node at (dd) {$#8$};
\end{tikzpicture}}

\newcommand{\lhs}[9]{\raisebox{-30pt}{\resizebox{5.4cm}{4.5cm}{\begin{tikzpicture}
\path[use as bounding box](-1,-2) rectangle(5,3);
\coordinate (a) at (0,-1);     
\coordinate (b) at (0,-1.5);   
\coordinate (c) at (0,1);      
\coordinate (d) at (0,1.5);    %
\coordinate (e) at (3,2);      
\coordinate (f) at (4,1);      
\coordinate (g) at (4,1.5);    %
\coordinate (h) at (4,-1);     
\coordinate (i) at (4,-1.5);   %
\coordinate (j) at (3,-2);     
\coordinate (k) at (2,1);      
\coordinate (l) at (2,1.5);    %
\coordinate (m) at (2,-1);     
\coordinate (n) at (2,-1.5);   %
\coordinate (o) at (3,0);      
\coordinate (p) at (1,0);      
\coordinate (q) at (3,1);      %
\coordinate (r) at (3,-1);     %
\draw (a) to [out=0, in=180] (k);
\draw (k)--(f);
\draw (c) to [out=0, in=180] (m);
\draw (m)--(h);
\draw (e)--(j);
\draw[fill=white] (a) circle (.25);
\draw[fill=white] (c) circle (.25);
\draw[fill=white] (e) circle (.25);
\draw[fill=white] (f) circle (.25);
\draw[fill=white] (h) circle (.25);
\draw[fill=white] (j) circle (.25);
\draw[fill=white] (k) circle (.25);
\draw[fill=white] (m) circle (.25);
\draw[fill=white] (o) circle (.25);
\path[fill=white] (p) circle (.25);
\path[fill=white] (q) circle (.35);
\path[fill=white] (r) circle (.35);
\node at (a) {$#1$};
\node at (c) {$#2$};
\node at (e) {$#3$};
\node at (f) {$#4$};
\node at (h) {$#5$};
\node at (j) {$#6$};
\node at (k) {$#7$};
\node at (m) {$#8$};
\node at (o) {$#9$};
\node at (p) {$\scriptstyle R_{z_i,z_j}$};
\node at (q) {$z_j$};
\node at (r) {$z_i$};
\end{tikzpicture}}}}

\newcommand{\rhs}[9]{\raisebox{-30pt}{\resizebox{5.4cm}{4.5cm}
{\begin{tikzpicture}
\path[use as bounding box](-1,-2) rectangle(5,3);
\coordinate (a) at (0,-1); 
\coordinate (b) at (0,1);  
\coordinate (c) at (1,2);  
\coordinate (d) at (4,1);  
\coordinate (e) at (4,-1); 
\coordinate (f) at (1,-2); 
\coordinate (g) at (2,-1); 
\coordinate (h) at (2,1);  
\coordinate (i) at (1,0);  
\coordinate (r) at (3,0);  
\coordinate (s) at (1,-1);  
\coordinate (t) at (1,1);  
\draw (a)--(g) to [out=0, in=180] (d);
\draw (b)--(h) to [out=0, in=180] (e);
\draw (c)--(f);
\draw[fill=white] (a) circle (.25);
\draw[fill=white] (b) circle (.25);
\draw[fill=white] (c) circle (.25);
\draw[fill=white] (d) circle (.25);
\draw[fill=white] (e) circle (.25);
\draw[fill=white] (f) circle (.25);
\draw[fill=white] (g) circle (.25);
\draw[fill=white] (h) circle (.25);
\draw[fill=white] (i) circle (.25);
\path[fill=white] (r) circle (.25);
\path[fill=white] (s) circle (.35);
\path[fill=white] (t) circle (.35);
\node at (a) {$#1$};
\node at (b) {$#2$};
\node at (c) {$#3$};
\node at (d) {$#4$};
\node at (e) {$#5$};
\node at (f) {$#6$};
\node at (g) {$#7$};
\node at (h) {$#8$};
\node at (i) {$#9$};
\node at (r) {$\scriptstyle R_{z_i,z_j}$};
\node at (s) {$z_j$};
\node at (t) {$z_i$};
\end{tikzpicture}}}}

\newcommand{\botcharge}[2]{\raisebox{-29pt}{$\begin{array}{cc}#1\\#2\end{array}$}}

\newcommand{\topcharge}[2]{\raisebox{23pt}{$\begin{array}{cc}#2\\#1\end{array}$}}

\title{Yang-Baxter Equations for General Metaplectic Ice}
\author{Claire Frechette}
\date{\today}

\begin{document}

\maketitle

\begin{abstract}
In this paper, we extend results connecting quantum groups to spherical Whittaker functions on metaplectic covers of $GL_r(F)$, for $F$ a nonarchimedean local field.  Brubaker, Buciumas, and Bump showed that for a certain metaplectic $n$-fold cover of $GL_r(F)$ a set of Yang-Baxter equations model the action of standard intertwiners on principal series Whittaker functions. These equations arise from a Drinfeld twist of the quantum affine Lie superalgebra $\Uv{n},$ where $v = q^{-1}$ for $q$ the cardinality of the residue field. We extend their results to all metaplectic covers of $GL_r(F)$, providing new solutions to Yang-Baxter equations matching the scattering matrix for the associated Whittaker functions. Each cover has an associated integer invariant $n_Q$ and the resulting solutions are connected to the quantum group $\Uv{n_Q}$ and quantum superalgebra $\Uv{1|n_Q}$.
\end{abstract}

\section{Introduction}\label{intro}


The Whittaker functions for principal series representations of reductive groups over local fields play an integral role in number theory and representation theory. Among other interesting connections, they appear in the computations of local $L$-factors of integral representations of automorphic forms at unramified places, in the computation of Fourier coefficients of Eisenstein series, and in the Langlands-Shahidi method. Fascinatingly, these applications persist in the \emph{metaplectic case}, where we consider representations on metaplectic covers of a $p$-adic reductive group, despite the fact that global Fourier-Whittaker coefficients are no longer Euler products in this case. The explicit computation of these coefficients has been a source of important applications.
For example, in \cite{BFHeisAnnals}, Bump, Friedberg, and Hoffstein study Eisenstein series on metaplectic covers of $GSp_4$ over a global field in order to prove non-vanishing results for twisted $L$-functions with applications to the Birch Swinnerton-Dyer Conjecture. Additionally, Eisenstein series on a double cover of $GL_6$ appeared in Chinta's work \cite{Chintabiquad} on mean values of biquadratic zeta functions. 

Returning to the local setting, let $F$ be a nonarchimedean local field. In the algebraic case (i.e., non-metaplectic case), the spherical Whittaker function for an unramified principal series representation is unique up to scalar multiplication, and its values were computed by Shintani \cite{Shintani} for $GL_r(F)$ and Casselman and Shalika \cite{CassShal} in general. Casselman and Shalika's method relied on computing the effect of intertwining operators for these Whittaker models, which boils down to the computation of a  scalar, more precisely a rational function in the Langlands parameters for the representation. In the metaplectic case, we must require that $F$ contains suitably many roots of unity. In this setting, Whittaker functionals for metaplectic principal series are rarely unique, and thus intertwining operators act in a more complicated way; for covers of $GL_r(F)$, Kazhdan and Patterson \cite{KP} determined the effect of intertwining operators on a natural basis of Whittaker functionals, whose structure constants we call the \emph{scattering matrix coefficients}. These scattering matrix coefficients were a key ingredient in their investigation of generalized theta series, as well as the linchpin in the construction of a metaplectic Casselman-Shalika formula by Chinta and Offen \cite{COcassshal} for covers of $GL_r(F)$ and McNamara \cite{McNCS} for general reductive groups.

More recently, Brubaker, Buciumas, and Bump have made connections between metaplectic Whittaker functions and quantum groups \cite{BBBIce} that equate the scattering matrices of certain Whittaker functions to Drinfeld twists of $R$-matrices arising from affine quantum groups. This observation was made by connecting Whittaker functions to the six-vertex lattice model, and was proven for one specific $n$-fold metaplectic cover of $GL_r(F)$ over a $p$-adic field $F$ containing the $2n$-th roots of unity. Furthermore, the associated lattice model was shown to be solvable using the Drinfeld twist of an $R$-matrix on a quantum affine superalgebra, a particular kind of quasitriangular Hopf algebra with an additional graded structure. Applying a Drinfeld twist to such a structure preserves its quasitriangularity, i.e. preserves the existence of a universal $R$-matrix. 

In this paper, we prove that analogous results hold for any metaplectic cover of $GL_r(F)$ with $F$ as above. Using a similar connection to lattice models, we express Whittaker functions on any metaplectic cover as the partition function of a solvable lattice model, which then can be viewed as a model for modules of a quantum group. There are two major hurdles in extending the lattice model phenomenon to general covers: first, the Whittaker functions depend on a set of factors of $n$ controlled by the choice of cover of $G$, so we must rework the Boltzmann weights of the ice model in \cite{BBBIce} to reflect this dependance. Second, the Whittaker function coefficients rely on a family of Gauss sums which we must incorporate into the model and show to be permissible factors in the Drinfeld twist procedure in order to produce a solvable lattice model. Both of these components affect the resulting quantum group, and will be important considerations in extending this work to covers of any other reductive groups. 

Additionally, we note that many of the results in \cite{BBBIce} rely on previous results which also only consider the case of that specific $n$-fold metaplectic cover, such as those in \cite{McNcrystal}, so we must also modify these results for generic metaplectic covers as well. Broadening these findings to the more general case shows that this approach is much more widely applicable than previously known -- for instance, many choices of covers allow us to weaken the conditions on our field $F$ to contain only the $n$-th roots of unity -- and suggest that one could attempt to push these methods further to remove such dependencies. 

Let $G = GL_r(F)$ be a split reductive group over a local field $F$. For any positive integer $n$, Brylinski and Deligne \cite{BD} associate a family of $n$-fold metaplectic covers of any split reductive group to a quadratic form $Q$ with certain nice properties; let $\twid{G}$ denote the associated metaplectic cover of $G$ and let $n_Q = n/\gcd(n, Q(\a^\vee))$, where $\a^\vee$ is any simple coroot. We will see that this integer $n_Q$ controls many of the arithmetic properties of the cover. The construction of these covers is considered in detail in Section \ref{metchar} and explicitly parametrized for our choice of $G$. The principal series representations for $\twid{G}$ are indexed by characters $\chi_{\boldsymbol{z}}$ of a subgroup of the metaplectic torus $\twid{T}$, where ${\boldsymbol{z}}$ denotes the Langlands parameter of the character; using $\chi_{\boldsymbol{z}}$, we define an unramified principal series representation $\pi_{\boldsymbol{z}}:= \pi_{\chi_{\boldsymbol{z}}}$ with a vector space $\frak{W}^{\boldsymbol{z}}:=\frak{W}^{\chi_{\boldsymbol{z}}}$ of spherical Whittaker functions; see Section \ref{metwhit} for the details of these constructions. Unlike in the algebraic case, this space of metaplectic Whittaker functions is not usually one-dimensional, but is still finite dimensional, as computed by Kahzdan and Patterson \cite{KP}. 

\begin{thm} For any choice of $n$ and $Q$ defining a metaplectic cover as above, there exists a solvable lattice model whose partition function equals the values of Whittaker functions $\mathcal{W}_\g^{\boldsymbol{z}}$ in an explicit basis $\{\g\}$ of the space of Whittaker functions $\frak{W}^{\boldsymbol{z}}$. The quantum group associated to the $R$-matrix is the quantum affine Lie superalgebra $U_{\sqrt{v}}(\widehat{\frak{gl}}(1|n_Q))$.\end{thm}

To state this result precisely, we break it into its major components: the lattice model is given in Section \ref{metice}, and Theorems \ref{ybe1}, \ref{ybe2}, and \ref{ybe3} prove its solvability. In Section \ref{Drin}, we explore the $R$-matrices for evaluation modules of the quantum superalgebra $U_{\sqrt{v}}(\widehat{\frak{gl}}(1|n_Q))$ and prove that the weights on the ice model come from a Drinfeld twist on the $R$-matrix of this superalgebra. Sections \ref{whitcalcs} and \ref{icewhit} calculate the Whittaker function and prove that the partition function produces the same values, culminating in Theorems \ref{BIGTHM} and \ref{indivWhit}, which give an explicit normalization for this connection.

Section \ref{pfmain} is reserved for proving that this connection between Whittaker functions and lattice models respects the intertwining operators on the space of Whittaker functions. This is the second main result of the paper, but to state it precisely we first must introduce some notation. For a simple reflection $s_i$ in the Weyl group of $G$, let $\overline{\mathcal{A}_{s_i}}: \pi_{\boldsymbol{z}} \rightarrow \pi_{s_i\boldsymbol{z}}$ denote the normalized intertwining operator on metaplectic principal series, defined precisely in Section \ref{metwhit}. As studied by Kazhdan and Patterson \cite{KP}, these intertwiners induce maps $\frak{W}^{\boldsymbol{z}} \rightarrow \frak{W}^{s_i\boldsymbol{z}}$ on the spaces of Whittaker functions. Let $U:=U_{\sqrt{v}}(\widehat{\frak{gl}}(n_Q))$ be the quantized enveloping algebra of the untwisted affine Lie algebra $\widehat{\frak{gl}}(n_Q)$, the central extension of the loop algebra of $\frak{gl}(n_Q)$. While it is standard to denote this algebra using $U_q$, we reserve $q$ for the cardinality of the residue field of $F$, and we set $v = q^{-1}$. For $V, W$ vector spaces, let $\tau = \tau_{V,W}: V\otimes W \rightarrow W\otimes V$ be the standard swapping map $v\otimes w \mapsto w\otimes v$. For any $U$-modules $V$ and $W$, there exists an $R$-matrix $R_{V,W} \in \End(V\otimes W)$ such that $\tau R_{V,W}: V\otimes W \rightarrow W\otimes V$ is a module homomorphism. We discuss the precise form of these $R$-matrices for specific choices of $V$ and $W$ in Sections \ref{metice} and \ref{Drin}; in particular, the $R$-matrices we consider satisfy parametrized Yang-Baxter equations (see Jimbo \cite{Jimbo}, or Frenkel and Reshetikhin \cite{FR}).  For a complex parameter $z$, let $V_+(z)$ be the $n_Q$-dimensional evaluation module of $\frak{gl}(n)$, which is a $U$-module with basis $v_\a(z)$. We can expand the $R$-matrix $R_{z_i,z_j}$ on the space $V(z_i)\otimes V(z_j)$ in terms of our basis elements according to the formula
\[ R_{z_i,z_j} \left( v_\a(z_i)\otimes v_\b(z_j)\right) = \sum_{\g,\d} R^{\g,\d}_{\a,\b} (z_i,z_j) \cdot v_\g(z_i) \otimes v_\d(z_j)\]
where $R^{\g,\d}_{\a,\b}(z_i,z_j)$ are rational functions in complex parameters $z_i$ and $z_j$.

\begin{Thm}\label{maintheorem}
Let $\twid{G}$ be any metaplectic cover of $GL_r(F)$, with corresponding quadratic form $Q$. There exists a Drinfeld twist of the $R$-matrix for $U_{\sqrt{v}}(\widehat{\frak{gl}}(n_Q))$ and a homomorphism $\theta_{\boldsymbol{z}}$ from the space $\frak{W}^{\boldsymbol{z}}$ of Whittaker functions with $\boldsymbol{z}:=(z_1,...,z_r)$ to the vector space $V_+(z_1)\otimes \cdots\otimes V_+(z_r)$ such that the following diagram commutes:
\begin{center}
\begin{tikzcd}
\frak{W}^{\boldsymbol{z}} \arrow[r,"\theta_{\boldsymbol{z}}"] \arrow[d,"\overline{\mathcal{A}}_{s_i}" left]& V_+(z_1) \otimes \cdots \otimes V_+(z_i) \otimes V_+(z_{i+1}) \otimes \cdots \otimes V_+(z_r)\arrow[d, "\left(\tau R_{z_i,z_{i+1}}\right)_{i,i+1}"]\\
\frak{W}^{s_i\boldsymbol{z}} \arrow[r,"\theta_{s_i\boldsymbol{z}}"] &V_+(z_1) \otimes \cdots \otimes V_+(z_{i+1}) \otimes V_+(z_i) \otimes \cdots \otimes V_+(z_r)
\end{tikzcd}
\end{center}
where $(\tau R_{z_i,z_{i+1}})_{i,i+1}$ applies the map $\tau R_{z_i,z_{i+1}}$ to the $i$th and $(i+1)$st tensor components and the identity map to the remaining ones.
\end{Thm}

Note that although the solvable lattice models provide the bridge to connect Whittaker functions to quantum groups, they do not appear in the final form of Theorem \ref{maintheorem}, and recourse to the lattice model is not necessary for the proof of this theorem. This bypass is remarkable, and suggests that the connection between Whittaker functions on metaplectic covers and quantum groups might extend to other reductive groups.

\textbf{Acknowledgements}: This work was supported by NSF grant DMS-1745638. The author thanks Ben Brubaker for his mentorship and support, as well as Dan Bump, Dihua Jiang, and Pavlo Pylyavskyy for helpful comments. Finally, we thank Gurbir Dhillon for making us aware of the computation (\ref{slrcounteqn}) and related facts about the classification of covers.

\section{Characterization of Metaplectic Covers of $GL_r$}\label{metchar}

Let $G := GL_r(F)$, where $F$ is a non-archimedean local field with ring of integers $\frak{o}$ and uniformizer $\varpi$. Let $q$ be the cardinality of the residue field $\frak{o}/\varpi\frak{o}$ and $n$ a positive integer such that $q\equiv 1\pmod{2n}$, so $F$ contains the 2n-th roots of unity $\mu_{2n}$. Let $T := T(F)$ be a split maximal torus of $G$ and $Y:=\text{Hom}(\mathbb{G}_m,T)$ be the group of cocharacters of $T$. We also fix an embedding $\e: \mu_n \rightarrow\C^\times$.

Metaplectic central extensions $\widetilde{G}$ of $G$ by $\mu_n$ are given by exact sequences of the form
\[1 \rightarrow \mu_n \rightarrow \widetilde{G}\xrightarrow{p} G\rightarrow 1,\]
and are in correspondence with symmetric $W$-invariant bilinear forms $B:  Y\times Y \rightarrow \Z$ for which the associated quadratic form $Q(\a^\vee) := B(\a^\vee,\a^\vee)/2$ has $Q(\a^\vee) \in \Z$ for all coroots $\a^\vee$ \cite{BD}. Equivalently, the metaplectic groups $\twid{G}$ may be thought of as the set $G\times \mu_n$ with multiplication given by a cocycle $\sigma \in H^2(G,\mu_n)$. This construction has been well studied in various levels of generality by Kazhdan and Patterson \cite{KP}, Matsumoto \cite{Matsumoto}, Brylinski-Deligne \cite{BD}, McNamara \cite{McNrep}, Gan, Gao, and Weissman \cite{GanGao, Ast}, and others. In this paper, we consider the metaplectic covers of $GL_r(F)$, which were originally enumerated by Kazhdan and Patterson \cite{KP}, and we loosely follow the construction of McNamara \cite{McNrep}.

\begin{defn}
A \emph{$n$-fold metaplectic cover} $\twid{G}$ of $G$ is a non-algebraic central extension given by the short exact sequence
\[1 \rightarrow \mu_n \rightarrow \widetilde{G}\xrightarrow{p} G\rightarrow 1\]
or more generally, a non-algebraic central extension given by the short exact sequence
\[1 \rightarrow \mu_m \rightarrow \widetilde{G}\xrightarrow{p} G\rightarrow 1\]
with $n|m$, such that the corresponding cocycle $\sigma \in H^2(G,\mu_m)$ is \emph{essentially of degree $n$}, in that $[\sigma^n]$ becomes trivial in $H^2(G,\C^\times)$ under the inclusion induced by an embedding $\vare: \mu_m \rightarrow \C^\times.$
\end{defn}

For $n$-fold metaplectic covers of $GL_r(F)$, we can write down explicit formulas for their associated cocycles and bilinear forms. As in Section 5 of \cite{BBBFhecke}, we slightly modify McNamara's set-up by considering $n$-fold covers that are \emph{essentially of degree n} as well as ``purely" $n$-fold covers; this shift allows us to include the important example of the metaplectic cover corresponding to the Euclidean inner product for the standard identification of $Y$ with $\Z^r$. In fact, we will show that to accommodate the metaplectic cover corresponding to any symmetric $W$-invariant bilinear form on the cocharacter lattice, it suffices to consider $m = 2n$.

By \cite{McNrep}, a cocycle $\sigma \in H^2(G,\mu_n)$ restricts to coroot spaces $X_{\a^\vee}$, which are isomorphic to $SL(2,F)$. For each coroot $\a^\vee$, this restriction gives a cocycle in $H^2(SL(2,F),\mu_n)$, which is the image of the integer $Q(\a^\vee)$ under the map
\[ \xi: \Z \rightarrow H^2(SL(2,F),\mu_n).\]
For metaplectic extensions of $SL_r(F)$, this set of values $Q(\a^\vee)$ for positive simple coroots $\a^\vee$ completely determines the cocycle $\sigma$; that is, if two bilinear forms $B, B'$ have $Q(\a^\vee) = Q'(\a^\vee)$ for all positive simple coroots $\a$, they correspond to equivalent central extensions. For metaplectic extensions of $GL_r(F)$, we need additional data to specify the cocycle completely, but these values $Q(\a^\vee)$ still play an important part in our calculations.

For $GL_r(F)$, we can give explicit formulas for cocycles corresponding to $n$-fold metaplectic covers of $G$. Formulas for these cocycles on arbitrary elements of $G$ are rather complicated, but restricting to $T$, they become quite elegant. Fortunately, we may work in the other direction as well: we begin by examining a set of cocycles defined on the torus, and then extend them up to $G$ and show that this set of cocycles on $G$ suffices to give us any $n$-fold metaplectic cover of $G$. 

Let $( \cdot,\cdot)_{2n}$ be the $2n$-th Hilbert symbol (see Neukirch \cite{Neukirch}), and for $k|2n$, let $(\cdot,\cdot)_{k} = (\cdot,\cdot)_{2n}^{2n/k}$. For simplicity, when $k=n$, we will denote $(\cdot,\cdot)_n$ simply as $(\cdot,\cdot)$. 
Then we define $\sigma_{b,c} \in H^2(T,\mu_{2n})$ to be the cocycle given on the torus by the formula below; that is, for $\boldsymbol{x},\boldsymbol{y}\in T$, 
\begin{align}\label{cocycle}
\sigma_{b,c}(\boldsymbol{x},\boldsymbol{y}) = (\det(\boldsymbol{x}),\det(\boldsymbol{y}))_{2n}^c \prod_{i>j} (x_i,y_j)^{-b}
\end{align}  
for some $b,c \in \Z$, where $\boldsymbol{x} = \text{diag}(x_1,...,x_n)$. Note that this cocycle is a slight modification of the ones used by Kazhdan and Patterson \cite{KP}; like them, we use Matsumoto's theorem for $SL(n+1)$ to extend our cocycle from $T$ to $G$. Some choices of $b,c$ may give equivalent metaplectic extensions; we will explain later in this section when we may or may not determine this overlap. For precise computations of the cocycle on arbitrary elements of $\twid{G}$, we refer the reader to Sections 4 and 5 of \cite{BLS}.

\begin{Prop}\label{alltheforms}
The cocycle $\sigma_{b,c}$ on $G$ generates an $n$-fold metaplectic cover $\twid{G}$, and we may describe its associated bilinear form explicitly in terms of $b,c\in \Z$. That is, for $\boldsymbol{x},\boldsymbol{y}\in \twid{T}$ such that $p(\boldsymbol{x}) = x^\lambda$ and $p(\boldsymbol{y}) = y^\mu$ for $x,y\in F^\times$ and $\lambda,\mu\in Y$, we have
\begin{align}\label{commutator} [\boldsymbol{x}, \boldsymbol{y}] = (x,y)^{B(\lambda,\mu)},\end{align}
where, under the standard identification of $Y$ with $\Z^r$, we may express $B$ as the matrix (\ref{Bmatrix}). Furthermore, any such symmetric $W$-invariant bilinear form $B$ takes this form for some $b,c\in \Z$ and may thereby be linked back to a cocycle $\sigma_{b,c}$.
\end{Prop}

\begin{proof}
For $G:= GL_r(F)$, the torus $T$ is the subgroup of diagonal matrices, with cocharacter group $Y$ generated by the $r$ fundamental coweights $\vare_i^\vee: F^\times \rightarrow T$, given by $\vare_i^\vee(a) = \text{diag}(1,...,1, a,1,...,1)$, where $a$ is in the $i$-th entry. There are $r-1$ simple coroots: $\coch{1} - \coch{2}, \coch{2} -\coch{3},...,\coch{r-1} - \coch{r}$. In keeping with the usual notation, for $\lambda \in Y$ and $a\in F^\times,$ we set $a^\lambda:=\lambda(a)$. 

Regarding the second statement, since $Y$ has a finite basis, we can express a bilinear form $B: Y \times Y \rightarrow \Z$ in the form
\[ B(\a,\b) = \boldsymbol{x}^T A\boldsymbol{y}\]
for some matrix $A$, where $\a, \b$ correspond to the vectors $\boldsymbol{x},\boldsymbol{y}$, respectively. In our case, $B$ is symmetric, so $A$ must be a symmetric matrix. Furthermore, since the Weyl group $W$ associated to $G$ is isomorphic to the symmetric group $S_r$, the $W$-invariance of $B$ determines that $A$ must be invariant under conjugation by permutation matrices. Thus, the matrix $A = (a_{i,j})$ associated to a symmetric $W$-invariant bilinear form $B$ must have $a_{i,i} = c$ for all $i$ and $a_{i,j} = d$ for all $i\neq j$ for some pair $c,d\in \Z$. For our purposes, we will re-parametrize slightly: let $b = c-d \in \Z$. Let $B_{b,c}$ be the bilinear form corresponding to this parametrization. That is, the matrix $A$ corresponding to $B_{b,c}$ is of the form below,
\begin{align}\label{Bmatrix} A = \begin{bmatrix} c &  & \cdots &c-b \\
&c& &\vdots\\
\vdots&&\ddots&\\
c-b&\cdots&&c
\end{bmatrix}.\end{align}
Abusing notation, we will frequently conflate $B$ with $A$ and refer to the bilinear form as its associated matrix.

For the first statement, let $\sigma_{b,c}$ be the cocycle given on $T$ by Equation \ref{cocycle}. We can calculate the commutator of any two elements $\boldsymbol{x} = \text{diag}(x_1,...,x_r), \boldsymbol{y} = \text{diag}(y_1,...,y_r) \in T$ directly using our cocycle: 
\begin{align}
[\boldsymbol{x},\boldsymbol{y}] = \frac{\sigma_{b,c}(\boldsymbol{x}, \boldsymbol{y})}{\sigma_{b,c}(\boldsymbol{y},\boldsymbol{x})} = \prod_{i} (x_i,y_i)^c \prod_{i\neq j} (x_i,y_j)^{c-b}. \label{cocyclecalc}
\end{align}
In the event that $\boldsymbol{x} = x^\lambda, \boldsymbol{y} = y^\mu$ for some $x,y\in F^\times, \lambda,\mu \in Y$, we then have that $x_i= x^{\lambda_i}$ and $y_i=y^{\mu_i}$ for all $i$, yielding the commutator
\begin{align*}
[x^\lambda,y^\mu] &= \prod_{i} (x,y)^{\lambda_i\mu_ic} \prod_{i\neq j} (x,y)^{\lambda_i\mu_j(c-b)}\\
&= (x,y)^{\sum_{i= 1}^r \lambda_i\mu_ic + \sum_{i\neq j} \lambda_i\mu_j(c-b)}\\
& = (x,y)^{B_{b,c}(\lambda,\mu)}. 
\end{align*}
Thus, the central extension generated by $\sigma_{b,c}$ corresponds to $B_{b,c}$ under equation \ref{commutator} as desired. 

It remains to check that $[\sigma_{b,c}^n]$ is trivial in $H^2(G,\C^\times)$. Using Equation \ref{cocycle}, the cocycle $\sigma^n_{b,c}$ is given by
\[ (g,h) \mapsto (\det(g),\det(h))_2^c.\]
We construct a section $\gamma: G\rightarrow \C^\times$ that splits $\sigma_{b,c}^n$: for $g\in G$, we have $\det(g) \in F^\times$, and for any $a\in F^\times,$ we have a quadratic form $x^2 - ay^2$ on $F^2$. In \cite{Weil}, Weil constructed a map $\gamma$ from the space of quadratic forms to $\mu_8\subset \C^\times$. Define $\g_0(a) := \g(x^2 - ay^2)^c$. Then by the last formula on page 176 of \cite{Weil}, we have
\[ \frac{\g_0(a)\g_0(b)}{\g_0(ab)} = (a,b)_2^c.\]
Thus, the section $g\mapsto \g_0(\det(g))$ splits our cocycle in $H^2(G,\mu_8)$ and therefore in $H^2(G,\C^\times)$ as well.
\end{proof}

\begin{Remark} Many bilinear forms give precisely the same metaplectic cover, and many more give (morally) equivalent covers. Since Hilbert symbols $(\cdot,\cdot)_m$ give $m$-th roots of unity, bilinear forms $B_{b,c}$ and $B_{b',c'}$ with $b\equiv b' \pmod{n}$ and $c\equiv c'\pmod{2n}$ will give exactly the same cocycle, and thus the same metaplectic cover. Thus we may restrict our attention to $b\in \Z/n\Z$ and $c\in\Z/2n\Z$, giving us $2n^2$ distinct cocycles $\sigma_{b,c}$, some of which give equivalent metaplectic extensions. Let $\vare: \mu_{2n} \rightarrow \C^\times$ be any embedding. According to \cite{KP}, we can distinguish the metaplectic coverings $\twid{G}_j$ given by two cocycles $\sigma_{b_j,c_j}$ for $j = 1,2$ precisely when the induced extensions
\[ 1\rightarrow \C^\times \rightarrow \twid{G}_{j,\vare} \rightarrow G \rightarrow 1\]
are inequivalent, which happens when $2(c_1-c_2)\nequiv 0 \pmod{n}$ or $b_1\nequiv b_2 \pmod{n}$ and does not depend on the choice of embedding. Furthermore, notice that for $B_{b,c}$, the associated quadratic form $Q_{b,c}$ gives $Q_{b,c}(\a^\vee) = b$ for any simple coroot $\a^\vee$. Recall that a metaplectic extension of $SL_r(F)$ is completely determined by these values $Q(\a^\vee)$. Thus, we have that $B_{b,c}$ and $B_{b,b}$ give equivalent extensions of $SL_r(F)$. Since many of our later calculations will actually occur in $SL_r(F) \subset GL_r(F)$ or depend only on these values $Q(\a^\vee)$, for the purposes of this paper we may think of $B_{b,c}$ and $B_{b,b}$ as morally equivalent extensions for any $c$.
\end{Remark}

\section{The Whittaker Functions for Metaplectic Covers}\label{metwhit}

Now that we have constructed metaplectic covering groups of $GL_r(F)$, we fix a metaplectic cover $\twid{G}$ with bilinear form $B_{\twid{G}}$ for this section and construct its genuine unramified principal series representations, as well as the Whittaker functions on these representations. See \cite{McNrep} for a convenient source.

\begin{Def}A representation $\pi$ of $\twid{G}$ is  \emph{genuine} if the central $\mu_n$ acts by scalars via the fixed embedding $\e: \mu_n \rightarrow\C^\times$. That is to say, $\pi(\z g) = \e(\z) \pi(g)$ for all $\z\in \mu_n$. 
\end{Def}

We first construct genuine irreducible representations of the metaplectic torus $\twid{T} := p^{-1}(T)$. To this end, let $K$ be the maximal compact subgroup $GL_r(\frak{o})$ of $G$; then $T(\frak{o}) = T\cap K$. Let $\twid{T}(\frak{o})$ be the preimage of $T(\frak{o})$ in $\twid{G}$ and let $H$ be the centralizer of $\twid{T}(\frak{o})$ in $\twid{T}$. In fact, $H$ is the maximal abelian subgroup of $\twid{T}$ \cite{McNrep}. Now, let $\boldsymbol{s}: G\rightarrow \twid{G}$ be the standard section and let $\Lambda$ be the lattice defined by
\[ \Lambda := \{x \in Y: \boldsymbol{s}(\varpi^x)\in H\} = \{x\in Y : B_{\twid{G}}(x,y) \in n\Z \text{ for all } y\in Y.\}.\]

\noindent By \cite{McNrep}, we have $H = \Lambda \times \mu_n \times T(\frak{o})$ as a set. 

To construct a genuine irreducible representation of $\twid{T}$, we take a genuine character $\chi$ of $H$ that is trivial on $\twid{T}\cap K$, noting that our central extension splits over K \cite{McNrep}. That is, let $\chi$ be in the set
\begin{align}\label{chiset}\{\chi \in \text{Hom}(H/(\twid{T}\cap K), \C^\times) : \chi(\z) = \epsilon(\z) \text{ for all }\z\in \mu_n\}.\end{align}
Inducing $\chi$ from $H$ to $\twid{T}$ gives an irreducible representation $(\pi_\chi, i(\chi))$ on $\twid{T}$, by Theorem 5.1 of \cite{McNrep}.

Next, we build a representation of $\twid{G}$ from $i(\chi)$. Let $B$ be the standard Borel subgroup $B\supseteq T$ in $G$, and let $\twid{B}$ be its preimage in $\twid{G}$, not to be confused with the bilinear form $B_{\twid{G}}$. Inflate the representation $(\pi_\chi, i(\chi))$ to $\twid{B}$ under the canonical surjection $\twid{B} \rightarrow \twid{T}$, and induce again to get the unramified principal series $I(\chi) := \text{Ind}_{\twid{B}}^{\twid{G}} (i(\chi)).$ More concretely, $I(\chi)$ is the set of locally constant functions $f: \twid{G} \rightarrow i(\chi)$ such that
\[ f(bg) = (\d^{1/2}\chi)(b)f(g) \hspace{0.5cm} \text{ for all }g\in \twid{G}, b\in \twid{B}\]
where $\d$ is the modular quasicharacter of $B$. Note that $\twid{G}$ acts on $I(\chi)$ by right translation.

We then need to determine when these representations are irreducible. For $GL_r(F)$, the genuine characters in (\ref{chiset}) may be parametrized by the diagonal elements in $GL(r,\C)$, which we will denote as $r$-vectors $\boldsymbol{z}$: we can express every element $h\in H$ as $h= \z \boldsymbol{s}(t)$, where $\z\in \mu_n$ and $t = (t_1,...,t_r)\in T$. Define
\begin{align}\label{defchi}
\chi_{\boldsymbol{z}}(h) = \z \prod_{i=1}^r z_i^{\text{ord}(t_i)}.
\end{align}
We will denote the representation doubly induced from $\chi_{\boldsymbol{z}}$ as $(\pi_{\boldsymbol{z}}, I(\chi_{\boldsymbol{z}}))$. This does not depend uniquely on $\boldsymbol{z}$: in particular, if $\boldsymbol{z}^n = (\boldsymbol{z}')^n$, then $\pi_{\boldsymbol{z}} \cong \pi_{\boldsymbol{z}'}$. Furthermore, $\pi_{\boldsymbol{z}}$ is irreducible if and only if $\boldsymbol{z}^{n\a} \neq q^{\pm 1}$ for all coroots $\a$.

Let $(\pi_{\boldsymbol{z}}, I(\chi_{\boldsymbol{z}}))$ be irreducible. By Theorem 3.1 of \cite{McNCS}, $I(\chi_{\boldsymbol{z}})^K$ is one-dimensional; fix a vector $\phi_K^{\boldsymbol{z}} := \phi_K^{\chi_{\boldsymbol{z}}}$ in this space, which we shall henceforth call the \emph{spherical vector} for the representation $I(\chi_{\boldsymbol{z}})$. All the subsequent results are independent of this choice.

Let $U$ be the unipotent radical of $B$. By Matsumoto's description of the metaplectic cover in \cite{Matsumoto} we know that $\twid{G}$ splits over $U$, so we may regard the root subgroups $U_\a$ for positive roots $\a$ as subgroups of $\twid{G}$. Let $\Phi^+$ denote the positive roots and $\Phi^-$ the negative roots. Then for any element $w$ of the Weyl group, define the corresponding unipotent subgroup
\[ U_w:= \prod_{\a\in \Phi^+, w(\a) \in \Phi^-} U_\a.\]
Define the intertwining operator $\mathcal{A}_w: I(\chi_{\boldsymbol{z}}) \rightarrow I(w\chi_{\boldsymbol{z}})$ by
\begin{align*} \mathcal{A}_w(f)(g): = \int_{U_w} f\left(w^{-1}ug\right)du.
\end{align*}
(By abuse of notation, both $w\in W$ and a representative of $w$ in $K$ are labelled by $w$.)

As noted above, each representation $I(w\chi_{\boldsymbol{z}})$ has a spherical vector $\phi_K^{w\boldsymbol{z}}$; we may choose these to be compatible with the Weyl group action of the intertwining operators, so
\[ \mathcal{A}_w \phi_K^{\boldsymbol{z}} = c_w(\chi_{\boldsymbol{z}})\phi_K^{w\boldsymbol{z}}\]
where for simple reflections $s = s_{\a^\vee}$, 
\[ c_s(\chi_{\boldsymbol{z}})  = \frac{1-q^{-1}\boldsymbol{z}^{n_Q\a^\vee}}{1-\boldsymbol{z}^{n_Q\a^\vee}},\]
where $\boldsymbol{z}^{\mu^\vee}$ means that for $\boldsymbol{z} = (z_1,...,z_r)$ and  $\mu^\vee = \sum c_i\vare_i$ a coweight, we take $\boldsymbol{z}^{\mu^\vee} = \prod z_i^{c_i}$.  We see that $c_{w_1w_2}(\chi_{\boldsymbol{z}}) = c_{w_1}(w_2\chi_{\boldsymbol{z}})c_{w_2}(\chi_{\boldsymbol{z}})$ whenever $\ell(w_1w_2) = \ell(w_1) + \ell(w_2)$. Let $\overline{\mathcal{A}}_w$ be the normalized intertwining operator
\begin{align}\label{intertwine} \overline{\mathcal{A}}_w:= c_w(\chi_{\boldsymbol{z}})^{-1} \mathcal{A}_w.
\end{align}

Fix an unramified character $\psi: U \rightarrow \C^\times$. That is to say, for any simple reflection $s\in W$, the restriction of $\psi$ to the root subgroup $U_\a$ is trivial on $U_\a \cap K$ but nontrivial on any open subgroup of $U_s$ with larger abelianization than $U_s\cap K$.

\begin{Def} A \emph{Whittaker functional} (with respect to $\psi$) on the representation $(\pi,V)$ is a linear functional $W$ for which
\[ W(\pi(u)v) = \psi(u)W(v)\]
for all $u\in U, v\in V$.
\end{Def}

By Section 6 of \cite{McNCS}, there is a unique $i(\chi_{\boldsymbol{z}})$-valued Whittaker functional on $(\pi_{\boldsymbol{z}}, I(\chi_{\boldsymbol{z}}))$, which we shall denote $\boldsymbol{W}^{\boldsymbol{z}}$, given by
\begin{align}\label{whitdef}
\boldsymbol{W}^{\boldsymbol{z}} (\phi) := \int_{U^-} \phi(uw_0)\psi(u)du \hspace{0.5cm} \text{ for all } \phi\in I(\chi_{\boldsymbol{z}})
\end{align}
where $U^-$ is the opposite unipotent radical, viewed as a subgroup of $\twid{G}$. Note that this Whittaker functional is usually defined over $U$, rather than $U^-$, but we may define it over $U^-$ instead since there exists a $T$-invariant splitting of our metaplectic cover over $U$ by Deligne as explained in \cite{MW}, which is stated globally but applies locally as noted. 

We then have an isomorphism between $i(\chi_{\boldsymbol{z}})^*$ and the space of $\C$-valued Whittaker functionals on $I(\chi_{\boldsymbol{z}})$ given by composition:
\[ L \mapsto L\circ \boldsymbol{W}^{\boldsymbol{z}} \hspace{0.5cm} \text{ for all } L\in i(\chi_{\boldsymbol{z}})^*.\]
Thus, it is natural to index the Whittaker functionals of a representation $(\pi_{\boldsymbol{z}},I(\chi_{\boldsymbol{z}}))$ on $\twid{G}$ with a basis of $i(\chi_{\boldsymbol{z}})^*$. Following the construction in \cite{McNCS}, let $v_0 = \phi_K^{\boldsymbol{z}}(1) \in i(\chi_{\boldsymbol{z}})$. Let $\Gamma\subset (\Z/n_Q\Z)^r$ be a subset of coweights such that $ \{\varpi^\g : \g \in\Gamma\}$ is a set of coset representatives for $\twid{T}/H$. When unambiguous, we may identify $\Gamma$ with this set of representatives. Then the set $\{ \pi_{\boldsymbol{z}}(\varpi^\g)v_0 : \g\in \Gamma\}$ is a basis for $i(\chi_{\boldsymbol{z}})$. Let $\{L_\g^{\boldsymbol{z}} : \g\in\Gamma\}$ denote the corresponding dual basis of $i(\chi_{\boldsymbol{z}})^*$.

For $\varpi^\mu \in \twid{T}$, we can always write $\varpi^\mu = \overline{t}\cdot h$ for some $h\in H, \overline{t} \in \twid{T}/H$. That is to say, we can write $\mu =  \nu+ \b$ for $\b \in \Lambda,\nu\in \Gamma$. Then we have that 
\begin{align}\label{linfunc} L_\g^{\boldsymbol{z}}(\pi_{\boldsymbol{z}}(\varpi^\mu)v_0) = \begin{cases} \chi_{\boldsymbol{z}}(\varpi^\mu) &\text{ if } \nu = \g\\0 &\text{otherwise}\end{cases}.
\end{align}

Therefore, the set $\{ W^{\boldsymbol{z}}_\g = L_\g^{\boldsymbol{z}}\circ \boldsymbol{W}^{\boldsymbol{z}}: \g\in \Gamma\}$ gives us a basis of the space of Whittaker functionals. 

\begin{Def}\label{sphericalWhit} A \emph{spherical Whittaker function} on $(\pi_{\boldsymbol{z}}, I(\chi_{\boldsymbol{z}}))$ is the map $\twid{G} \rightarrow \C$ given by
\[ \mathcal{W}^{\boldsymbol{z}}_\g(g) = W^{\boldsymbol{z}}_\g(\pi_{\boldsymbol{z}}(g)\phi_K^{\boldsymbol{z}}).\]
Note that $\mathcal{W}^{\boldsymbol{z}}_\g$ satisfies
\[ \mathcal{W}^{\boldsymbol{z}}_\g(\z ugk) = \z \psi(u) \mathcal{W}^{\boldsymbol{z}}_\g(g)\]
for $\z\in \mu_n, u\in U, g\in \twid{G}, k\in K$. Let $\frak{W}^{\boldsymbol{z}}$ be the set of such Whittaker functions. By the metaplectic Iwasawa decomposition, $G = U\twid{T}K$, it suffices to compute the values of $\mathcal{W}^{\boldsymbol{z}}_\g$ on the torus elements $\twid{T}$. Furthermore, it is straightforward to check that $\mathcal{W}_\g^{\boldsymbol{z}}$ vanishes on torus elements $\varpi^\mu$ unless $\mu$ is a dominant coweight.  \end{Def}

\begin{Remark}\label{Dhilloncalc}
Since spherical Whittaker functions are in bijection with elements of $\twid{T}/H$, we might ask for the cardinality of these sets. For some $b,c$, we have that $B_{\twid{G}} = B_{b,c}$. We may restrict this cover to $\twid{SL}_r \subset \twid{G}$, on which one may check that 
\begin{align} |\twid{T}_{SL}/H_{SL}| = \frac{n^{r-1}}{(br,n)(b,n)^{r-2}}.\end{align}\label{slrcounteqn}
In the more general case, it is tricky to write down a closed formula for $|\twid{T}/H|$, but for reasons that will become clearer as we examine the computation of these Whittaker functions, we will see that on a given element of $\twid{G}$, most of them will be zero. So we merely note that $\twid{T}_{SL}/H_{SL}$ and $\twid{T}/H$ share a common subgroup with with order $(n/(b,n))^{r-1}$ and so the total number of Whittaker functions may be expressed as 
\[ |\twid{T}/H| = (\ast)_{b,c} \cdot \frac{n^{r-1}}{(b,n)^{r-1}}\]
where the multiplier $(\ast)_{b,c}$ depends on $b,c,$ and $n$.
\end{Remark}

\section{Calculating Whittaker Functions on Metaplectic Covers}\label{whitcalcs}

We have several methods for studying the values of spherical Whittaker functions. One method relies on intertwining operators for principal series representations: Casselman and Shalika \cite{CassShal} followed this method for linear groups, as did Kazhdan and Patterson \cite{KP}, McNamara \cite{McNCS}, Chinta and Offen \cite{COcassshal}, and Brubaker, Buciumas, and Bump \cite{BBBIce} for various metaplectic covering groups. This approach uses the fact that for $w$ an element of the Weyl group, $\boldsymbol{W}^{w\chi} \circ \overline{\mathcal{A}}_{w}$ is an $i(\chi_{\boldsymbol{z}})$-valued Whittaker functional on $I(\chi_{\boldsymbol{z}})$. Thus, composing it with an element of $i(\chi_{\boldsymbol{z}})^*$, we get a $\C$-valued Whittaker functional, which we can expand in our chosen basis as follows:
\begin{align}\label{intonWhit} \mathcal{W}_\nu^{w\boldsymbol{z}} \circ \overline{\mathcal{A}}_w = \sum_{\mu\in\Gamma} \tau_{\nu,\mu}^{(w)}\mathcal{W}_\mu^{\boldsymbol{z}}
\end{align}
where $\tau_{\nu,\mu}^{(w)}$ are (up to normalization) rational functions supported on $n_Q$-th powers of $\boldsymbol{z}$. We call these structure constants $\tau_{\nu,\mu}^{(w)}$ the \emph{scattering matrix coefficients} of the space of Whittaker functions. Since the reflections $s_{\a^\vee}$ corresponding to simple coroots $\a^\vee$ generate the Weyl group, it suffices to compute $\tau_{\nu,\mu}^{(s_{\a^\vee})}$, which was done for the metaplectic covers of $GL_r$ by Kazhdan and Patterson in \cite{KP} and reworked for more general groups by McNamara in \cite{McNCS}. 

Recall that $n_Q = n/\text{gcd}(n,Q(\a^\vee))$ for any simple coroot $\a^\vee$, and let $g(a)$ denote the normalized $n$-th Gauss sum given by 
\begin{align}\label{gausssum} g(a,b) = \frac{1}{q}\int_{\frak{o}_F^\times} (\varpi,t)^a\psi(\varpi^b t) dt,
\end{align}
where $(\cdot,\cdot)$ denotes the $n$-th power Hilbert symbol \cite{Neukirch}. We will primarily use these Gauss sums with $b=-1$, so define $g(a) := g(a,-1)$, but we require the full definition for Section \ref{icewhit}.

\begin{Remark}
This Gauss sum has the following properties:
\begin{itemize}
\item If $b < -1$, then $g(a,b) = 0$.
\item If $b \geq 0$, then $g(a,b) = 1 - q^{-1}$ if $a\equiv0\pmod{n}$ and 0 otherwise.
\item The argument $a$ depends only on the residue class $a \pmod{n}$. In particular, $g(0) = -q^{-1}$, and $g(a)g(n-a) = q^{-1}$ if $n$ does not divide $a$.
\item  Comparing to the normalized $n$-th Gauss sum $\frak{g}(a)$ used in \cite{McNCS}, we have $g(a) = q^{-1}\frak{g}(-a)$.
\end{itemize}
\end{Remark}

\begin{Prop}[Kazhdan-Patterson, \cite{KP}, Lemma I.3.3; McNamara, \cite{McNCS}, Theorem 13.1]\label{TAU} Let $s:=s_{\a^\vee}$ be a simple reflection with corresponding coroot $\a^\vee$ and let $\mu,\nu\in \Gamma$. The structure constants $\tau_{\nu,\mu} := \tau^{(s)}_{\nu,\mu}$ can be broken into two pieces:
\[ \tau_{\nu,\mu} = \tau_{\nu,\mu}^1 + \tau_{\nu,\mu}^2\]
where $\tau^1$ vanishes unless $\nu\sim \mu \mod{n_Q\Lambda}$ and $\tau^2$ vanishes unless $\nu \sim s(\mu) + \a^\vee \mod{n_Q\Lambda}$. In these cases,
\[ \tau_{\mu,\mu}^1 = (1- q^{-1}) \frac{\boldsymbol{z}^{[\mu_{i+1} - \mu_i]\alpha^\vee}}{1 - q^{-1}\boldsymbol{z}^{n_Q\a^\vee}}\]
where $[\ ]$ denotes least residue modulo $n_Q$ and
\[ \tau^2_{s(\mu)+\a^\vee,\mu} = g(-B(\a^\vee, \mu) +Q(\a^\vee)) \cdot \boldsymbol{z}^{-\a^\vee}\frac{1-\boldsymbol{z}^{n_Q\a^\vee}}{1-q^{-1}\boldsymbol{z}^{n_Q\a^\vee}}.\] 
Furthermore, under this normalization of the Whittaker functions, the structure constants are independent of the choice of coset representatives $\mu$ and $\nu$.
\end{Prop}

\begin{Remark}
Note that the form of these $\tau_{\nu,\mu}$ is slightly different than that appearing in \cite{McNCS}, owing to the fact that our Whittaker functions are supported on cosets $\g\in\Gamma$ as opposed to on the $n_Q$-th power lattice. Writing $\mu = \nu + \beta$ for $\beta \in \Lambda, \nu \in \Gamma$, where we require $\nu$ to be the minimal representative for its coset, we may compare McNamara's functional $\lambda_\gamma$ to ours via 
\[ \lambda_\gamma(\pi_{\boldsymbol{z}}(\varpi^\mu)v_0) = \boldsymbol{z}^{-\nu} L_\gamma^{\boldsymbol{z}}(\pi_{\boldsymbol{z}}(\varpi^\mu)v_0).\]
\end{Remark}

\begin{proof}
In Section 6 of \cite{BBBFhecke}, Brubaker, Buciumas, Bump, and Friedberg calculate the structure constants for this normalization, which differ by an additional factor of $\boldsymbol{z}^{(-B(\a^\vee,\mu)/Q(\a^\vee))\a^\vee}$ on $\tau^1$ and $\boldsymbol{z}^{-\a^\vee}$ on $\tau^2$ from the forms in Theorem 13.1 of \cite{McNCS}. Applying Proposition \ref{alltheforms}, we may simplify the power of $\boldsymbol{z}$ appearing in the numerator of $\tau^1$: for $\mu = (\mu_1,...,\mu_r)$, we have $B(\a,\mu) = b(\mu_i - \mu_{i+1})$ and $Q(\a) = b$, so 
\[\boldsymbol{z}^{n_Q \left\lceil \frac{B(\alpha^\vee,\mu)}{n_QQ(\alpha^\vee)}\right\rceil\a^\vee - \frac{B(\alpha^\vee,\mu)}{Q(\alpha^\vee)}\alpha^\vee} = \boldsymbol{z}^{n_Q \left\lceil \frac{\mu_i-\mu_{i+1}}{n_Q}\right\rceil\a^\vee - (\mu_i-\mu_{i+1})\alpha^\vee} = \boldsymbol{z}^{[\mu_{i+1} - \mu_i]\alpha^\vee}.\]

Applying a similar argument within the Gauss sum of $\tau^2$, we see that both $\tau^1$ and $\tau^2$ depend solely on the equivalence class of $\mu_i -\mu_{i+1}$ modulo $n_Q$. Any other  representative $\mu'$ for the coset indexed by $\mu$ will differ by an element $\lambda \in \Lambda$. However, the definition of $\Lambda$ requires $B(\a^\vee,\lambda) \in n\Z$, which forces $\lambda_i - \lambda_{i+1} \equiv 0 \pmod{n_Q}$. Thus, our structure constants are independent of coset representative choice.
\end{proof}

We will return to the exploration of these scattering matrix coefficients $\tau_{\nu,\mu}$ in Section \ref{icewhit}, in which we relate them to values of the $R$-matrices explored in the next two sections.

A second method of computing the values of the spherical Whittaker function is by partitioning the preimage of $U^-$ in $\twid{G}$ into cells according to the torus components of the metaplectic Iwasawa decomposition, then computing the defining integral of equation (\ref{whitdef}) directly cell by cell. This is the approach taken in McNamara's \cite{McNcrystal}, which exploits the Chevalley-Steinberg presentation of reductive groups and their metaplectic covers. McNamara examines the case of the $n$-fold cover for $G = SL_{r}(F)$; we will extend his results to our context.

Let $\chi = \chi_{\boldsymbol{z}}$ be a genuine character of $H/(\twid{T}\cap K)$ as defined in (\ref{defchi}). We may extend $\chi$ naturally to $\twid{T}$: every $x\in \twid{T}$ may be written as $x = \zeta \boldsymbol{s}(t)$ for some $t\in T$, so we set $\chi_{\boldsymbol{z}}(x) = \zeta \prod z_i^{\text{ord}(t_i)}$, mirroring (\ref{defchi}). For each $\g \in \Gamma$, let $f_\g: \twid{G} \rightarrow \C$ be the function
\begin{align}
f_\g\left( \z u t k\right) = \begin{cases} \z \d^{1/2}(t) \prod_{i=1}^r z_i^{\lambda_i} &\text{ if }\g = \lambda\\
0 & \text{ else}\end{cases}
\end{align}
for $t = \varpi^\lambda \in \twid{T}/T(\frak{o})$, where $\d$ is the modular quasicharacter of $B$.
Then, we have that $f_\g(g) = L_\g^{\boldsymbol{z}}(\phi_K^{\boldsymbol{z}}(g))$, and therefore,
\[ \mathcal{W}^{\boldsymbol{z}}_\g(g) = \int_{U^-} f_\g(uw_0g)\psi(u)du\]
where again, $w_0$ is a representative in $K$ for $w_0$ in the Weyl group.
Recall, by definition \ref{sphericalWhit}, we need only calculate the Whittaker function on torus elements with dominant weight, so we consider the torus element $\varpi^\lambda$ for a dominant integer partition $\lambda$. Also, note that $w_0\varpi^\lambda w_0^{-1} = \varpi^{(\lambda_r,...,\lambda_1)}$, which we shall denote $\varpi^{w_0\lambda}$. Under the change of coordinates $u \mapsto \varpi^{w_0\lambda} u \varpi^{-w_0\lambda}$, we obtain 
\[ \mathcal{W}^{\boldsymbol{z}}_\g(\varpi^\lambda) =  \int_{U^-} f_\g(\varpi^{w_0\lambda} u) \psi_{w_0\lambda}(u)du,\]
where $\psi_{w_0\lambda}(u) := \psi(\varpi^{w_0\lambda} u \varpi^{-w_0\lambda})$.
Note that $f_\g(\varpi^\lambda g) = (\d^{1/2}\chi)(\varpi^\lambda) f_{\g-\lambda}(g)$ for any $g\in \twid{G}$. Thus,
\begin{align}\label{mainwhitint} \mathcal{W}^{\boldsymbol{z}}_\g(\varpi^\lambda) =  (\d^{1/2}\chi)(\varpi^{w_0\lambda})\int_{U^-} f_{\g -w_0\lambda}(u) \psi_{w_0\lambda}(u)du.\end{align}

Hence, computing $\mathcal{W}^{\boldsymbol{z}}_\g(\varpi^\lambda)$ reduces to computing 
\begin{align}\label{Ilambda} I_{\g,\lambda} = \int_{U^-} f_{\g-w_0\lambda}(u) \psi_{w_0\lambda}(u)du
\end{align}
and this calculation occurs entirely in $\twid{SL}_r(F) \subset \twid{G}$, which allows us to apply many of the calculations done by McNamara in  \cite{McNcrystal}. Note also that the Whittaker function here is defined slightly differently from that in \cite{McNcrystal} in two ways: the argument of $f_\g$ and the fact that we use $f_\g$ instead of $f = \sum_{\g\in \Gamma} f_\g$. However, this second choice produces equivalent information: each term in the Whittaker function $\mathcal{W}(g) = \sum_{\g\in \Gamma}\mathcal{W}^{\boldsymbol{z}}_\g(g)$ analogous to McNamara's is supported on a different coset of  $\twid{T}/H$, so it is straightforward to recover $\mathcal{W}$ from $\mathcal{W}^{\boldsymbol{z}}_\g$ or vice versa. For our later connection to lattice models, however, we prefer to consider each piece $\mathcal{W}_\g^{\boldsymbol{z}}$ separately.

\begin{Remark} 
If $\varpi^{\g-w_0\lambda} \not\in SL_r(F)$, then the function $f_{\g-w_0\lambda}(u) $ is identically zero on $U^-$ and thus $I_{\g,\lambda} = 0$. Hence, for a given $\lambda$, we may calculate the maximum number of nonzero Whittaker values $\mathcal{W}^{\boldsymbol{z}}_\g(\varpi^\lambda)$: it is precisely the number of cosets of $|\twid{T}/H|$ whose representatives differ from $w_0\lambda$ by a cocharacter for $SL_r$, which is exactly $|\twid{T}_{SL}/H_{SL}|$ as calculated in Remark \ref{Dhilloncalc}.
\end{Remark}

McNamara computes the equivalent of $I_\lambda = \sum_{\g\in\Gamma}I_{\g,\lambda}$ explicitly for the unique $n$-fold metaplectic cover of $SL_r(F)$ by expressing it as a sum over the nodes in a crystal parametrized by $\lambda$.  In the remainder of the section, we will extend this computation to any metaplectic cover of $GL_r(F)$.

Let $I$ be a finite indexing set for the set of simple roots. As usual, let $w_0$ denote the longest word in the Weyl group $W$, and set $N:= \frac{r(r-1)}{2}$, which is the length of $w_0$. Let $\ideal{\cdot,\cdot}$ denote the pairing of roots and coroots induced by our fixed bilinear form $B_{b,c}$. Let $l_{\a^\vee}$ be the length of a coroot $\a^\vee$ induced by this pairing. Then $l_{\a^\vee} = Q(\a^\vee) = b\cdot||\a^\vee||^2$, where $||\cdot||$ is the standard Euclidean norm.

In order to compute this integral, we first describe a decomposition of $U^{-}$ according to a particular basis for the highest weight representation parametrized by pairs $(\boldsymbol{i},\boldsymbol{m})$, where $\boldsymbol{i} = (i_1,...,i_N)$ is an $N$-tuple of elements of $I$ such that $w_0 = s_{i_1}s_{i_2}\cdots s_{i_N}$ is a reduced decomposition of the long word $w_0\in W$, and $\boldsymbol{m}\in \NN^N$ with certain additional conditions (see Proposition \ref{mdef}). This choice of long word gives a total ordering $\a_1 <\a_2 \cdots < \a_N$ on the positive roots by $\a_j = s_{i_1}\cdots s_{i_{j-1}}\a_{i_j}$, where $\a_{i_j} \in I$. We prefer to retain the symbol $\a_i$ for the $i$-th simple root, so we will immediately fix $w_0 = s_r(s_{r-1}s_r)(s_{r-2}s_{r-1}s_r) \cdots(s_1s_2\cdots s_{r-1}s_r)$, which gives us the ordering $(i,j) < (k,\ell)$ if $i<k$ or if $i=k$ and $j<\ell$.

We will denote parts of $\boldsymbol{m}$ as $m_{\a}:=m_{i,j}$, where $\a=(i,j)$, and we list these in $\boldsymbol{m}$ in the same order as the positive roots. For a fixed $\boldsymbol{i}$, the tuples $(\boldsymbol{i},\boldsymbol{m})$ index the cells $C_{\boldsymbol{m}}^{\boldsymbol{i}}$ in a specific partition of $U^-$, described in detail in Section 4 of \cite{McNcrystal}. This basis of ordered pairs happens to arise from string data in the Kashiwara crystal $B(\lambda)$  associated to a highest weight $\lambda$. See Bump and Schilling \cite{BScrystal} for a full treatment of these crystals.

 Note that differences in indexing on parts of $\lambda$ occur between our paper and McNamara's because we prefer to index $\lambda$ as a partition, rather than a sum of fundamental weights, in order to facilitate the connections we make to lattice models in Section \ref{metice}.

\begin{Prop}\label{mdef}
For $\boldsymbol{i} = (r,r-1,r,r-2,r-1,r,...,1,2,...,r)$, the tuple $(\boldsymbol{i},\boldsymbol{m}) \in B(\lambda + \rho)$ if and only if for each $(i,j) \in \Phi^+$, we have $0\leq m_{i,j}$ and 
\begin{align}\label{mcond}
\sum_{k=j}^r m_{i,k} \leq \lambda_{r-i} - \lambda_{r-i+1} + 1 + \sum_{k=j+1}^r m_{i+1,k}.
\end{align}
\end{Prop}

Connecting to combinatorial methods of constructing the Whittaker function as seen in \cite{McNcrystal} and \cite{BBF}, we decorate the tuple $(\boldsymbol{m},\a)$ with the following boxing rule: for $\a = \sum c_j \vare_j$, let $m_\a = \sum c_jm_j$. Then, circle $m_\a$ if $m_\a  = 0$ and box $m_\a$ if, for $\a = (i,j)$, we have equality in Equation (\ref{mcond}).

Let $g_Q(a)$ be the normalized $n_Q$-th Gauss sum $g(Q(\a^\vee)a)$, where $g(a)$ is defined in (\ref{gausssum}). We then define the weight function 
\begin{align}\label{weighting} w(\boldsymbol{m},\a) = \begin{cases} g_Q(r_\a,s_\a) & \text{ if $m_\a$ is not circled}\\ q^{r_\a} & \text{ if $m_\a$ is circled but not boxed}\\ 0 &\text{ if $m_\a$ is both boxed and circled}\end{cases}
\end{align}
where $r_{i,j} = \sum_{k\leq i}m_{k,j}$ and $s_{i,j} = \lambda_{r-i} - \lambda_{r-i+1} + \sum_{k=j+1}^r m_{i+1,k} - \sum_{k=j}^r m_{i,k}$.

\begin{Thm}\label{crystalsum}
For any metaplectic cover of $GL_r(F)$ and dominant weight $\lambda$, the integral $I_\lambda$ used to calculate the Whittaker function is given by
\[ I_\lambda = \sum_{\g\in \Gamma} I_{\g,\lambda} = \sum_{(\boldsymbol{i},\boldsymbol{m}) \in B(\lambda + \rho)} \prod_{\a\in \Phi^+} w(\boldsymbol{m},\a)x_\a^{m_\a}.\]
\end{Thm}

\begin{proof}
Recall that $U^-$ is the negative unipotent radical of $G$, which we identify with the corresponding subgroup in $\twid{G}$, since our metaplectic extension admits a $T$-equivariant splitting over $U$ and thus splits over $U^-$ \cite{MW}. Since the calculation of $I_\lambda$ occurs entirely in $U^-$, we can view our calculations for $G = GL_r(F)$ instead as calculations in the embedded subgroup $SL_r(F)$. As seen in Section \ref{metchar}, the bilinear form controls multiplication inside the metaplectic covering group. Notice that for $SL_r(F)$, there is a unique metaplectic extension given by the dot product; this is the case for which McNamara calculates an equivalent $I_\lambda$ in \cite{McNcrystal}. However, for $GL_r(F)$, our metaplectic extensions are given by bilinear forms $B_{b,c}$, so multiplication in the embedded copy of $SL_r(F)$ in $\twid{G}$ is modified by a power of $b$. 

We therefore treat the $n$ used by McNamara as a formal parameter, noting that in our context, this parameter is in fact our $n_Q$. That is, it depends both on our original choice of $n$ used in Section \ref{metchar} and on the bilinear form related to our metaplectic extension. To that end, we make the following modifications to McNamara's calculations: at every instance of $l_{\a^\vee}$, we use the length of $\a^\vee$ as defined above, which equals $Q(\a^\vee)$ as opposed to the standard Euclidean norm. We also replace any $n$-th order Gauss sums by $n_Q$-th order Gauss sums. We allow $\lambda$ to vary over cocharacters of $GL_r$, noting that McNamara's calculations for $SL_r$ treat $\lambda$ as a generic partition. Under these modifications, Proposition 6.1, Lemma 6.3, Theorem 6.4 of \cite{McNcrystal} remain true for any metaplectic cover of $GL_r(F)$.

Marshaling these results together, we manipulate $I_\lambda$ following McNamara: using Algorithm 4.4 to rewrite $u = \prod_{\a \in \Phi^+} e_{-\a}(x_\a)$ in terms of $w_\a = \varpi^{-m_i}u_\a$, we have that for the case when all $m_\a > 0$: $f = \sum_{\g\in \Gamma} f_\g$ takes the form
\[ f(u) = \prod_{\a\in \Phi^+} (q^{-\langle \rho,\a\rangle}x_\a)^{m_\a} (\varpi,u_\a)^{\sum_{\b\leq\a\in \Phi^+}m_\a\cdot B(\b,\a^\vee)}\]
and $\psi_{w_0\lambda}(u) = \prod_{i=1}^r \psi(\varpi^{\lambda_{r-i} - \lambda_{r-i+1}}x_{i,i+1})$ becomes
\[ \psi_{w_0\lambda}(u) = \prod_{\a \in \Phi^+} \psi \left(\varpi^{s_\a}\prod_{\b\geq\a\in \Phi^+}u_\b^{B(\a,\b^\vee)/b} \right).\]
Under the change of variables $t_\a  = \prod_{\b>\a\in \Phi^+}u_\b^{B(\a,\b^\vee)}$, the power of $q$ in $f(u)$ cancels out with that coming from the Jacobian and we then have that
\begin{align*}
\int_{C_{\boldsymbol{m}}^{\boldsymbol{i}}} f(u) \psi_{w_0\lambda}(u)du &= \prod_{\a\in \Phi^+}x_\a^{m_\a} \int_{\mathfrak{o}_F^\times} (\varpi,t_\a)^{r_\a\cdot b}\psi(\varpi^{s_\a}t_\a)dt_\a\\
&= \prod_{\a\in\Phi^+} w(\boldsymbol{m},\a)x_\a^{m_a}.
\end{align*}
For the cases in which some $m_\a = 0$, the alterations necessary are discussed in the proof of Theorem 8.4 of \cite{McNcrystal} and pass through to our general case unhindered. Summing over all cells $C_{\boldsymbol{m}}^{\boldsymbol{i}}$ then completes the proof.
\end{proof}

\section{Metaplectic Ice and Yang-Baxter Equations}\label{metice}

In \cite{BBBIce}, Brubaker, Buciumas, and Bump explore connections between \emph{solvable }``metaplectic ice" lattice models, i.e., those with weights satisfying a Yang-Baxter equation, and the Whittaker functions for the metaplectic $n$-fold cover corresponding to the dot product ($B_{\twid{G}} = B_{1,1}$). In particular, they give explicit formulas connecting the partition functions of these models to values of the Whittaker functions. 

To construct a \emph{metaplectic ice} lattice model, we build a rectangular grid with finitely many rows and columns, with vertices located at each intersection of a row and column. Note that every vertex has valence 4. Every edge is assigned a \emph{spin} (either $+$ or $-$). We fix a set of spin conditions on the boundary edges, which defines a \emph{system}, and allow the spins on interior edges to vary; we will call a choice of these interior spins a \emph{state} of our system.

In particular, \cite{BBBIce} considers a way to define boundary conditions on these models such that systems with $r$ rows are parametrized by integer partitions of length $r$. Let $\lambda = (\lambda_1,...,\lambda_r)$ be an integer partition. Build a grid with $r$ rows and $N$ columns, where $N$ is any integer greater than or equal to $\lambda_1 + r$ (not to be confused with the $N$ of section \ref{whitcalcs}). Set the boundary edge spins to be $+$ on the left and bottom edges and $-$ on the right edges. For the top edges, number them from right to left in increasing order, starting with 0. Then, add $\rho := (r-1,...,3,2,1,0)$ to $\lambda$ and label the top edges numbered $(\lambda + \rho)_i$ with $-$ spin for all $1\leq i\leq r$. Label the remaining top edges with $+$ spin. See Figure \ref{icestate} for an example of these conditions.

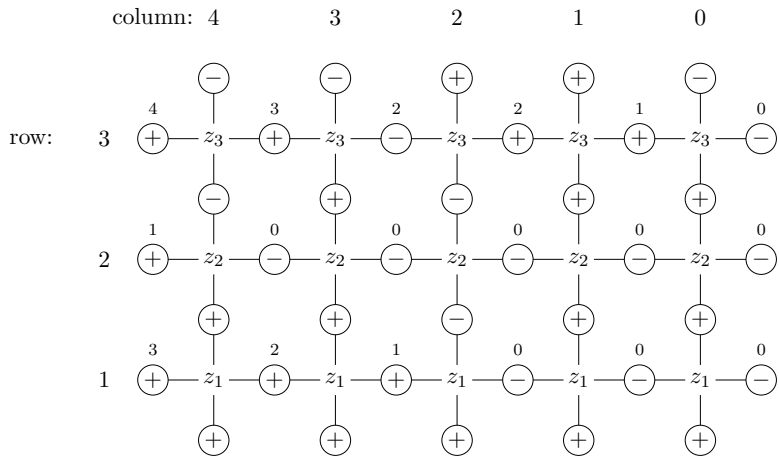
\begin{figure}[h]
\centering
\scalebox{0.8}{
\begin{tikzpicture}
  \coordinate (ad) at (3,0);
  \coordinate (af) at (5,0);
  \coordinate (ah) at (7,0);
  \coordinate (aj) at (9,0);
  \coordinate (al) at (11,0);
  \coordinate (bc) at (2,1);
  \coordinate (be) at (4,1);
  \coordinate (bg) at (6,1);
  \coordinate (bi) at (8,1);
  \coordinate (bk) at (10,1);
  \coordinate (bm) at (12,1);
  \coordinate (cd) at (3,2);
  \coordinate (cf) at (5,2);
  \coordinate (ch) at (7,2);
  \coordinate (cj) at (9,2);
  \coordinate (cl) at (11,2);
  \coordinate (dc) at (2,3);
  \coordinate (de) at (4,3);
  \coordinate (dg) at (6,3);
  \coordinate (di) at (8,3);
  \coordinate (dk) at (10,3);
  \coordinate (dm) at (12,3);
  \coordinate (ed) at (3,4);
  \coordinate (ef) at (5,4);
  \coordinate (eh) at (7,4);
  \coordinate (ej) at (9,4);
  \coordinate (el) at (11,4);
  \coordinate (fc) at (2,5);
  \coordinate (fe) at (4,5);
  \coordinate (fg) at (6,5);
  \coordinate (fi) at (8,5);
  \coordinate (fk) at (10,5);
  \coordinate (fm) at (12,5);
  \coordinate (gd) at (3,6);
  \coordinate (gf) at (5,6);
  \coordinate (gh) at (7,6);
  \coordinate (gj) at (9,6);
  \coordinate (gl) at (11,6);
  \coordinate (bd) at (3,1);
  \coordinate (bf) at (5,1);
  \coordinate (bh) at (7,1);
  \coordinate (bj) at (9,1);
  \coordinate (bl) at (11,1);
  \coordinate (dd) at (3,3);
  \coordinate (df) at (5,3);
  \coordinate (dh) at (7,3);
  \coordinate (dj) at (9,3);
  \coordinate (dl) at (11,3);
  \coordinate (fd) at (3,5);
  \coordinate (ff) at (5,5);
  \coordinate (fh) at (7,5);
  \coordinate (fj) at (9,5);
  \coordinate (fl) at (11,5);
  \coordinate (bcx) at (2,1.5);
  \coordinate (bex) at (4,1.5);
  \coordinate (bgx) at (6,1.5);
  \coordinate (bix) at (8,1.5);
  \coordinate (bkx) at (10,1.5);
  \coordinate (bmx) at (12,1.5);
  \coordinate (dcx) at (2,3.5);
  \coordinate (dex) at (4,3.5);
  \coordinate (dgx) at (6,3.5);
  \coordinate (dix) at (8,3.5);
  \coordinate (dkx) at (10,3.5);
  \coordinate (dmx) at (12,3.5);
  \coordinate (fcx) at (2,5.5);
  \coordinate (fex) at (4,5.5);
  \coordinate (fgx) at (6,5.5);
  \coordinate (fix) at (8,5.5);
  \coordinate (fkx) at (10,5.5);
  \coordinate (fmx) at (12,5.5);
  \draw (ad)--(gd);
  \draw (af)--(gf);
  \draw (ah)--(gh);
  \draw (aj)--(gj);
  \draw (al)--(gl);
  \draw (bc)--(bm);
  \draw (dc)--(dm);
  \draw (fc)--(fm);
  \draw[fill=white] (ad) circle (.25);
  \draw[fill=white] (af) circle (.25);
  \draw[fill=white] (ah) circle (.25);
  \draw[fill=white] (aj) circle (.25);
  \draw[fill=white] (al) circle (.25);
  \draw[fill=white] (bc) circle (.25);
  \draw[fill=white] (be) circle (.25);
  \draw[fill=white] (bg) circle (.25);
  \draw[fill=white] (bi) circle (.25);
  \draw[fill=white] (bk) circle (.25);
  \draw[fill=white] (bm) circle (.25);
  \draw[fill=white] (cd) circle (.25);
  \draw[fill=white] (cf) circle (.25);
  \draw[fill=white] (ch) circle (.25);
  \draw[fill=white] (cj) circle (.25);
  \draw[fill=white] (cl) circle (.25);
  \draw[fill=white] (dc) circle (.25);
  \draw[fill=white] (de) circle (.25);
  \draw[fill=white] (dg) circle (.25);
  \draw[fill=white] (di) circle (.25);
  \draw[fill=white] (dk) circle (.25);
  \draw[fill=white] (dm) circle (.25);
  \draw[fill=white] (ed) circle (.25);
  \draw[fill=white] (ef) circle (.25);
  \draw[fill=white] (eh) circle (.25);
  \draw[fill=white] (ej) circle (.25);
  \draw[fill=white] (el) circle (.25);
  \draw[fill=white] (fc) circle (.25);
  \draw[fill=white] (fe) circle (.25);
  \draw[fill=white] (fg) circle (.25);
  \draw[fill=white] (fi) circle (.25);
  \draw[fill=white] (fk) circle (.25);
  \draw[fill=white] (fm) circle (.25);
  \draw[fill=white] (gd) circle (.25);
  \draw[fill=white] (gf) circle (.25);
  \draw[fill=white] (gh) circle (.25);
  \draw[fill=white] (gj) circle (.25);
  \draw[fill=white] (gl) circle (.25);
  \path[fill=white] (bd) circle (.25);
  \path[fill=white] (bf) circle (.25);
  \path[fill=white] (bh) circle (.25);
  \path[fill=white] (bj) circle (.25);
  \path[fill=white] (bl) circle (.25);
  \path[fill=white] (dd) circle (.25);
  \path[fill=white] (df) circle (.25);
  \path[fill=white] (dh) circle (.25);
  \path[fill=white] (dj) circle (.25);
  \path[fill=white] (dl) circle (.25);
  \path[fill=white] (fd) circle (.25);
  \path[fill=white] (ff) circle (.25);
  \path[fill=white] (fh) circle (.25);
  \path[fill=white] (fj) circle (.25);
  \path[fill=white] (fl) circle (.25);
  \node at (bd) {$z_1$};
  \node at (bf) {$z_1$};
  \node at (bh) {$z_1$};
  \node at (bj) {$z_1$};
  \node at (bl) {$z_1$};
  \node at (dd) {$z_2$};
  \node at (df) {$z_2$};
  \node at (dh) {$z_2$};
  \node at (dj) {$z_2$};
  \node at (dl) {$z_2$};
  \node at (fd) {$z_3$};
  \node at (ff) {$z_3$};
  \node at (fh) {$z_3$};
  \node at (fj) {$z_3$};
  \node at (fl) {$z_3$};
  \node at (0,5) {row:};
  \node at (1.2,5) {3};
  \node at (1.2,3) {2};
  \node at (1.2,1) {1};
  \node at (gd) {$-$};
  \node at (gf) {$-$};
  \node at (gh) {$+$};
  \node at (gj) {$+$};
  \node at (gl) {$-$};
  \node at (fc) {$+$};
  \node at (fe) {$+$};
  \node at (fg) {$-$};
  \node at (fi) {$+$};
  \node at (fk) {$+$};
  \node at (fm) {$-$};
  \node at (ed) {$-$};
  \node at (ef) {$+$};
  \node at (eh) {$-$};
  \node at (ej) {$+$};
  \node at (el) {$+$};
  \node at (dc) {$+$};
  \node at (de) {$-$};
  \node at (dg) {$-$};
  \node at (di) {$-$};
  \node at (dk) {$-$};
  \node at (dm) {$-$};
  \node at (cd) {$+$};
  \node at (cf) {$+$};
  \node at (ch) {$-$};
  \node at (cj) {$+$};
  \node at (cl) {$+$};
  \node at (bc) {$+$};
  \node at (be) {$+$};
  \node at (bg) {$+$};
  \node at (bi) {$-$};
  \node at (bk) {$-$};
  \node at (bm) {$-$};
  \node at (ad) {$+$};
  \node at (af) {$+$};
  \node at (ah) {$+$};
  \node at (aj) {$+$};
  \node at (al) {$+$};
  \node at (11,7) {$0$};
  \node at (9,7) {$1$};
  \node at (7,7) {$2$};
  \node at (5,7) {$3$};
  \node at (3,7) {$4$};
  \node at (2,7.04) {column:};
  \node at (bcx) {$\scriptstyle 3$};
  \node at (bex) {$\scriptstyle 2$};
  \node at (bgx) {$\scriptstyle 1$};
  \node at (bix) {$\scriptstyle 0$};
  \node at (bkx) {$\scriptstyle 0$};
  \node at (bmx) {$\scriptstyle 0$};
  \node at (dcx) {$\scriptstyle 1$};
  \node at (dex) {$\scriptstyle 0$};
  \node at (dgx) {$\scriptstyle 0$};
  \node at (dix) {$\scriptstyle 0$};
  \node at (dkx) {$\scriptstyle 0$};
  \node at (dmx) {$\scriptstyle 0$};
  \node at (fcx) {$\scriptstyle 4$};
  \node at (fex) {$\scriptstyle 3$};
  \node at (fgx) {$\scriptstyle 2$};
  \node at (fix) {$\scriptstyle 2$};
  \node at (fkx) {$\scriptstyle 1$};
  \node at (fmx) {$\scriptstyle 0$};
\end{tikzpicture}}
\caption{An admissible state of metaplectic ice for the parameters $r = 3, N=5$, and  $\lambda=(2,2,0)$.
This gives $\lambda+\rho=(4,3,0)$, so top edges $4,3,0$ get $-$ spin. Each horizontal edge is labelled with its charge; note that this state is only $n_Q$-admissible for $n_Q = 1,2$.}
\label{icestate}
\end{figure}

\begin{Def}
We say that a state of metaplectic ice is \emph{admissible} if at each vertex, the choice of spins on the four adjacent edges matches one of the six configurations shown in Figure \ref{modwts}. We will only be considering admissible states for the remainder of this paper.
\end{Def}

Let $n_Q$ be a positive integer. We will later specify it to depend on our fixed $n$ and a quadratic form $Q$ as in Section \ref{metwhit}, but the results in this section are independent of this specification. We will name other variables in this section in a similarly suggestive manner; they may all be viewed as general parameters until we specify them in Section \ref{pfmain}.

\begin{Def}
Define the \emph{charge} at each horizontal edge to be the number of $+$ spins in its row that are on or to the right of this edge. (For example, all horizontal edges in  Figure \ref{icestate} are labelled with their charge.) We then define a state to be \emph{$n_Q$-admissible} if it is admissible and in addition, every horizontal edge with a $-$ spin has charge $\equiv 0 \pmod{n_Q}$. Restricting to these cases, we can then consider all our charges modulo $n_Q$.
\end{Def}

The \emph{Boltzmann weight} of a state is obtained by attaching a weight to each vertex and taking the product over all vertices in the state. Accordingly, label the vertices in the $i$-th row of the grid as $z_i$, numbering from bottom to top. In Figure \ref{modwts} we choose a set of weights depending on these $z_i$ for each type of admissible vertex. Note that these are a generalization of the modified Boltzmann weights from Section 5 of \cite{BBBIce}. In these figures, the weight assigned to a vertex in the $i$-th row of a state depends on $i$, the charge of the adjacent horizontal edges (mod $n_Q$), and the type of vertex (i.e. the spins of all four adjacent edges). To define these weights, we also need to define a pair of functions $g_Q$ and $\d$. For a fixed parameter $v$, let $g_Q(t)$ be a function on integers $t$ such that: $g_Q$ is periodic modulo $n_Q$, $g_Q(0) = -v$, and $g_Q(t)g_Q(n_Q-t) = v$ if $t\not\equiv 0 \pmod{n_Q}$. (Note that the $n_Q$-th Gauss sum discussed in Section \ref{metwhit} is such a function.) Let $\d(t)$ be the function on integers $t$ defined by
\[ \d(t) = \begin{cases} 1 &\text{ if }t\equiv 0\pmod{n_Q}\\ 0 &\text{ else }\end{cases}.\] 
We will denote the Boltzmann weight of a state $\mathfrak{s}$ as $B^{(n_Q)}(\mathfrak{s})$, in order to make clear the dependence on our choice of $n_Q$ and also maintain consistency with previous literature.

\begin{figure}[h]
\noindent\makebox[1\textwidth][c]{
\begin{tabular}{|@{\hspace{-4pt}}c@{\hspace{-4pt}}|@{\hspace{-4pt}}c@{\hspace{-4pt}}|@{\hspace{-3pt}}c@{\hspace{-3pt}}|@{\hspace{-4pt}}c@{\hspace{-4pt}}|@{\hspace{-4pt}}c@{\hspace{-4pt}}|@{\hspace{-4pt}}c@{\hspace{-4pt}}|}
\hline
$\tt{a}_1$&$\tt{a}_2$&$\tt{b}_1$&$\tt{b}_2$&$\tt{c}_1$&$\tt{c}_2$\\
\hline
$\begin{array}{c}\gammaice{+}{+}{+}{+}{a{+}1}{\mathmakebox[\widthof{$a{+}1$}]{a}}\\z_i^{-n_Q\delta(a)}\end{array}$ &
$\begin{array}{c}\gammaice{-}{-}{-}{-}{\mathmakebox[\widthof{$a{+}1$}]{0}}{\mathmakebox[\widthof{$a{+}1$}]{0}}\\ 1\vphantom{z_i^{-n\delta(a)}}\end{array}$ &
$\begin{array}{c}\gammaice{+}{-}{+}{-}{a{+}1}{\mathmakebox[\widthof{$a{+}1$}]{a}}\\g_Q(a)z_i^{-n_Q\delta(a)}\end{array}$ &
$\begin{array}{c}\gammaice{-}{+}{-}{+}{\mathmakebox[\widthof{$a{+}1$}]{0}}{\mathmakebox[\widthof{$a{+}1$}]{0}}\\ 1\vphantom{z_i^{-n_Q\delta(a+1)}}\end{array}$ &
$\begin{array}{c}\gammaice{-}{+}{+}{-}{\mathmakebox[\widthof{$a{+}1$}]{0}}{\mathmakebox[\widthof{$a{+}1$}]{0}}\\ (1-v) z_i^{-n_Q}\end{array}$ &
$\begin{array}{c}\gammaice{+}{-}{-}{+}{\mathmakebox[\widthof{$a{+}1$}]{1}}{\mathmakebox[\widthof{$a{+}1$}]{0}}\\1\end{array}$ \\
\hline
\end{tabular}}
\caption{The Boltzmann weights at a vertex for our ice model. The illustrated vertices are in the $i$-th row and have charge $\equiv a \pmod{n_Q}$. The Boltzmann weight of any other configuration is zero.}
\label{modwts}
\end{figure}

\begin{Example}
Let $\mathfrak{s}$ be the state shown in Figure \ref{icestate}, which is only $n_Q$-admissible for $n_Q = 1,2$. Suppose $n_Q = 2$. The Boltzmann weight of the top left vertex (row 3, column 4)  is type $\tt{b}_1$, so it has charge $g_Q(1)z_3^{-2\d(1)}$, which equals $g_Q(1)$. Reading left to right across the rest of row 3, we then have one type $\tt{c}_2$ vertex (weight $1$), one type $\tt{c}_1$ vertex (weight $(1-v)z_3^{-2}$), another type $\tt{a}_1$ vertex (weight $z_3^{-2\d(1)}$), and another type $\tt{c}_2$ vertex (weight $1$). So the total weight for this row is $-v(1-v)z_3^{-2}$. Continuing this process across the remaining rows, we calculate that the weight of the second row is $1$ and the weight of row one is $z_1^{-2}$. So, the total Boltzmann weight of our state is $B^{(n_Q)}(\mathfrak{s}) = -g_Q(1)(1-v)z_1^{-2}z_3^{-2}.$ \end{Example}

\begin{Def}
We may extend our notion of system conditions to include not only boundary spin requirements but also boundary charge requirements. In this case, we specify the vector $c=(c_1,...,c_r)$, where $c_i$ is the left boundary charge of the $i$-th row. Then, given a system $(\frak{S};c)$ of boundary conditions, the \emph{partition function} $Z(\frak{S};c)$ is
\[ Z(\frak{S};c) = \sum_{\frak{s}\in (\frak{S};c)} B^{(n_Q)}(\frak{s}).\]
\end{Def}

This ice model is closely related to the solvable six-vertex models of statistical mechanics, and we can use techniques from the study of six-vertex models, such as Yang-Baxter equations, in this context as well. One quick shift in perspective is necessary: our Boltzmann weights depend on the charge of horizontal edge, which is a global statistic. In order to use statistical mechanics techniques, we need vertex weights to be completely local, relying only on the incident edges. Thus, we will proceed with \emph{decorated spins} as in \cite{BBBIce}, where each edge is labelled with an ordered pair $(\sigma, a)$ of a spin $\sigma$ ($+$ or $-$) and some integer $a \pmod{n_Q}$. The point here is that we can consider the decoration $a$ as part of the data on a horizontal edge, rather than as a global statistic superimposed on our state. Since we will stay in the context of $n_Q$-admissible states, we will require that  if $\sigma = -$, we have $a\equiv 0\pmod{n_Q}$, and that adjacent horizontal edges have decorated spins $(\sigma,a)$ and $(\tau,b)$ appearing as one of the six admissible vertex states in Figure \ref{modwts}. We see then that if we set the decorations on the rightmost edges in a state to be identically $0$ and impose these restrictions, we recover the Boltzmann charge on a vertex as originally stated. Thus, we can consider Boltzmann weights as local, which allows us to calculate the weight of a subgraph of an $n_Q$-admissible state.

In this context, we view the transfer matrices associated with Yang-Baxter equations as the weighted $R$-vertices shown in Figure \ref{modR}. Attaching one of these diagonal vertices to an ice state allows us to switch the $i$-th and $j$-th strands at that point and evaluate what effect changing or not changing the boundary conditions on those strands has on the partition function of the total state.

\begin{figure}[h]
\noindent\makebox[1\textwidth][c]{
\begin{tabular}{|@{\hspace{-2pt}}c@{\hspace{-2pt}}|@{\hspace{-2pt}}c@{\hspace{-2pt}}|@{\hspace{-2pt}}c@{\hspace{-2pt}}|@{\hspace{-2pt}}c@{\hspace{-2pt}}|}\hline
\multicolumn{3}{|c|@{\hspace{-2pt}}}{$\tt{a}_1$}&$\tt{a}_2$\\\hline
$\begin{array}{c}\gamgam{+}{+}{+}{+}{a\vphantom{b}}{a\vphantom{b}}{a\vphantom{b}}{a\vphantom{b}}\\ \dfrac{-v + \boldsymbol{z}^{n_Q \alpha^\vee}}{1-v \boldsymbol{z}^{n_Q \alpha^\vee}} \vphantom{\dfrac{(1-v)}{1-v \boldsymbol{z}^{n_Q \alpha^\vee}} \cdot \left\{\hspace{-5pt}\begin{array}{ll}
\boldsymbol{z}^{n_Q\alpha^\vee}\kern-5pt & a>b, \\
1 & a<b \end{array}\hspace{-7pt}\right.}\end{array}$ & 
$\begin{array}{c}\gamgam{+}{+}{+}{+}{b}{a\vphantom{b}}{b}{a\vphantom{b}}\\ g_Q(a-b) \dfrac{1- \boldsymbol{z}^{n_Q \alpha^\vee}}{1-v \boldsymbol{z}^{n_Q \alpha^\vee}} \vphantom{\dfrac{(1-v)}{1-v \boldsymbol{z}^{n_Q \alpha^\vee}} \cdot \left\{\hspace{-5pt}\begin{array}{ll}
\boldsymbol{z}^{n_Q \alpha^\vee}\kern-5pt & a>b, \\
1 & a<b \end{array}\hspace{-7pt}\right.}\end{array}$ & 
$\begin{array}{c}\gamgam{+}{+}{+}{+}{b}{a\vphantom{b}}{a\vphantom{b}}{b}\\
\dfrac{(1-v)}{1-v \boldsymbol{z}^{n_Q \alpha^\vee}} \cdot \left\{\hspace{-5pt}\begin{array}{ll}
\boldsymbol{z}^{n_Q \alpha^\vee}\kern-5pt & a>b, \\
1 & a<b \end{array}\hspace{-7pt}\right.
\end{array}$ & 
$\begin{array}{c}\gamgam{-}{-}{-}{-}{0}{0}{0}{0}\\1\vphantom{\dfrac{(1-v)}{1-v \boldsymbol{z}^{n_Q \alpha^\vee}} \cdot \left\{\hspace{-5pt}\begin{array}{ll}
\boldsymbol{z}^{n_Q \alpha^\vee}\kern-5pt & a>b, \\
1 & a<b \end{array}\hspace{-7pt}\right.}\end{array}$\\ 
\hline\multicolumn{4}{|c|}{}\\[-12pt]\hline
$\tt{b}_1$&$\tt{b}_2$&$\tt{c}_1$&$\tt{c}_2$\\\hline
$\begin{array}{c}\gamgam{+}{-}{+}{-}{a\vphantom{0}}{0}{a\vphantom{0}}{0}\\ \dfrac{v(1 - \boldsymbol{z}^{n_Q \alpha^\vee})}{1-v \boldsymbol{z}^{n_Q \alpha^\vee}} \end{array} $& 
$\quad\begin{array}{c}\gamgam{-}{+}{-}{+}{0}{a\vphantom{0}}{0}{a\vphantom{0}}\\ \dfrac{1 - \boldsymbol{z}^{n_Q \alpha^\vee}}{1-v \boldsymbol{z}^{n_Q \alpha^\vee}} \end{array}$ & 
$\begin{array}{c}\gamgam{-}{+}{+}{-}{0}{a\vphantom{0}}{a\vphantom{0}}{0}\\ \dfrac{(1-v) \boldsymbol{z}^{n_Q \alpha^\vee}}{1-v \boldsymbol{z}^{n_Q \alpha^\vee}} \end{array}$ & 
$\begin{array}{c}\gamgam{+}{-}{-}{+}{a\vphantom{0}}{0}{0}{a\vphantom{0}}\\ \dfrac{(1-v)}{1-v \boldsymbol{z}^{n_Q \alpha^\vee}} \end{array}$ \\ 
\hline
\end{tabular}
}
\caption{\parindent=0pt\parskip=0pt Boltzmann weights for the $R$-vertex $R_{z_i,z_j}$. Here, let $\boldsymbol{z} = (z_1,...,z_n)$ and $\a^\vee$ the coroot with $\a^\vee_i = 1, \a^\vee_j = -1,$ and all other coordinates equal to zero. We assume that $b$ is not equal to $a$.}
\label{modR}
\end{figure}

To suit this purpose, we limit the types (i.e. choices of decorated spins on edges) of $R$-vertices allowed to those in Figure \ref{modR}. To calculate the Boltzmann weight of an ice model involving an $R$-vertex, treat the $R$-vertex like any other vertex and multiply its weight into the Boltzmann weight of the rest of the state. When we take the partition function of a system involving an $R$-vertex, we also consider all possible choices of $R$-vertices that suit the conditions of the system when creating our states.

\begin{Thm}\label{ybe1}
Fix the decorated spins $(\sigma,a), (\tau,b),(\theta,c),(\rho,d)$ and the spins $\alpha,\beta$ as boundary conditions as in the states below. Then, using the set of weights in Figures \ref{modwts} and \ref{modR}, the partition functions of the following two states are equal.

\vspace{-1cm}

\begin{align}\label{mybe}
\begin{array}{c} \lhs{\botcharge{\sigma}{a}}{\topcharge{\tau}{b}}{\beta}{\topcharge{\theta}{c}}{\botcharge{\rho}{d}}{\alpha}{\topcharge{\nu}{e}}{\botcharge{\mu}{f}}{\gamma} \end{array}
\hspace{1cm}
\begin{array}{c} \rhs{\botcharge{\sigma}{a}}{\topcharge{\tau}{b}}{\beta}{\topcharge{\theta}{c}}{\botcharge{\rho}{d}}{\alpha}{\botcharge{\psi}{h}}{\topcharge{\phi}{g}}{\delta} \end{array}
\end{align}
\end{Thm}

\begin{proof} In \cite{BBBIce}, Theorem 3.1, Brubaker, Buciumas, and Bump prove the analogous theorem for a slightly different set of Boltzmann weights. The weights in this paper are obtained from theirs by a change of basis and a slight modification. We also set their $n$ to equal our $n_Q$ and their $g(a)$ to equal our $g_Q(a)$. 

Given $z\in \C,z\neq 0$, define $V(z)$ to be the $(n_Q+1)$-dimensional vector space spanned by $v_\a = v_\a(z)$, where $\a$ runs through the decorated spins $\{-0,+1,....,+n_Q\}$.
We may view the vertices in Figure \ref{modwts} as acting on the space $V(z_i)$ and the $R$-vertices in Figure \ref{modR} as acting on the space $V(z_i)\otimes V(z_j)$. Thus, changing the basis we use for these spaces changes the weight attached to a vertex: let $f(\a,z)$ be a function on $\a$ a decorated spin and $z\in \C$. Then, take the Boltzman weights from Figures \ref{modwts} and \ref{modR}, with $z_i$ and $z_j$ as in those figures, of the following vertices:

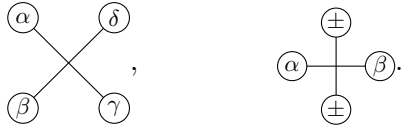
\[\vcenter{\hbox to .7in{\gamgam{\beta}{\alpha}{\delta}{\gamma}{}{}{}{}}},\qquad\qquad
\vcenter{\hbox to .7in{\gammaice{\alpha}{\pm}{\beta}{\pm}{}{}{}{}}}.\]

We multiply these weights, respectively, by 

\begin{align}
\label{frecipe}
\frac{f (\alpha, z_i) f (\beta, z_j)}{f (\gamma, z_i) f (\delta, z_j)},\qquad\qquad
\frac{f(\alpha,z_i)}{f(\beta,z_i)},
\end{align}
and the result gives us a new set of Boltzmann weights, that still satisfy all the same equations as the original sets. Thus, since the original weights in \cite{BBBIce} satisfied the Yang Baxter equation in (\ref{mybe}), the new weights after change of basis will as well. 

We choose the function
\[ f(\a,z) = \begin{cases} z^a &\text{ if }\a = +a,0< a\leq n,\\ 1&\text{ if }\a =-0 \end{cases}\]
and apply this change of basis to the unmodified weights in \cite{BBBIce}. We divide the resulting vertex weights by $z_i$  and the resulting $R$-vertex weights by $1-v(z_i/z_j)^{n_Q}$ to obtain the weights in Figures \ref{modwts} and \ref{modR}, respectively. This division does not affect the Yang Baxter equation, since it happens on both sides, so  (\ref{mybe}) holds with our weights. Note that the original proof in \cite{BBBIce} checks all 32 possible cases of boundary spin choice, which can be found in an older version of the paper. We present a particularly interesting case of this calculation with our weights below, and include the rest of the modified calculations for our weights in the Appendix (Section \ref{appendix}) for the sake of completeness. 
\end{proof}

\begin{Example}\label{nicecase}
Consider the boundary spin conditions in Case 10 (as numbered in Section \ref{appendix}): with the boundary spins $(\sigma,\tau,\b,\theta,\rho,\a) = (+,+,-,+,-,+)$. We have three subcases, based on the choice of $(a,b,c,d)$. 

\textbf{Case 10a:} For $k\nequiv n_Q$, suppose that the decorated spins on the boundary are as follows:

\begin{center}
\begin{tabular}{|c|c|c|c|c|c|}\hline$a$&$b$&&$c$&$d$&\\$\sigma$&$\tau$&$\b$&$\theta$&$\rho$&$\a$\\\hline$1$&$k+1$&&$k$&$0$&\\+&+&$-$&+&$-$&+ \\\hline\end{tabular}
\end{center}

We can complete these boundary conditions to a state in only one way on each side of the figures in Equation (\ref{mybe}):

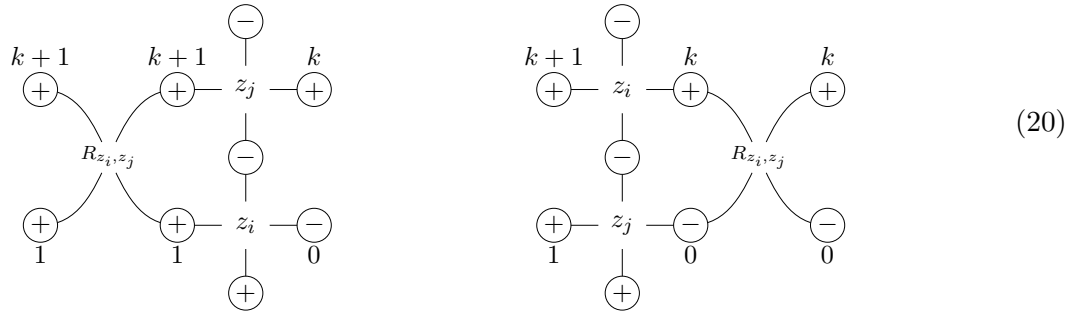
\begin{align}\label{mybeexample}
\begin{array}{c} \lhs{\botcharge{+}{1}}{\topcharge{+}{k+1}}{-}{\topcharge{+}{k}}{\botcharge{-}{0}}{+}{\topcharge{+}{k+1}}{\botcharge{+}{1}}{-} \end{array}
\hspace{1cm}
\begin{array}{c} \rhs{\botcharge{+}{1}}{\topcharge{+}{k+1}}{-}{\topcharge{+}{k}}{\botcharge{-}{0}}{+}{\botcharge{-}{0}}{\topcharge{+}{k}}{-} \end{array}
\end{align}

Writing this information in a more compact format, we have:

\begin{center}
\text{ left hand side} \hspace{4cm} \text{ right hand side}\\
\begin{tabular}{|c|c|c|c|}
\hline $e$&$f$&&\multirow{2}{*}{weight}\\$\nu$&$\mu$&$\g$&\\\hline\small{$k+1$}&\small{$1$}&&\multirow{2}{*}{\small{$\dfrac{g_Q(k)(1-v)\zn}{1-v\zn}$}}\\+&+&$-$&\\\hline
\end{tabular}
\hspace{0.5cm}
\begin{tabular}{|c|c|c|c|}
\hline $g$&$h$&&\multirow{2}{*}{weight}\\$\phi$&$\psi$&$\d$&\\
\hline\small{$k$}&\small{0}&&\multirow{2}{*}{\small{$\dfrac{g_Q(k)(1-v)\zn}{1-v\zn}$}}\\+&$-$&$-$&\\\hline\end{tabular}
\end{center}

Since the weight on each side is the same, the Yang Baxter equation is satisfied.

\textbf{Case 10b:} The second case with these boundary spin conditions also assumes $k\nequiv n_Q$. Then, assign $(a,b,c,d)$ as

\begin{center}
\begin{tabular}{|c|c|c|c|c|c|}\hline$a$&$b$&&$c$&$d$&\\$\sigma$&$\tau$&$\b$&$\theta$&$\rho$&$\a$\\\hline$k+1$&$1$&&$k$&$0$&\\+&+&$-$&+&$-$&+ \\\hline\end{tabular}
\end{center}
Again, we only have one state on each side of Equation (\ref{mybe}), which we write as tables, as we did above:

\begin{center}
\text{ left hand side} \hspace{4cm} \text{ right hand side}\\
\begin{tabular}{|c|c|c|c|}
\hline $e$&$f$&&\multirow{2}{*}{weight}\\$\nu$&$\mu$&$\g$&\\\hline\small{$k+1$}&\small{$1$}&&\multirow{2}{*}{\small{$\dfrac{g_Q(-k)g_Q(k)(1-\zn)}{1-v\zn}$}}\\+&+&$-$&\\\hline
\end{tabular}
\hspace{0.5cm}
\begin{tabular}{|c|c|c|c|}
\hline $g$&$h$&&\multirow{2}{*}{weight}\\$\phi$&$\psi$&$\d$&\\
\hline\small{$0$}&\small{$k$}&&\multirow{2}{*}{\small{$\dfrac{v(1-\zn)}{1-v\zn}$}}\\$-$&+&+&\\\hline\end{tabular}
\end{center}

Since we required our function $g_Q$ to satisfy $g_Q(a)g_Q(-a) = v$, the weights are the same.

\textbf{Case 10c:} We now account for the case where $k \equiv 0\pmod{n_Q}$, in which we have multiple states on the right hand side. Suppose our boundary spins are decorated as:

\begin{center}
\begin{tabular}{|c|c|c|c|c|c|}\hline$1$&$1$&&$0$&$0$&\\+&+&$-$&+&$-$&+ \\\hline\end{tabular}
\end{center}

Then our states complete as:

\begin{center}
\text{ left hand side} \hspace{4cm} \text{ right hand side}\\
\begin{tabular}{|c|c|c|c|}\hline $e$&$f$&&\multirow{2}{*}{weight}\\$\nu$&$\mu$&$\g$&\\
\hline\small{$1$}&\small{$1$}&&\multirow{2}{*}{\small{$\dfrac{-vz_j^{-n_Q}(-v+\zn)}{1-v\zn}$}}\\+&+&$-$&\\\hline
\end{tabular}
\hspace{0.5cm}
\begin{tabular}{|c|c|c|c|}
\hline $g$&$h$&&\multirow{2}{*}{weight}\\$\phi$&$\psi$&$\d$&\\
\hline\small{$0$}&\small{0}&&\multirow{2}{*}{\small{$\dfrac{-v(1-v)z_i^{-n_Q}\zn}{1-vq\zn}$}}\\+&$-$&$-$&\\\hline\small{$0$}&\small{0}&&\multirow{2}{*}{\small{$\dfrac{vz_j^{-n_Q}(1-\zn)}{1-v\zn}$}}\\$-$&+&+&\\\hline
\end{tabular}
\end{center}
Note that the weights on the right hand side add to equal the weight obtained from the left hand side, so our partition functions are equal.
\end{Example}

There are other types of Yang-Baxter equations satisfied by these weights.

\begin{Thm}\label{ybe2}
Fix boundary conditions $\a,\b,\gamma,\d,\epsilon,\phi$. Then, using the weights in Figures \ref{modwts} and \ref{modR}, the partition functions of the following two systems are equal.

\begin{align} \label{paramybe} \begin{array}{c} \resizebox{5cm}{4cm}{
\begin{tikzpicture}[baseline=3pt]
\coordinate (a1) at (0,0);
\coordinate (b1) at (0,2);
\coordinate (c1) at (0,4);
\coordinate (a2) at (2,0);
\coordinate (b2) at (2,2);
\coordinate (c2) at (2,4);
\coordinate (a3) at (4,0);
\coordinate (b3) at (4,2);
\coordinate (c3) at (4,4);
\coordinate (a4) at (6,0);
\coordinate (b4) at (6,2);
\coordinate (c4) at (6,4);
\coordinate (r) at (3,3);
\coordinate (s) at (1,1);
\coordinate (t) at (5,1);
\draw (a1) to [out=0,in=180] (b2) to [out=0,in=180] (c3) to [out=0,in=180] (c4);
\draw (b1) to [out=0,in=180] (a2) to [out=0,in=180] (a3) to [out=0,in=180] (b4);
\draw (c1) to [out=0,in=180] (c2) to [out=0,in=180] (b3) to [out=0,in=180] (a4);
\draw[fill=white] (a1) circle (.25);
\draw[fill=white] (b1) circle (.25);
\draw[fill=white] (c1) circle (.25);
\draw[fill=white] (a4) circle (.25);
\draw[fill=white] (b4) circle (.25);
\draw[fill=white] (c4) circle (.25);
\path[fill=white] (r) circle (.25);
\path[fill=white] (s) circle (.25);
\path[fill=white] (t) circle (.25);
\node at (t) {$R_{z_i,z_j}$};
\node at (r) {$R_{z_i,z_k}$};
\node at (s) {$R_{z_j,z_k}$};
\node at (a1) {$\alpha$};
\node at (b1) {$\beta$};
\node at (c1) {$\gamma$};
\node at (c4) {$\delta$};
\node at (b4) {$\epsilon$};
\node at (a4) {$\phi$};
\end{tikzpicture}} \end{array} \hspace{1cm}
\begin{array}{c} \resizebox{5cm}{4cm}{
\begin{tikzpicture}[baseline=3pt]
\coordinate (a1) at (0,0);
\coordinate (b1) at (0,2);
\coordinate (c1) at (0,4);
\coordinate (a2) at (2,0);
\coordinate (b2) at (2,2);
\coordinate (c2) at (2,4);
\coordinate (a3) at (4,0);
\coordinate (b3) at (4,2);
\coordinate (c3) at (4,4);
\coordinate (a4) at (6,0);
\coordinate (b4) at (6,2);
\coordinate (c4) at (6,4);
\coordinate (r) at (3,1);
\coordinate (s) at (1,3);
\coordinate (t) at (5,3);
\draw (c1) to [out=0,in=180] (b2) to [out=0,in=180] (a3) to [out=0,in=180] (a4);
\draw (b1) to [out=0,in=180] (c2) to [out=0,in=180] (c3) to [out=0,in=180] (b4);
\draw (a1) to [out=0,in=180] (a2) to [out=0,in=180] (b3) to [out=0,in=180] (c4);
\draw[fill=white] (a1) circle (.25);
\draw[fill=white] (b1) circle (.25);
\draw[fill=white] (c1) circle (.25);
\draw[fill=white] (a4) circle (.25);
\draw[fill=white] (b4) circle (.25);
\draw[fill=white] (c4) circle (.25);
\path[fill=white] (r) circle (.25);
\path[fill=white] (s) circle (.25);
\path[fill=white] (t) circle (.25);
\node at (s) {$R_{z_i,z_j}$};
\node at (r) {$R_{z_i,z_k}$};
\node at (t) {$R_{z_j,z_k}$};
\node at (a1) {$\alpha$};
\node at (b1) {$\beta$};
\node at (c1) {$\gamma$};
\node at (c4) {$\delta$};
\node at (b4) {$\epsilon$};
\node at (a4) {$\phi$};
\end{tikzpicture}} \end{array}
\end{align}
\end{Thm}

Theorem  \ref{ybe2} gives an example of a parametrized Yang-Baxter equation. We will delay the proof of this theorem, as well as the next, to Section \ref{Drin}, in order to make use of the relationship between our weights and the $R$-matrices attached to a particular quantum group.

\begin{Thm}\label{ybe3}
Fix boundary conditions $\a,\b,\gamma,\d$. Then the partition function of 
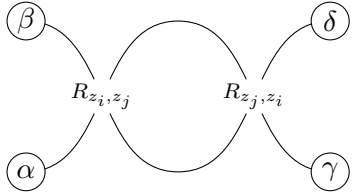
\[\begin{tikzpicture}
\coordinate (a1) at (0,0);
\coordinate (b1) at (0,2);
\coordinate (a2) at (2,0);
\coordinate (b2) at (2,2);
\coordinate (a3) at (4,0);
\coordinate (b3) at (4,2);
\coordinate (r) at (1,1);
\coordinate (s) at (3,1);
\draw (a1) to [out=0,in=180] (b2) to [out=0,in=180] (a3);
\draw (b1) to [out=0,in=180] (a2) to [out=0,in=180] (b3);
\draw[fill=white] (a1) circle (.25);
\draw[fill=white] (b1) circle (.25);
\draw[fill=white] (a3) circle (.25);
\draw[fill=white] (b3) circle (.25);
\path[fill=white] (r) circle (.25);
\path[fill=white] (s) circle (.25);
\node at (r) {$\scriptstyle R_{z_i,z_j}$};
\node at (s) {$\scriptstyle R_{z_j,z_i}$};
\node at (a1) {$\alpha$};
\node at (b1) {$\beta$};
\node at (a3) {$\gamma$};
\node at (b3) {$\delta$};
\end{tikzpicture}\]
equals
\[\left\{\begin{array}{ll}1&\text{if $\alpha=\gamma$,
  $\beta=\delta$}\\ 0 &\text{otherwise.}
\end{array}\right.\]
\end{Thm}

\begin{Remark} We have to check Theorem \ref{ybe1} by hand because we do not have a quantum interpretation for this Yang-Baxter equation, due to the fact that the vertical strand would correspond to a 2-dimensional quantum group module, and our quantum group has appears to have no such module. We could similarly check Theorems \ref{ybe2} and \ref{ybe3} by hand, but the connection to the quantum group, which is a natural algebraic source for solutions to Yang-Baxter equations, suggests recourse to a more direct proof. Note also that it was reasonable to expect an RRR relation, as in (\ref{paramybe}), because the intertwining operators satisfy braid relations. These results suggest that there may exist solvable lattice models with similar Yang-Baxter equations mimicking the relations of intertwining operators over other metaplectic groups.
\end{Remark}

\section{Relation to $\widehat{\mathfrak{gl}}(1|n)$ and the Drinfeld Twisting}\label{Drin}

One interesting property of quantum groups is that they are a natural source of solutions to the Yang-Baxter equations. In \cite{PS}, Perk and Schultz found new solutions to the Yang-Baxter equations, which were then shown to be related to the $R$-matrix of a quantum group by Yamane \cite{Yamane}; Zhang \cite{Zhang} then studied the affine version of this quantum group and its $R$-matrix. Bazhanov and Shadrikov \cite{BStriangle} introduced graded (supersymmetric) versions of the Yang-Baxter equations; we use Kojima \cite{Kojima} as a convenient reference for our case.

We will relate the $R$-vertex weights of Figure \ref{modR} to the universal $R$-matrix of the affine quantum group $U:= U_{\sqrt{v}}(\widehat{\frak{gl}}(1|n_Q)).$ All of our modifications will preserve the Yang-Baxter equations presented in Theorems \ref{ybe2} and \ref{ybe3}. Note that $U$ is a quasitriangular Hopf superalgebra, that is, a $\Z/2\Z$-graded Hopf algebra equipped with comultiplication $\Delta$, antipode $S$, and an invertible element $R \in U \widehat{\otimes} U$, which we call the universal $R$-matrix, such that
\begin{align*}
(\Delta\otimes \text{id}) R &= R_{13}R_{23},\\
(\text{id} \otimes \Delta) R &= R_{13}R_{12}, \\
\Delta^{op}(x) = R \Delta(x) R^{-1},& \hspace{0.5cm} \forall x \in \Uq{1|n_Q}
\end{align*}
where $\Delta^{op}(x) = \tau \Delta(x)$ and $\tau(a\otimes b) = b\otimes a$. In the first two equations, $R_{ij}$ is understood to be acting on the $i$- and $j$-th components of a tensor product of 3 modules. We refer the reader to Chari and Pressley \cite{CP}, Chapters 6 and 9 for an explicit description of generators and relations for this quantum group. In addition, Kojima presents in \cite{Kojima} an explicit formula for the universal $R$-matrix of $U$ achieved using this construction.

In order to examine Kojima's universal $R$-matrix, we must first understand the structure of certain  $U$-modules. For every $z\in \C^\times$, there is an $(1|n_Q)$-dimensional evaluation $U$-module $V_z$ with basis $v_0,...,v_{n_Q}$. Associate the decorated spins from Section \ref{metice} with the basis elements $v_0,...,v_{n_Q}$; that is, we assign $-0$ to $v_0$, and $+1,....,+n_Q$ to $v_1,...,v_{n_Q}$. Note that this matches the assignment discussed in Section \ref{intro}, with the addition of the basis element $v_0$ to match the addition of the decorated spin $-0$. Recall that $+0 = +n_Q$; we choose this representative to match the ordering $v_a \leq v_{n_Q}$ for all $a$. Note that this ordering also matches the calculations we did to prove Theorem \ref{ybe1}, in which $+a < +0$ for all $a\neq 0$. Accordingly, this puts the negative $-0$ in the $1$-graded piece and the positive spins $+a$ in the $0$-graded piece of $V_z$; let $[v_\a]$ denote the grading of $v_\a$. For the sake of brevity, we will conflate the basis element with the decorated spin in the following considerations. 

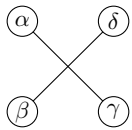
\[\gamgam{\beta}{\alpha}{\delta}{\gamma}{}{}{}{}\]

We then can compare the weight $R_{z_i,z_j}$ of the $R$-vertex above, which acts on strands $i$ and $j$ of our ice model, with the $R$-matrix $R_{\a,\b}^{\gamma,\delta}$ considered in Definition 2.1 of Kojima's \cite{Kojima}, which is induced by the action of the universal $R$-matrix for $U$ on the tensor products $V(z_i)\otimes V(z_j)$ of standard evaluation modules $V(z)$:

\begin{align}\label{Ronmod}
R_{i,j}(z)(v_\g\otimes v_\d) = \sum_{\g,\d} (v_\a\otimes v_\b)R_{\a,\b}^{\g,\d}(z).
\end{align}

In Figure \ref{kojrmx}, we have the $R$-vertex weights in this paper on the left and the weights used by Kojima on the right. Note that the $q$ in the Kojima weights is not the same as the $q$ we have been using for the characteristic of the residue field; rather, Kojima's $q = \sqrt{v}$ for $v$ as defined in Section \ref{metice}.

\begin{figure} [h] 
\noindent\makebox[1\textwidth][c]{
\begin{tabular}{|c|c|c|c|}
\hline
\multicolumn{2}{|c|}{}&Figure 3& Kojima\\
\hline\xrowht[()]{30pt}
\multirow{3}{1em}{\vspace{-1.5cm}$\tt{a}_1$} &I& $\dfrac{-v + \boldsymbol{z}^{n_Q \alpha^\vee}}{1-v \boldsymbol{z}^{n_Q \alpha^\vee}}$&$\dfrac{-q^2+z}{1-q^2z}$\\
\cline{2-4} \xrowht[()]{30pt}
&II& $g_Q(a-b) \dfrac{1- \boldsymbol{z}^{n_Q \alpha^\vee}}{1-v \boldsymbol{z}^{n_Q \alpha^\vee}}$&$\dfrac{q(1-z)}{1-q^2z}$\\
\cline{2-4} \xrowht[()]{30pt}
&III& $\begin{aligned}[c]&\dfrac{(1-v)}{1-v \boldsymbol{z}^{n_Q \alpha^\vee}}\cdot\begin{cases}\boldsymbol{z}^{n_Q \alpha^\vee}&\text{if $a > b$}\\1&\text{if $a < b$}\end{cases}\end{aligned}$ & $\begin{aligned}[c]&\dfrac{(1-q^2)}{1-q^2z}\cdot\begin{cases}z&\text{if $a > b$}\\1&\text{if $a < b$}\end{cases}\end{aligned}$\\
\hline \xrowht[()]{20pt}
$\tt{a}_2$&IV& 1&-1\\
\hline \xrowht[()]{30pt}
$\tt{b}_1$&V& $\dfrac{v(1 - \boldsymbol{z}^{n_Q \alpha^\vee})}{1-v \boldsymbol{z}^{n_Q \alpha^\vee}}$&$\dfrac{q(1-z)}{1-q^2z}$\\
\hline \xrowht[()]{30pt}
$\tt{b}_2$& VI &$\dfrac{1 - \boldsymbol{z}^{n_Q \alpha^\vee}}{1-v \boldsymbol{z}^{n_Q \alpha^\vee}}$&$\dfrac{q(1-z)}{1-q^2z}$\\
\hline \xrowht[()]{30pt}
$\tt{c}_1$& VII &$\dfrac{(1-v)\boldsymbol{z}^{n_Q\a}}{1-v \boldsymbol{z}^{n_Q \alpha^\vee}}$&$\dfrac{(1-q^2)z}{1-q^2z}$\\
\hline \xrowht[()]{30pt}
$\tt{c}_2$& VIII & $\dfrac{(1-v)}{1-v \boldsymbol{z}^{n_Q \alpha^\vee}}$&$\dfrac{1-q^2}{1-q^2z}$\\
\hline
\end{tabular}}
\caption{We compare the $R$-vertex weights of this paper to the values of Kojima's $\widehat{\mathfrak{g}\mathfrak{l}}(1|n)$
$R$-matrix. Note: it is assumed that $a\not\equiv b \pmod{n}$. For consistency, the numberings in the second column correspond to the categorization of R-weights used in Section 4 of \cite{BBBIce}.}
\label{kojrmx}
\end{figure}

Note the similarities between our weights and Kojima's weights. Since Kojima's weights already satisfy graded versions of the desired Yang-Baxter equations of Theorems \ref{ybe2} and \ref{ybe3}, we will exploit this similarity to prove that the ungraded versions of these equations hold with our weights.

\begin{Lemma}[Kojima, Section 2 of \cite{Kojima}]\label{Kybes}
Let $R_{i,j}(z)$ be the weight induced by the action of the $R$-matrix $R(z)$ on $V_{z_i}\otimes V_{z_j}$. The $R$-matrix $R(z)$ satisfies the graded Yang-Baxter equations
\begin{align}\label{Kybe2}
R_{i,j}(z_i/z_j) R_{i,k}(z_i/z_k) R_{j,k}(z_j/z_k) = R_{j,k}(z_j/z_k) R_{i,k}(z_i/z_k) R_{i,j}(z_i/z_j),
\end{align}
where we view $R(z)$ as acting on $V(z_i) \otimes V(z_j) \otimes V(z_k)$ in the equation above, and 
\begin{align}\label{Kybe3}
R_{i,j}(z) R_{j,i}(1/z) = 1. 
\end{align}
\end{Lemma}

Substituting $\zn$ for $z$, we have identical weights in cases I, III,VII, and VIII. We can then modify the Kojima weights in cases II, V, and VI by a procedure called a \emph{Drinfeld twist}, originally due to Drinfeld \cite{Drin} and Reshetikhin \cite{Res}, which preserves the graded Yang-Baxter equations (\ref{Kybe2}) and (\ref{Kybe3}). A convenient reference for our case is Section 4 of \cite{BBBFhecke}, in which Brubaker, Buciumas, Bump, and Friedburg consider a particular Drinfeld twist of the quantum group $\Uv{n}$ in the case of the $n$-fold metaplectic cover. 

Let $F = \sum_i f_i\otimes f^i \in \Uv{1|n_Q} \otimes \Uv{1|n_Q}$ be an invertible element satisfying
\begin{align}\label{Drinconds}
\begin{split}
(\Delta \otimes \text{id}) F &= F_{13}F_{23}, \hspace{0.5cm} (\text{id}\otimes \Delta)F = F_{13}F_{12},\\
F_{12}F_{13}F_{23} &= F_{23}F_{13}F_{12}, \hspace{0.5cm} FF_{21} = 1
\end{split}
\end{align}
where $F_{21} = \tau F = \sum_i f^i \otimes f_i$. Let $u = \sum_i f^i S(f_i)$. We can then use $F$ to twist our original quantum group $\Uq{1|n_Q}$.
\begin{Thm}[Reshetikhin, Theorem 1 of \cite{Res}]
Let $U_{q}^F(\widehat{\frak{gl}}(1|n_Q)$ be the coalgebra that is $\Uq{1|n_Q}$ as a vector space, with the same unit, counit, and multiplication as $\Uq{1|n_Q}$, but comultiplication, antipode, and universal $R$-matrix given by
\begin{align*}
\Delta^F(a) = F\Delta(a)F^{-1}, \hspace{0.2cm} S^F(a) = uS(a)u^{-1},\hspace{0.2cm} R^F = F_{21}RF^{-1}.
\end{align*}
Then $\UqF{1|n_Q}$ is a quasitriangular Hopf superalgebra, called the \emph{Drinfeld twist} of $\Uq{1|n_Q}$.
\end{Thm}

Following the notation of \cite{Res}, choose $F$ to be 
\begin{align}\label{F} F &= \sum_{i} e_{i,i}\otimes e_{i,i} + \sum_{+i}q^{-1/2}\cdot e_{-0,+i}\otimes e_{+i,-0} + \sum_{+i} q^{1/2}\cdot e_{+i,-0} \otimes e_{-0,+i}\\
&+\sum_{+i < +j}\left(q\cdot g_Q(i-j)\right)^{1/2} \cdot e_{+i,+j}\otimes e_{+j,+i}+\sum_{+i < +j}\left(q\cdot g_Q(i-j)\right)^{-1/2}\cdot e_{+j,+i}\otimes e_{+i,+j}
\end{align}
Straightforward calculations verify that $F$ satisfies (\ref{Drinconds}), so we may use this $F$ to take a Drinfeld twist. We may similarly express the Kojima $R$-matrix in such a form, 
\begin{align*}
R &= -e_{-0,-0}\otimes e_{-0,-0} + \sum_{+i} \frac{-q^2 + z}{1-q^2z} \cdot e_{+i,+i}\otimes e_{+i,+i} + \sum_{i\neq j} \frac{q(1-z)}{1-q^2z} \cdot e_{i,i} \otimes e_{j,j}\\
& +\sum_{i<j} \frac{1-q^2}{1-q^2z}\cdot e_{i,j}\otimes e_{j,i} + \sum_{i<j} \frac{(1-q^2)z}{1-q^2z} \cdot e_{j,i} \otimes e_{i,j}.
\end{align*} 
Applying the Drinfeld twist using $F$ in Equation (\ref{F}) yields the universal $R$-matrix of $U^F$:
\begin{align*}
R^F &= -e_{-0,-0}\otimes e_{-0,-0} + \sum_{+i} \frac{-q^2 + z}{1-q^2z} \cdot e_{+i,+i}\otimes e_{+i,+i}\\
& +\sum_{+i\neq +j} g_Q(i-j)\frac{1-z}{1-q^2z} \cdot e_{+i,+i} \otimes e_{+j,+j}\\
& +\sum_{+i} \frac{q^2(1-z)}{1-q^2z} \cdot e_{-0,-0} \otimes e_{+i,+i}+\sum_{+i} \frac{1-z}{1-q^2z} \cdot e_{+i,+i} \otimes e_{-0,-0}\\
& +\sum_{i<j} \frac{1-q^2}{1-q^2z}\cdot e_{i,j}\otimes e_{j,i} + \sum_{i<j} \frac{(1-q^2)z}{1-q^2z} \cdot e_{j,i} \otimes e_{i,j}.
\end{align*} 
Note that only the $e_{i,i}\otimes e_{j,j}$ terms have been affected, which are precisely cases II, V, VI in Figure \ref{kojrmx}. For case II, we note also that this calculation uses the fact that $g_Q(i-j)g_Q(j-i) = v$ to combine cases when $i <j$ and $i>j$. Futhermore, when we plug in $q^2 = v$ and $z = \zn$, the new coefficients of the terms for cases II, V, and VI precisely match the $R$-matrix weights in Figure \ref{modR}.

Now, we have matched all cases except IV. However, we note that the Yang-Baxter equations in Lemma \ref{Kybes} are both graded Yang-Baxter equations. Kojima remarks that multiplying his weights $R^{j_1,j_2}_{k_1,k_2}(z)$ by the signature $(-1)^{[v_{k_1}][v_{k_2}]}$ gives the Perk-Schultz $R$-matrix $R^{PS}(z)$, which satisfies the ungraded versions of these Yang-Baxter equations. This amounts to changing the sign on all cases with only odd-graded spins (i.e. I, II, and III), but it works just as well to change the sign on all cases with only even-graded spins, which only affects VI.

\begin{proof}[Proof of Theorems \ref{ybe2} and \ref{ybe3}]
Starting with Kojima's $R$-matrix weights $R_{\g,\d}^{\a,\b}$, we apply the two methods of changing our $R$-matrices discussed above to obtain the $R$-vertex weights in Figure \ref{modR}. Then Theorems \ref{ybe2} and \ref{ybe3} follow from Equations (\ref{Kybe2}) and (\ref{Kybe3}) of Lemma \ref{Kybes}, respectively.
\end{proof}

Note that the methods we used in this section apply to reductive groups in general, as opposed to just $GL_r(F)$, so one future direction for our research would be to develop analogous $R$-vertex weights satisfying similar Yang-Baxter equations for metaplectic covers of other reductive groups.

\section{Proof of Theorem \ref{maintheorem}}\label{pfmain}

We now have assembled all the necessary pieces to prove the first of our main results. Let $\twid{G}$ be the metaplectic cover associated to a bilinear form $B$ and quadratic form $Q$. Set $v = q^{-1}$, $g_Q$ to the $n_Q$-th Gauss sum given in Equation (\ref{gausssum}), and $n_Q = n/\text{gcd}(n, Q(\a^\vee))$ for any simple coroot $\a^\vee$.

\begin{Prop}\label{tauR}
The $R$-matrix weights in Figure \ref{modR} on vertices with all positive spins are precise normalizations of the entries in the scattering matrix. That is to say, let $\boldsymbol{c} = (c_1,...,c_r)$ be an integer partition with $c_i \in (0,n]$ for all $i$. Let $\nu = \rho - \boldsymbol{c}$, and let $\tau_{\nu,\mu}(\boldsymbol{z}) := \tau_{\nu,\mu}^{s_i}(\boldsymbol{z})$ as in Proposition \ref{TAU}, where the reflection $s_i$ corresponds to the simple coroot $\a_i^\vee$. Let \emph{wt} return the Boltzmann weight of a vertex according to Figures \ref{modwts} and \ref{modR}. Then, given any pair of integers $a,b$ with $a \equiv c_i$ and $b\equiv c_{i+1}\pmod{n_Q}$, if $a\not\equiv b\pmod{n_Q}$, then
\[ \tau_{\nu,\nu}^1(\boldsymbol{z}) = \boldsymbol{z}^{(c_i - c_{i+1} -1)\a_i^\vee}\emph{wt}\left( \begin{array}{c}\gamgam{+}{+}{+}{+}{a}{b\vphantom{b}}{b\vphantom{b}}{a}\end{array}\right) \text{ \hspace{0.5cm} and \hspace{0.5cm}} \tau_{\nu, s_i\nu + \a_i^\vee}^2(\boldsymbol{z}) = \boldsymbol{z}^{-\a_i^\vee}\emph{wt}\left( \begin{array}{c}\gamgam{+}{+}{+}{+}{a}{b\vphantom{b}}{a\vphantom{b}}{b}\end{array}\right). \]
If $a\equiv b\pmod{n_Q}$, then 
\[\tau_{\nu,\nu}^1(\boldsymbol{z}) + \tau_{\nu,s_i\nu + \a_i}^2(\boldsymbol{z}) = \boldsymbol{z}^{-\a_i^\vee}\emph{wt}\left( \begin{array}{c}\gamgam{+}{+}{+}{+}{a}{a\vphantom{b}}{a\vphantom{b}}{a}\end{array}\right).\]
\end{Prop}

In the case where $n_Q = n$, Proposition \ref{tauR} recovers Proposition 5.3 of \cite{BBBIce}.

\begin{proof}
For notational clarity, we set $s:=s_i$ and $\a^\vee:=\a_i^\vee$. Then we plug $\nu$ into the equations from Proposition \ref{tauR}, recalling that $[n_Q] = 0$, not $n_Q$. Here $\nu_i = r -i - c_i$, so
\begin{align*}
\tau_{\nu,\nu}^1(\boldsymbol{z}) = (1-v)\frac{\boldsymbol{z}^{[\nu_{i+1} - \nu_i]\a^\vee}}{1-v\boldsymbol{z}^{n_Q\a^\vee}} &= (1-v) \frac{\boldsymbol{z}^{[c_i - c_{i+1} - 1]\a^\vee}}{1-v\zn}\\
&= \boldsymbol{z}^{(c_i-c_{i+1}-1)\a^\vee}\frac{1-v}{1-v\zn}\cdot \begin{cases} \zn & \text{ if } c_i < c_{i+1}\\ 1 & \text{ if }c_i > c_{i+1}.\end{cases}
\end{align*}
Similarly,
\begin{align*}
\tau_{\nu, s\nu + \a^\vee}^2(z) &= g(-B(\a^\vee, s\nu+\a^\vee) + Q(\a^\vee)) \cdot \boldsymbol{z}^{-\a^\vee}\frac{1-\boldsymbol{z}^{n_Q\a^\vee}}{1-v\boldsymbol{z}^{n_Q\a^\vee}}\\ &= \boldsymbol{z}^{-\a^\vee}g_Q(c_{i+1}-c_i)\frac{1-\zn}{1-v\zn}
\end{align*}
Checking against the table in Figure \ref{modR}, we obtain the desired relations in the first case. For the case of $a\equiv b\pmod{n_Q}$, that is $c_i = c_{i+1}$, we have $g_Q(c_{i+1}-c_i) = g_Q(0) = -v$, so
\begin{align*}
\tau_{\nu,\nu}^1(\boldsymbol{z}) + \tau^2_{\nu,s\nu + \a^\vee}(\boldsymbol{z}) &= \boldsymbol{z}^{-\a^\vee}\cdot \frac{(1-v)\zn}{1-v\zn} - v\boldsymbol{z}^{-\a^\vee}\cdot\frac{1-\zn}{1-q^{-1}\zn}\\ &= \boldsymbol{z}^{-\a^\vee}\cdot\frac{\zn - v}{1-v\zn}
\end{align*}
which gives us the desired equality.
\end{proof}

Let $\theta_{\boldsymbol{z}}: \frak{W}^{\boldsymbol{z}} \rightarrow \bigotimes_i V_+(z_i)$ be the map
\[\theta_{\boldsymbol{z}}(\mathcal{W}^{\boldsymbol{z}}_\nu) = \boldsymbol{z}^{-\nu}\cdot  v_{\boldsymbol{c}}\]
for $\nu = -\boldsymbol{c} + \rho$, where $\boldsymbol{c} = (c_1,...,c_r)$ and $v_{\boldsymbol{c}} = v_{c_1} \otimes \cdots \otimes v_{c_r}$.

\begin{proof}[Proof of Theorem \ref{maintheorem}]
Given a set of representatives $\Gamma$ for $\mathfrak{W}^{\boldsymbol{z}}$, set the representatives for $\mathfrak{W}^{s_i \boldsymbol{z}}$ to be $\Gamma^{s_i} := s_i \Gamma + \a_i^\vee$. Consider $\nu \in \Gamma$: if $s_i\nu + \a \not\sim \nu$, both $\nu$ and $s_i\nu + \a$ appear in $\Gamma^{s_i}$, whereas if $s_i\nu + \a \sim \nu$, only $s_i\nu + \a$ appears in $\Gamma^{s_i}$. Tracing $\mathcal{W}^{\boldsymbol{z}}_\nu$ through the diagram of Theorem \ref{maintheorem}, 
\begin{center}
\begin{tikzcd}
\frak{W}^{\boldsymbol{z}} \arrow[r,"\theta_{\boldsymbol{z}}"] \arrow[d,"\overline{\mathcal{A}}_{s_i}" left]& V_+(z_1) \otimes \cdots \otimes V_+(z_i) \otimes V_+(z_{i+1}) \otimes \cdots \otimes V_+(z_r)\arrow[d, "\left(\tau R_{z_i,z_{i+1}}\right)_{i,i+1}"]\\
\frak{W}^{s_i\boldsymbol{z}} \arrow[r,"\theta_{s_i\boldsymbol{z}}"] &V_+(z_1) \otimes \cdots \otimes V_+(z_{i+1}) \otimes V_+(z_i) \otimes \cdots \otimes V_+(z_r)
\end{tikzcd}
\end{center}
we first consider $\theta_{s_i\boldsymbol{z}}\overline{\mathcal{A}}_{s_i}(\mathcal{W}^{\boldsymbol{z}}_\nu)$. There are precisely two cosets on which $\tau_{\nu,\mu}$ is nonzero, namely $\mu \sim \nu$ and $\mu \sim s_i\nu + \a_i^\vee$. Matching the representatives for $\Gamma^{s_i}$, when these are different cosets, we will use $\nu$ for the first and set $\mu := s_i \nu + \a_i^\vee$ for the second; when the same, we will use only $\mu$. Note that for $\nu \in \Gamma$ written as $\nu -\rho = \boldsymbol{c}$, we have $\mu -\rho = s_i(\nu-\rho) = s_i \boldsymbol{c}$. We split into cases here: suppose $\nu \not\sim s_i\nu + \a^\vee_i$. Then
\begin{align}\label{side1} \theta_{s_i\boldsymbol{z}}\overline{\mathcal{A}}_{s_i}(\mathcal{W}^{\boldsymbol{z}}_\nu) = \theta_{s_i\boldsymbol{z}}\left(\tau_{\nu,\nu}^1\mathcal{W}^{\boldsymbol{z}}_\nu + \tau_{\nu,\mu}^2\mathcal{W}^{\boldsymbol{z}}_\mu\right) = (s_i\boldsymbol{z})^{-\nu}\tau_{\nu,\nu}^1v_{\nu-\rho} + (s_i\boldsymbol{z})^{-\mu}\tau_{\nu,\mu}^2v_{\mu-\rho}.\end{align}
If we follow the other direction around the diagram, we have $\theta_{\boldsymbol{z}}(\mathcal{W}^{\boldsymbol{z}}_\nu) = \boldsymbol{z}^{-\nu}v_{\nu-\rho}$ and $\left(\tau R_{z_i,z_{i+1}}\right)_{i,i+1}$ acts on the $v_{c_i}\otimes v_{c_{i+1}}$ component of $v_{\nu-\rho}$ as
\[ \tau R(v_{c_i} \otimes v_{c_{i+1}}) = \sum_{c_k,c_l} R_{c_k,c_l}^{c_i,c_{i+1}} (v_{c_k} \otimes v_{c_l}).\]
The only nonzero terms on the right hand side are those corresponding to $v_{c_i}\otimes v_{c_{i+1}}$ and $v_{c_{i+1}}\otimes v_{c_i}$. Thus, we have that
\begin{align}\label{side2a}\left(\tau R_{z_i,z_{i+1}}\right)_{i,i+1}\theta_{\boldsymbol{z}}(\mathcal{W}^{\boldsymbol{z}}_\nu) = \boldsymbol{z}^{-\nu}\left(R_{c_i,c_{i+1}}^{c_i,c_{i+1}}\right)_{i,i+1} v_{\nu - \rho} +  \boldsymbol{z}^{-\nu}\left(R_{c_{i+1},c_i}^{c_i,c_{i+1}}\right)_{i,i+1} v_{\mu - \rho}.\end{align}
Since $(s_i\boldsymbol{z})^\nu \cdot \boldsymbol{z}^{-\nu} = \boldsymbol{z}^{(c_i-c_{i+1}-1)\a^\vee}$ and $(s_i\boldsymbol{z})^\mu \cdot \boldsymbol{z}^{-\nu} = \boldsymbol{z}^{-\a^\vee}$, when we compare like terms on the right hand sides of (\ref{side1}) and (\ref{side2a}), the first half of Proposition \ref{tauR} gives us equality of coefficients.

If instead $\nu \sim s_i\nu + \a_i^\vee$, then
\begin{align*}\theta_{s_i\boldsymbol{z}}\overline{\mathcal{A}}_{s_i}(\mathcal{W}^{\boldsymbol{z}}_\nu) = \theta_{s_i\boldsymbol{z}}\left((\tau_{\nu,\nu}^1+ \tau_{\nu,\mu}^2)\mathcal{W}^{\boldsymbol{z}}_\mu\right) = (s_i\boldsymbol{z})^{-\mu}(\tau_{\nu,\nu}^1 + \tau_{\nu,\mu}^2) \ v_{\mu-\rho}.\end{align*}

Going around the other side, note that $c_i = c_{i+1}$, so $\mu-\rho = \nu - \rho$. The $R$-matrix action leaves only one term, corresponding to $v_{c_i} \otimes v_{c_i}$,
\begin{align*} \left(\tau R_{z_i,z_{i+1}}\right)_{i,i+1}\theta_{\boldsymbol{z}}(\mathcal{W}^{\boldsymbol{z}}_\nu) = \boldsymbol{z}^{-\nu}\left(R_{c_i,c_i}^{c_i,c_i}\right)_{i,i+1} v_{\nu - \rho}.\end{align*}

We already checked that $(s_i\boldsymbol{z})^\mu \cdot \boldsymbol{z}^{-\nu} = \boldsymbol{z}^{-\a^\vee}$, so the second half of Proposition \ref{tauR} provides equality of coefficients as desired. 
\end{proof}

\begin{Remark}
As constructed, $\theta_{\boldsymbol{z}}$ is a vector space homomorphism. Given a bilinear form such that each coset $\Gamma$ contains only one equivalence class modulo $n_Q$, however, we obtain an injection. In this case, we may adjust slightly to $\theta_{\boldsymbol{z}}(\mathcal{W}_\nu^{\boldsymbol{z}}) = \boldsymbol{z}^{[-\nu ]} \cdot v_{\boldsymbol{c}}$ for $\boldsymbol{c} = \rho - \nu \pmod{n_Q}$ to make the entire diagram independent of choice of coset representatives.
\end{Remark}

\section{Metaplectic Ice and Whittaker Functions}\label{icewhit}

One of our goals in defining this structure of metaplectic ice is to relate it back to metaplectic covers via the following result.

\begin{Thm}\label{BIGTHM} Let $\lambda$ be a partition with $r$ parts and $n_Q$ a fixed positive integer. Then the partition function $Z(\frak{S}_\lambda)$ is (up to normalization) a value of a $p$-adic spherical Whittaker function on a metaplectic cover of $GL(r,F)$.
\end{Thm}

Note: for the bilinear form corresponding to the dot product, which has $n_Q = n$, this is Theorem 2.5 of \cite{BBBIce}; however, to prove the general case requires modifications to the existing literature, since many of the results used in \cite{BBBIce} are also proven only in the dot product case.

\begin{proof} Recall from Section \ref{whitcalcs} that we have an expression for $I_\lambda$ in terms of a sum over nodes in the Kashiwara crystal $B(\lambda+ \rho)$:
\begin{align} I_\lambda = \sum_{(\boldsymbol{i},\boldsymbol{m}) \in B(\lambda + \rho)} \prod_{\a\in \Phi^+} w(\boldsymbol{m},\a)x_\a^{m_\a}.\end{align}\label{WZsum}

There is a bijection between nodes $v = (\boldsymbol{i},\boldsymbol{m}) \in B(\lambda + \rho)$ and Gelfand-Tsetlin patterns $\frak{T}$ with top row $\lambda + \rho$ \cite{BBF}. We direct the reader to \cite{BBFbook} and \cite{Littelmann} for a careful treatment of how Gelfand-Tsetlin patterns relate to crystals and ice models in general; for the purposes of this paper, it suffices to know the connection for Type A. Without loss of generality, we choose $\boldsymbol{i}$ to be the decomposition of the long word as $w_0 = s_r(s_{r-1}s_r)(s_{r-2}s_{r-1}s_r) \cdots(s_1s_2\cdots s_{r-1}s_r)$ as in Section \ref{whitcalcs} and in \cite{BBF}. We have one entry in $\boldsymbol{m} = (m_\a)_{\a\in\Phi^+}$ for every positive root $\a = (i,j)$, which we index by their matrix representations. Given a strict Gelfand-Tsetlin pattern
\begin{align*}
\frak{T} = \left\{ \begin{matrix} a_{0,0} &&a_{0,1} && a_{0,2} &&\cdots&a_{0,r-2}&&a_{0,r-1}\\
&a_{1,1}&&a_{1,2}&&a_{1,3}&\cdots&&a_{1,r-1}&\\
&&\ddots&&&&&\iddots\\
&&&&&a_{r-1,r-1}
\end{matrix} \right\}
\end{align*}
the bijection gives us that $m_{i,j} = a_{r-j+2, r-i} - a_{r-j+1, r-i}$. Note that the inequalities of Proposition \ref{mdef} derive from the Gelfand-Tsetlin pattern inequalities $a_{k,l} \leq a_{k+1,l} \leq a_{k, l -1}$.

Define
\[ G((\boldsymbol{i},\boldsymbol{m})) = \prod_{\a \in \Phi^+} w(\boldsymbol{m},\a). \]
where the weighting function on the crystal is defined in Equation (\ref{weighting}). 

We claim that $G((\boldsymbol{i},\boldsymbol{m}))$ gives the same coefficient as does the Gelfand-Tsetlin description of the Whittaker function. Recall that we may weight a Gelfand-Tsetlin pattern $\frak{T}$ by attaching an integer $e_{i,j}$ defined as
\[ e_{i,j} = \sum_{k=j}^{r-1} a_{i,k} - a_{i-1,k}\]
to each entry $a_{i,j}$, then assigning a weight $\g(a_{i,j})$ defined by
\[ \g(a_{i,j}) = \begin{cases} q^{e_{i,j}} & \text{ if } a_{i-1,j-1} > a_{i,j} = a_{i-1,j}\\ g(e_{i,j},0) & \text{ if } a_{i-1,j-1} > a_{i,j} > a_{i-1,j}\\ g(e_{i,j},-1) & \text{ if } a_{i-1,j-1} = a_{i,j} > a_{i-1,j}\\  0 & \text{ if } a_{i-1,j-1} = a_{i,j} = a_{i-1,j}\\\end{cases},\]
where $g(a,b)$ is the Gauss sum used in \cite{McNcrystal}. The weight of the whole pattern is then $G(\frak{T}) = \prod_{1\leq i\leq j\leq r-1} \g(a_{i,j})$.

Examining the correspondence, we see that $m_\a$ is boxed if its corresponding pattern entry is left-leaning, in which case $s_\a = -1$, and circled if its entry is right-leaning. If $m_\a$ is neither boxed nor circled, the corresponding entry has strict inequalities, in which case $s_\a \geq 0$, and we recall that $g(a,b) = g(a,0)$ if $b\geq 0$. For $\a = (i,j)$, we have that $r_{i,j} = e_{r-j+2,r-i}$. Therefore, for $\frak{T}$ the Gelfand-Tsetlin pattern corresponding to $v = (\boldsymbol{i},\boldsymbol{m})$:
\[ G(v) = G(\frak{T}).\]
As in Theorem \ref{crystalsum}, the modification to the proof of Proposition 5 of \cite{BBF} necessary for our case amounts to replacing the $n$-th Gauss sum with the $n_Q$-th Gauss sum. So we may rewrite our expression for $I_\lambda$ as a sum over Gelfand-Tsetlin patterns. Furthermore, by Proposition 1 of \cite{MetIce}, there is a bijection between strict Gelfand-Tsetlin patterns with top row $\lambda+\rho$ and admissible states of ice with boundary conditions $\lambda$. Recall that we form the Gelfand-Tsetlin pattern corresponding to an ice state by listing off the vertical columns with negative spin above each row, from top to bottom. 
For instance, given our choice of long word, the crystal node and Gelfand-Tsetlin pattern corresponding to the ice state in Figure \ref{icestate} are
\begin{align*}
(\boldsymbol{i},\boldsymbol{m}) = ((2,1,2), (0,2,1)))\hfill \text{ and } \hfill \frak{T} = \left\{\begin{matrix} 4&&3&&0\\&4&&2\\&&2 \end{matrix}\right\}
\end{align*}

The bijection between crystal node weights and Gelfand-Tsetlin pattern weights is irrespective of the choice of metaplectic cover of $G$; the only modification necessary for our case is to change all of the $n$-Gauss sums to $n_Q$-Gauss sums. Since we do this on both sides of the bijection, the proofs given by Brubaker, Bump, and Friedberg in \cite{BBF} for the cover corresponding to the dot product extend as written. On the other hand, while the bijection between Gelfand-Tsetlin patterns and ice states doesn't depend on the metaplectic cover, we must check that it is weight preserving, i.e. that $G(\frak{T})$ is the coefficient of $B^{n_Q}(\frak{s})$ for the matching ice state $\frak{s}$.  It suffices to check that terms that are not $n_Q$-admissible have weight 0. To do this, we note that the Gauss sums $g(r,s)$ with $s\geq 0$ are nonzero only when $s \equiv 0$ modulo our chosen modulus. These sums occur when an entry in the Gelfand-Tsetlin pattern is neither left-leaning nor right-leaning, which forces a vertex in the ice state to take type $\boldsymbol{c}_1$ with charge $s\equiv 0 \pmod{n_Q}$ on both adjacent horizontal edges. Changing our Gauss sum to modulus $n_Q$ then causes this weight to be zero unless the ice state is $n_Q$-admissible. Thus, up to normalization of the power of $\boldsymbol{z}$, these bijections allow us to rewrite the sum in (\ref{WZsum}) as the partition function of the system $\frak{S}_\lambda$. 
\end{proof}

For the sake of completeness, we determine the precise normalization necessary. Of particular note to this normalization will be the fact that when we consider the boundary conditions on a system, we take all charges as residues mod $n_Q$, and as in previous calculations, if $c_i \equiv 0\pmod{n_Q}$, we consider it as $c_i = n_Q$, rather than $c_i=0$.
 
\begin{Thm}\label{indivWhit}
Let $\frak{S}_\lambda$ be the set of boundary conditions given by the partition $\lambda$, as described in Section \ref{metice}; note that this includes a choice of $N$ columns. Let  $\g \in \Gamma$, where we choose representatives according to Remark \ref{choosy}. Let  $\boldsymbol{c} \in (\Z/n_Q\Z)^r$ be defined by $\boldsymbol{c} = [\overline{N}- \g -w_0 \rho]$, where $[\cdot]$ denotes the residue mod $n_Q$ (following the convention above), we set $w_0(x_1,...,x_r) := (x_r,...,x_1)$, and $\overline{k} := (k,....,k)$ for any integer $k$. Then using the weights in Figure \ref{modwts}, we have
\begin{align}\label{thm82}
Z(\frak{S}_\lambda; \boldsymbol{c}) = \boldsymbol{z}^{-\overline{N} + w_0\rho  + \boldsymbol{c}} \cdot \d^{-1/2}(\varpi^\lambda)\mathcal{W}_\gamma^{\boldsymbol{z}}(\varpi^\lambda).
\end{align} 
\end{Thm}

\begin{Example}\label{section8ex} Referring back to the ice state in Figure \ref{icestate}, we had $\lambda = (2,2,0)$ and $N = 5$. This state was $n_Q$-admissible for $n_Q = 1,2$, so suppose $n_Q = 2$; the charge is then $\boldsymbol{c} = (1,1,2)$. We calculated the Boltzmann weight of that state to be $-g_Q(1)(1-v)z_1^{-2}z_3^{-2}$. Via the bijection above, we see that the corresponding crystal node is $(\boldsymbol{i},\boldsymbol{m}) = ((2,1,2),(0,2,1))$, with $\boldsymbol{m} = (m_{12},m_{13},m_{23})$. This node contributes the term $g(3,-1)g(2,0)z_1^2z_2^1z_3^{-3}$ to $I_{(4,3,3), (2,2,0)}$ and thus the term $g(3,-1)g(2,0)z_1^2z_2^3z_3^{-1}$ to $\mathcal{W}_{(4,3,3)}^{\boldsymbol{z}}(\varpi^{(2,2,0)}).$ Note that this power of $\boldsymbol{z}$ is in $(4,3,3)\Lambda$.
The Gauss sums match up, since $g(3,-1) = g_Q(1)$ and $g(2,0) = 1-v$, and multiplying this crystal node term by the normalizing power $z^{(-4,-3,-1)}$, where here $(-4,-3,-1) = - (5,5,5) +(0,1,2)+ (1,1,2),$ gives the correct power of $z$ of our Boltzmann weight. 
\end{Example}

\begin{proof}
It suffices to determining a map between the charge $\boldsymbol{c}$ and the parameters $\g,\lambda$ that suits the bijection of the proof above, which then will provide the desired normalization factor. From Equation (\ref{Ilambda}), we know that $I_{\gamma,\lambda}$ is supported on powers of $\boldsymbol{z}$ in the coset $(\gamma-w_0\lambda)\Lambda$. (Accordingly, $\mathcal{W}_\g^{\boldsymbol{z}}(\varpi^\lambda)$ is supported on the coset $\gamma\Lambda$.) Translating between the powers $x_\a$ used by McNamara and our powers of $\boldsymbol{z}$ by $x_\a = \boldsymbol{z}^\a$, we see that the power of $\boldsymbol{z}^{\boldsymbol{p}}$ provided by the node $(\boldsymbol{i},\boldsymbol{m})$ is 
\[ p_i = \sum_{j> i}m_{i,j} - \sum_{k<i}m_{k,i}.\]
On the other hand, we may express the total charge of an ice state, before we reduce all our charges modulo $n_Q$, in terms of the entries of the corresponding Gelfand-Tsetlin pattern. As we travel along a row from right to left, the charge increases by one every time we encounter a vertex with a plus in the western position; since the right hand boundary is a minus sign, charge will not begin to accumulate until we hit a vertex of type $\tt{c_2}$. From there, each vertex will increase the charge by one until we encounter a vertex of type $\tt{c_1}$, at which point, charge remains stagnant until another $\tt{c_2}$ vertex occurs. In order to connect back to the pattern, we incorporate the type $\tt{b_1}$ vertex as well and notice that this allows us to say that charge starts incrementing when we hit a vertex with a minus in the northern position and stops incrementing when we reach a vertex with a minus in the southern position. (With a type $\tt{b_1}$ vertex, these are the same vertex, which explains the lack of interruption.) However, for row $i$, northern minus signs correspond to Gelfand-Tsetlin entries in row $r-i$ and southern ones to entries in row $r-i+1$. Thus, we have that the unmodulated charge $\boldsymbol{c}'$ is
\[ c'_i = N - a_{r-i,r-i} + \sum_{j \geq r-i} a_{r-i+1,j} - a_{r-i,j}.\]
Chasing our bijection between the Gelfand-Tsetlin pattern and the crystal node back through these entries, we may rewrite in terms of $m_{i,j}$:
\begin{align*} c'_i &= N - \sum_{j\geq i} m_{i,j} +  \sum_{k<i}m_{k,i} - (\lambda + \rho)_{r-i+1}\\
&= N - p_i - (\lambda + \rho)_{r-i+1}.
\end{align*}
Using $w_0(\lambda + \rho)_i = (\lambda + \rho)_{r-i+1}$, we then have that $\boldsymbol{c}' = \overline{N} - \boldsymbol{p} - w_0(\lambda +\rho)$. Since $\boldsymbol{p}$ is in the coset $(\g-w_0\lambda)\Lambda$, we have that after we take the residue mod $n_Q$,
\[ \boldsymbol{c} = [\overline{N} - \g - w_0\rho].\]
It remains to then consider the normalization factor between $Z(\frak{S}_\lambda;\boldsymbol{c})$ and $I_{\g,\lambda}$. Due to the $\d(a)$ factor in the weights of Figure \ref{modwts}, we increment a power of $z_i^{-n_Q}$ each time the unmodulated charge $c'_i$ on the eastern position of a vertex is divisible by $n_Q$. Since these all must occur on interior horizontal edges, we then see that the total power of $\boldsymbol{z}$ in $Z(\frak{S}_\lambda;\boldsymbol{c})$ is precisely $z^{-\boldsymbol{c}' +\boldsymbol{c}}$. Here is where our convention becomes crucial: if the charge increments to a multiple of $n_Q$ on the left-most boundary (i.e. $c_i' \equiv 0\pmod{n_Q}$), the interior horizontal edges will not have attained that last multiple of $n_Q$, and therefore we need $c_i$ to be $n_Q$ rather than $0$.
By  (\ref{mainwhitint}) and multiplicativity of the modular quasicharacter, $\mathcal{W}_\g^{\boldsymbol{z}} = \delta^{1/2}(\boldsymbol{z}^\lambda)\boldsymbol{z}^{w_0\lambda} \cdot I_{\g,\lambda}$. Then, since $-\boldsymbol{c}' + \boldsymbol{c} = -\overline{N} + w_0(\lambda + \rho) + \boldsymbol{p} + \boldsymbol{c}$, we have that
\begin{align*} Z(\frak{S}_\lambda; \boldsymbol{c}) = \boldsymbol{z}^{-\overline{N} + w_0\rho  + \boldsymbol{c}} \cdot \d^{-1/2}(\varpi^\lambda)\mathcal{W}_\gamma^{\boldsymbol{z}}(\varpi^\lambda). &\qedhere
\end{align*}
\end{proof}

\begin{Remark}\label{choosy}
For many metaplectic covers, Theorem \ref{indivWhit} is completely independent of the choice of representatives $\Gamma$: for these covers, the set $\Lambda \subset (n_Q\Z)^r$, and thus all representatives of a single coset of $\Gamma$ are equivalent modulo $n_Q$, so all representatives will map to the same charge. However, in cases where $\Lambda$ is not wholly contained in $(n_Q\Z)^r$, i.e. when $r, n, b, c,$ and $c-b$ have sufficiently many factors in common, we must carefully choose our coset representatives in $\Gamma$ to respect the role played by $\boldsymbol{p}$ in the proof above, specifically in the map to $\boldsymbol{c}$. Note that all all $\boldsymbol{p}$ have $\varpi^{\boldsymbol{p}} \in \twid{SL}_r(F)$ by construction, and the subset of $\Lambda$ contained in $\twid{SL}_r(F)$ is always wholly contained in $(n_Q\Z)^r$, so any $\boldsymbol{p}$ contributing to a given Whittaker function will have the same residue class modulo $n_Q$. For any Whittaker function that is nonzero on $\varpi^\lambda$, we thus require the representative $\gamma$ to satisfy $\gamma - w_0\lambda \equiv \boldsymbol{p} \pmod{n_Q}$ for any $\boldsymbol{p}$ appearing in this Whittaker function. To achieve this end, it suffices to choose representatives $\gamma$ such that $\varpi^{\gamma - w_0\lambda} \in SL_r(F)$ for these Whittaker functions, and representatives not equivalent to these modulo $n_Q$ for all remaining cosets.
\end{Remark}

\begin{Example}
For instance, our calculation in Example \ref{section8ex} applies to any metaplectic cover with $n_Q = 2$, but for certain choices, we may need a different representative for the coset $(4,3,3)\Lambda$. Consider $n=6$ and $r=3$. On the one hand, the cover corresponding to $B_{3,4}$, i.e. $b=3,c=4$, has $\Lambda \subset (2\Z)^3$, so every coset representative for the coset $(4,3,3)\Lambda$ will map to the same charge $\boldsymbol{c}$ (in the Example above, when $N=5$, this is the charge $\boldsymbol{c} = (1,1,2)$). On the other hand, the cover corresponding to $B_{3,5}$ has $(1,1,1) \in \Lambda$ and thus each coset splits into two residue classes modulo $2$. For the coset $(4,3,3)\Lambda$, we need to pick a representative $\gamma$ such that $\gamma - w_0\lambda \equiv (2,1,-3) \pmod{2}$, where $(2,1,-3)$ is one $\boldsymbol{p}$ for this Whittaker function. We see that coincidentally the representative (4,3,3) still works, but without having already calculated $\boldsymbol{p}$ we might have accidentally chosen $(3,2,2)$ instead, which doesn't work. To dispense with the need to precalculate, we could instead choose the representative $(0,1,3)$ for this coset, since $(0,1,3) - (0,2,2) = (0,-1,1)$.
\end{Example}


\begin{Cor}
Under the precise equivalence of Theorem \ref{indivWhit}, we see that the functional equation on partition functions provided by the Yang-Baxter equation reproduces the functional equation of the scattering matrix. That is, from either source,
\begin{align}\label{train} Z(\frak{S}_{\lambda,s_i\boldsymbol{z}};\boldsymbol{c}) = \emph{wt}\left( \begin{array}{c}\gamgam{+}{+}{+}{+}{a}{b\vphantom{b}}{b\vphantom{b}}{a}\end{array}\right) Z(\frak{S}_{\lambda,\boldsymbol{z}};\boldsymbol{c}) + \emph{wt}\left( \begin{array}{c}\gamgam{+}{+}{+}{+}{a}{b\vphantom{b}}{b\vphantom{b}}{a}\end{array}\right) Z(\frak{S}_{\lambda,\boldsymbol{z}}; s_i\boldsymbol{c}).
\end{align}
\end{Cor}

\begin{proof}
We start by evaluating the partition function of the following system:
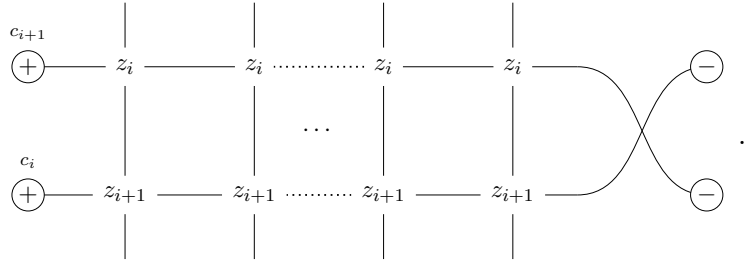
\[
\begin{array}{c}
\scalebox{.85}{\begin{tikzpicture}
  \draw (-0.5,0)--(3,0);
  \draw (-0.5,2)--(3,2);
  \draw [thick, dotted] (3,0)-- (5,0);
  \draw [thick, dotted] (3,2)-- (5,2);
  \draw (5,0) -- (8,0);
  \draw (5,2) -- (8,2);
  \draw (1,-1)--(1,3);
  \draw (3,-1)--(3,3);
  \draw (5,-1)--(5,3);
  \draw (7,-1)--(7,3);
  \coordinate (a1) at (8,0);
  \coordinate (c1) at (10,2);
  \coordinate (a2) at (8,2);
  \coordinate (c2) at (10,0);
  \draw (a1) to [out=0,in=180] (c1);
  \draw (a2) to [out=0,in=180] (c2);
  \draw[fill=white] (-0.5,2) circle (.25);
  \draw[fill=white] (-0.5,0) circle (.25);
  \draw[fill=white] (10,2) circle (.25);
  \draw[fill=white] (10,0) circle (.25);
  \path[fill=white] (1,0) ellipse (.5cm and .3cm);
  \path[fill=white] (1,2) ellipse (.3cm and .3cm);
  \path[fill=white] (3,0) ellipse (.5cm and .3cm);
  \path[fill=white] (3,2) ellipse (.3cm and .3cm);
  \path[fill=white] (5,0) ellipse (.5cm and .3cm);
  \path[fill=white] (5,2) ellipse (.3cm and .3cm);
  \path[fill=white] (7,0) ellipse (.5cm and .3cm);
  \path[fill=white] (7,2) ellipse (.3cm and .3cm);
  \node at (1,0) {$z_{i+1}$};
  \node at (1,2) {$z_i$};
  \node at (3,0) {$z_{i+1}$};
  \node at (3,2) {$z_i$};
  \node at (5,0) {$z_{i+1}$};
  \node at (5,2) {$z_i$};
  \node at (7,0) {$z_{i+1}$};
  \node at (7,2) {$z_i$};
  \node at (-0.5,0.5) {$\scriptstyle c_i$};
  \node at (-0.5,2.5) {$\scriptstyle c_{i+1}$};
  \node at (-0.5,2){$+$};
  \node at (-0.5,0){$+$};
  \node at (10,0){$-$};
  \node at (10,2){$-$};
  \node at (4,1) {$\cdots$};
  \end{tikzpicture}}
  \end{array}.
  \]
We have only one choice of R-vertex (type $\tt{a}_2$), so this partition function gives the left hand side of (\ref{train}). Then, repeatedly applying Theorem \ref{ybe1} to push the R-vertex through to the other side in a ``train argument," this is equal to the partition function of the following lattice:
 \[\begin{array}{c}
 \scalebox{.85}
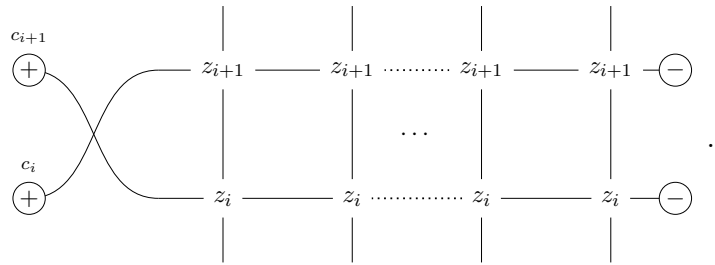
{\begin{tikzpicture}
  \draw (0,0)--(3,0);
  \draw (0,2)--(3,2);
  \draw [thick, dotted] (3,0)-- (5,0);
  \draw [thick, dotted] (3,2)-- (5,2);
  \draw (5,0) -- (8,0);
  \draw (5,2) -- (8,2);
  \draw (1,-1)--(1,3);
  \draw (3,-1)--(3,3);
  \draw (5,-1)--(5,3);
  \draw (7,-1)--(7,3);
  \coordinate (a1) at (-2,0);
  \coordinate (c1) at (0,2);
  \coordinate (a2) at (-2,2);
  \coordinate (c2) at (0,0);
  \draw (a1) to [out=0,in=180] (c1);
  \draw (a2) to [out=0,in=180] (c2);
  \draw[fill=white] (-2,2) circle (.25);
  \draw[fill=white] (-2,0) circle (.25);
  \draw[fill=white] (8,2) circle (.25);
  \draw[fill=white] (8,0) circle (.25);
  \path[fill=white] (1,2) ellipse (.5cm and .3cm);
  \path[fill=white] (1,0) ellipse (.3cm and .3cm);
  \path[fill=white] (3,2) ellipse (.5cm and .3cm);
  \path[fill=white] (3,0) ellipse (.3cm and .3cm);
  \path[fill=white] (5,2) ellipse (.5cm and .3cm);
  \path[fill=white] (5,0) ellipse (.3cm and .3cm);
  \path[fill=white] (7,2) ellipse (.5cm and .3cm);
  \path[fill=white] (7,0) ellipse (.3cm and .3cm);
  \node at (1,2) {$z_{i+1}$};
  \node at (1,0) {$z_i$};
  \node at (3,2) {$z_{i+1}$};
  \node at (3,0) {$z_i$};
  \node at (5,2) {$z_{i+1}$};
  \node at (5,0) {$z_i$};
  \node at (7,2) {$z_{i+1}$};
  \node at (7,0) {$z_i$};
  \node at (-2,2){$+$};
  \node at (-2,0){$+$};
  \node at (8,0){$-$};
  \node at (8,2){$-$};
  \node at (-2,0.5) {$\scriptstyle c_i$};
  \node at (-2,2.5) {$\scriptstyle c_{i+1}$};
  \node at (4, 1) {$\cdots$};
  \end{tikzpicture}}
  \end{array}.
  \]
Here, we have two choices for the R-vertex: either $c_i$ and $c_{i+1}$ do not swap, which yields the first term on the right hand side of (\ref{train}), or they do swap, which yields the second term.  

Alternately, we have $\mathcal{W}_\g^{s_i\boldsymbol{z}}\circ \overline{\mathcal{A}_{s_i}} = \tau^1_{\g,\g}\mathcal{W}_\g^{\boldsymbol{z}} + \tau^2_{\g,s_i\g +\a_i} \mathcal{W}_{s_i\g+\a_i}^{\boldsymbol{z}}$. Note that $\overline{N} - w_0\rho = \rho + \overline{k}$ for some $k$, so we have $\g =\rho- \boldsymbol{c}+\overline{k}$, where the $k$ shift accounts for the choice of $N$. Then $s_i\g+\a_i^\vee$ corresponds to $s_i\boldsymbol{c}$ under the same shift. Since
\[ \boldsymbol{z}^{\overline{N} - w_0\rho - \boldsymbol{c}}(s_i\boldsymbol{z})^{-\overline{N}+w_0\rho+c} = \boldsymbol{z}^{(-c_i+c_{i+1}+1)\a_i} \hspace{1cm} \text{ and } \hspace{1cm}\boldsymbol{z}^{\overline{N} - w_0\rho - s_i\boldsymbol{c}}(s_i\boldsymbol{z})^{-\overline{N}+w_0\rho+c} =\boldsymbol{z}^{\a_i},\]
when we plug in from Theorem \ref{indivWhit} the factors appearing from the powers of $\boldsymbol{z}$ out front will cancel with the factors of $\boldsymbol{z}$ appearing in Proposition \ref{tauR}, leaving just the R-vertex weights and the desired functional equation.
\end{proof}

\section{Appendix} \label{appendix}

These are the complete calculations for the proof of Theorem \ref{ybe1}. For readability, we omit the denominators in the $R$-weights. Thus, the actual values of these calculations are the printed value divided by $1-v\mathbf{z}^{n_Q\a}$. The orientation of tables follows that in Example \ref{nicecase}, where the central table is the system of boundary conditions, and the second set of tables are the states for that system on the left and right sides of Equation (\ref{mybe}), respectively. In cases where the second set of tables is stacked, the top one is the left hand side and the bottom the right hand side.

\textbf{Case 1a:}
\begin{center}
\begin{tabular}{|c|c|c|c|c|c|}\hline$k+1$&$k+1$&&$k$&$k$&\\+&+&+&+&+&+ \\\hline\end{tabular}
\end{center}
\begin{center}
\begin{tabular}{|c|c|c|c|}
\hline\small{$k+1$}&\small{$k+1$}&&\multirow{2}{*}{\small{$(z_iz_j)^{-n_Q\d(k)}\left(-v+\zz^{n_Q\a}\right)$}}\\+&+&+&\\\hline
\end{tabular}
\hspace{0.5cm}
\begin{tabular}{|c|c|c|c|}
\hline\small{$k+1$}&\small{$k+1$}&&\multirow{2}{*}{\small{$(z_iz_j)^{-n_Q\d(k)}\left(-v+\zz^{n_Q\a}\right)$}}\\+&+&+&\\\hline
\end{tabular}
\end{center}

\textbf{Case 1b:} $k\not\equiv \ell \pmod{n_Q}$.
\begin{center}
\begin{tabular}{|c|c|c|c|c|c|}\hline$\ell+1$&$k+1$&&$k$&$\ell$&\\+&+&+&+&+&+ \\\hline\end{tabular}
\end{center}
\begin{center}
\begin{tabular}{|c|c|c|c|}
\hline\small{$k+1$}&\small{$\ell+1$}&&\multirow{2}{*}{\small{$(1-v)z_i^{-n_Q\d(\ell)}z_j^{-n_Q\d(k)}\cdot\begin{cases}\zz^{n_Q\a} & k+1> \ell+1\\1&k+1<\ell+1\end{cases}$}}\\+&+&+&\\\hline
\end{tabular}

\begin{tabular}{|c|c|c|c|}
\hline\small{$k$}&\small{$\ell$}&&\multirow{2}{*}{\small{$(1-v)z_i^{-n_Q\d(k)}z_j^{-n_Q\d(\ell)}\cdot\begin{cases}\zz^{n_Q\a} & k>\ell\\1&k<\ell\end{cases}$}}\\+&+&+&\\\hline
\end{tabular}
\end{center}
Note: if $k\equiv 0\pmod{n_Q}$ but $\ell\not\equiv 0\pmod{n_Q}$, then $k >\ell \pmod{n_Q}$ but $k+1<\ell+1\pmod{n_Q}$, so the cancellation does work out properly, and vice versa.

\textbf{Case 1c:} $k\not\equiv \ell \pmod{n_Q}$.
\begin{center}
\begin{tabular}{|c|c|c|c|c|c|}\hline$k+1$&$\ell+1$&&$k$&$\ell$&\\+&+&+&+&+&+ \\\hline\end{tabular}
\end{center}
\begin{center}
\begin{tabular}{|c|c|c|c|}
\hline\small{$k+1$}&\small{$\ell+1$}&&\multirow{2}{*}{\small{$g_Q(\ell-k)z_i^{-n_Q\d(\ell)}z_j^{-n_Q\d(a)}\left(1-\zn\right)$}}\\+&+&+&\\\hline
\end{tabular}

\begin{tabular}{|c|c|c|c|}
\hline\small{$k$}&\small{$\ell$}&&\multirow{2}{*}{\small{$g_Q(\ell-k)z_i^{-n_Q\d(\ell)}z_j^{-n_Q\d(k)}\left(1-\zn\right)$}}\\+&+&+&\\\hline
\end{tabular}
\end{center}

\textbf{Case 2:} $(+,+,+,-,-,+)$. There are no admissible configurations in this case.\\

\textbf{Case 3a:} $k\not\equiv 0\pmod{n_Q}$.

\vspace{-0.7cm}
\begin{center}
\begin{tabular}{|c|c|c|c|c|c|}\hline$k+1$&$0$&&$0$&$k$&\\+&$-$&+&$-$&+&+ \\\hline\end{tabular}
\end{center}
\begin{center}
\begin{tabular}{|c|c|c|c|}
\hline\small{0}&\small{$k+1$}&&\multirow{2}{*}{\small{$1-v$}}\\$-$&+&+&\\\hline
\end{tabular}
\hspace{0.5cm}
\begin{tabular}{|c|c|c|c|}
\hline\small{0}&\small{$k$}&&\multirow{2}{*}{\small{$1-v$}}\\$-$&+&+&\\\hline
\end{tabular}
\end{center}

\textbf{Case 3b:} 

\vspace{-0.7cm}
\begin{center}
\begin{tabular}{|c|c|c|c|c|c|}\hline$1$&$0$&&$0$&$0$&\\+&$-$&+&$-$&+&+ \\\hline\end{tabular}
\end{center}
\begin{center}
\begin{tabular}{|c|c|c|c|}
\hline\small{0}&\small{1}&&\multirow{2}{*}{\small{$(1-v)z_i^{-n_Q}$}}\\$-$&+&+&\\\hline
\end{tabular}
\hspace{0.5cm}
\begin{tabular}{|c|c|c|c|}
\hline\small{0}&\small{$0$}&&\multirow{2}{*}{\small{$(1-v)z_i^{-n_Q}(1-\zz^{n_Q\a})$}}\\+&$-$&$-$&\\\hline\small{0}&\small{$0$}&&\multirow{2}{*}{\small{$(1-v)z_j^{-n_Q}$}}\\$-$&+&+&\\\hline
\end{tabular}
\end{center}

\textbf{Case 4a:} $k\nequiv 0\pmod{n_Q}$.

\vspace{-0.7cm}
\begin{center}
\begin{tabular}{|c|c|c|c|c|c|}\hline$k+1$&$0$&&$k$&$0$&\\+&$-$&+&+&$-$&+ \\\hline\end{tabular}
\end{center}
\begin{center}
\begin{tabular}{|c|c|c|c|}
\hline\small{$k+1$}&\small{0}&&\multirow{2}{*}{\small{$v(1-\zn)$}}\\+&$-$&+&\\\hline
\end{tabular}
\hspace{0.5cm}
\begin{tabular}{|c|c|c|c|}
\hline\small{0}&\small{$k$}&&\multirow{2}{*}{\small{$v(1-\zz^{n_Q\a})$}}\\$-$&+&+&\\\hline\end{tabular}
\end{center}

\textbf{Case 4b:}

\vspace{-0.7cm}
\begin{center}
\begin{tabular}{|c|c|c|c|c|c|}\hline$1$&$0$&&$0$&$0$&\\+&$-$&+&+&$-$&+ \\\hline\end{tabular}
\end{center}
\begin{center}
\begin{tabular}{|c|c|c|c|}
\hline\small{1}&\small{0}&&\multirow{2}{*}{\small{$vz_j^{-n_Q}(1-\zn)$}}\\+&$-$&+&\\
\hline\small{0}&\small{1}&&\multirow{2}{*}{\small{$(1-v)^2z_j^{-n_Q}$}}\\$-$&+&$-$&\\\hline
\end{tabular}
\hspace{0.5cm}
\begin{tabular}{|c|c|c|c|}
\hline\small{0}&\small{$0$}&&\multirow{2}{*}{\small{$(1-v)^2z_i^{-n_Q}\zn$}}\\+&$-$&$-$&\\\hline\small{0}&\small{$0$}&&\multirow{2}{*}{\small{$vz_j^{-n_Q}(1-\zn)$}}\\$-$&+&+&\\\hline
\end{tabular}
\end{center}

\textbf{Case 5:} 

\vspace{-0.7cm}
\begin{center}
\begin{tabular}{|c|c|c|c|c|c|}\hline$0$&$k+1$&&$0$&$k$&\\$-$&+&+&$-$&+&+ \\\hline\end{tabular}
\end{center}
\begin{center}
\begin{tabular}{|c|c|c|c|}
\hline\small{0}&\small{$k+1$}&&\multirow{2}{*}{\small{$z_i^{-n_Q\d(k)}(1-\zn)$}}\\$-$&+&+&\\\hline
\end{tabular}
\hspace{0.5cm}
\begin{tabular}{|c|c|c|c|}
\hline\small{$k$}&\small{0}&&\multirow{2}{*}{\small{$z_i^{-n_Q\d(k)}(1-\zn)$}}\\+&$-$&+&\\\hline\end{tabular}
\end{center}

\textbf{Case 6a:} $k\nequiv 0\pmod{n_Q}$.

\vspace{-0.7cm}
\begin{center}
\begin{tabular}{|c|c|c|c|c|c|}\hline$0$&$k+1$&&$k$&$0$&\\$-$&+&+&+&$-$&+ \\\hline\end{tabular}
\end{center}
\begin{center}
\begin{tabular}{|c|c|c|c|}
\hline\small{$k+1$}&\small{0}&&\multirow{2}{*}{\small{$(1-v)\zn$}}\\+&$-$&+&\\\hline
\end{tabular}
\hspace{0.5cm}
\begin{tabular}{|c|c|c|c|}
\hline\small{$k$}&\small{0}&&\multirow{2}{*}{\small{$(1-v)\zn$}}\\+&$-$&+&\\\hline\end{tabular}
\end{center}

\textbf{Case 6b:}

\vspace{-0.7cm}
\begin{center}
\begin{tabular}{|c|c|c|c|c|c|}\hline$0$&$1$&&$0$&$0$&\\$-$&+&+&+&$-$&+ \\\hline\end{tabular}
\end{center}
\begin{center}
\begin{tabular}{|c|c|c|c|}
\hline\small{$1$}&\small{0}&&\multirow{2}{*}{\small{$(1-v)z_j^{-n_Q}\zn$}}\\+&$-$&+&\\\hline\small{$0$}&\small{$1$}&&\multirow{2}{*}{\small{$(1-v)z_j^{-n_Q}(1-\zn)$}}\\$-$&+&$-$&\\\hline
\end{tabular}
\hspace{0.5cm}
\begin{tabular}{|c|c|c|c|}
\hline\small{0}&\small{$0$}&&\multirow{2}{*}{\small{$(1-v)z_i^{-n_Q}\zn$}}\\+&$-$&+&\\\hline\end{tabular}
\end{center}

\textbf{Case 7:} $(-,-,+,+,+,+)$. There are no admissible configurations in this case.\\

\textbf{Case 8:}

\vspace{-0.7cm}
\begin{center}
\begin{tabular}{|c|c|c|c|c|c|}\hline$0$&$0$&&$0$&$0$&\\$-$&$-$&+&$-$&$-$&+ \\\hline\end{tabular}
\end{center}
\begin{center}
\begin{tabular}{|c|c|c|c|}
\hline\small{$0$}&\small{0}&&\multirow{2}{*}{\small{$1-v\zn$}}\\$-$&$-$&+&\\\hline
\end{tabular}
\hspace{0.5cm}
\begin{tabular}{|c|c|c|c|}
\hline\small{$0$}&\small{0}&&\multirow{2}{*}{\small{$1-v\zn$}}\\$-$&$-$&+&\\\hline\end{tabular}
\end{center}

\textbf{Case 9a:} $k\nequiv 0\pmod{n_Q}$.

\vspace{-0.7cm}
\begin{center}
\begin{tabular}{|c|c|c|c|c|c|}\hline$1$&$k+1$&&$0$&$k$&\\+&+&$-$&$-$&+&+ \\\hline\end{tabular}
\end{center}
\begin{center}
\begin{tabular}{|c|c|c|c|}
\hline\small{$1$}&\small{$k+1$}&&\multirow{2}{*}{\small{$g_Q(k)(1-\zn)$}}\\+&+&+&\\\hline
\end{tabular}
\hspace{0.5cm}
\begin{tabular}{|c|c|c|c|}
\hline\small{$k$}&\small{0}&&\multirow{2}{*}{\small{$g_Q(k)(1-\zn)$}}\\+&$-$&$-$&\\\hline\end{tabular}
\end{center}

\textbf{Case 9b:} $k\nequiv 0\pmod{n_Q}$.

\vspace{-0.7cm}
\begin{center}
\begin{tabular}{|c|c|c|c|c|c|}\hline$k+1$&$1$&&$0$&$k$&\\+&+&$-$&$-$&+&+ \\\hline\end{tabular}
\end{center}
\begin{center}
\begin{tabular}{|c|c|c|c|}
\hline\small{$1$}&\small{$k+1$}&&\multirow{2}{*}{\small{$1-v$}}\\+&+&+&\\\hline
\end{tabular}
\hspace{0.5cm}
\begin{tabular}{|c|c|c|c|}
\hline\small{$0$}&\small{$k$}&&\multirow{2}{*}{\small{$1-v$}}\\$-$&+&+&\\\hline\end{tabular}
\end{center}

Note that $k+1 > 1$ here; if $k+1 \equiv 0\pmod{n_Q}$, we treat it as $n_Q$, not as $0$, although we write it as $0$.

\textbf{Case 9c:}

\vspace{-0.7cm}
\begin{center}
\begin{tabular}{|c|c|c|c|c|c|}\hline$1$&$1$&&$0$&$0$&\\+&+&$-$&$-$&+&+ \\\hline\end{tabular}
\end{center}
\begin{center}
\begin{tabular}{|c|c|c|c|}
\hline\small{$1$}&\small{$1$}&&\multirow{2}{*}{\small{$z_i^{-n_Q}(-v+\zn)$}}\\+&+&+&\\\hline
\end{tabular}
\hspace{0.5cm}
\begin{tabular}{|c|c|c|c|}
\hline\small{$0$}&\small{0}&&\multirow{2}{*}{\small{$-vz_i^{-n_Q}(1-\zn)$}}\\+&$-$&$-$&\\\hline\small{$0$}&\small{0}&&\multirow{2}{*}{\small{$(1-v)z_j^{-n_Q}$}}\\$-$&+&+&\\\hline
\end{tabular}
\end{center}

\textbf{Case 10a:} $k\nequiv 0\pmod{n_Q}$.

\vspace{-0.7cm}
\begin{center}
\begin{tabular}{|c|c|c|c|c|c|}\hline$1$&$k+1$&&$k$&$0$&\\+&+&$-$&+&$-$&+ \\\hline\end{tabular}
\end{center}
\begin{center}
\begin{tabular}{|c|c|c|c|}
\hline\small{$k+1$}&\small{$1$}&&\multirow{2}{*}{\small{$g_Q(k)(1-v)\zn$}}\\+&+&$-$&\\\hline
\end{tabular}
\hspace{0.5cm}
\begin{tabular}{|c|c|c|c|}
\hline\small{$k$}&\small{0}&&\multirow{2}{*}{\small{$g_Q(k)(1-v)\zn$}}\\+&$-$&$-$&\\\hline\end{tabular}
\end{center}

\textbf{Case 10b:} $k\nequiv 0\pmod{n_Q}$.

\vspace{-0.7cm}
\begin{center}
\begin{tabular}{|c|c|c|c|c|c|}\hline$k+1$&$1$&&$k$&$0$&\\+&+&$-$&+&$-$&+ \\\hline\end{tabular}
\end{center}
\begin{center}
\begin{tabular}{|c|c|c|c|}
\hline\small{$k+1$}&\small{$1$}&&\multirow{2}{*}{\small{$g_Q(-k)g_Q(k)(1-\zn)$}}\\+&+&$-$&\\\hline
\end{tabular}
\hspace{0.5cm}
\begin{tabular}{|c|c|c|c|}
\hline\small{$0$}&\small{$k$}&&\multirow{2}{*}{\small{$v(1-\zn)$}}\\$-$&+&+&\\\hline\end{tabular}
\end{center}

\textbf{Case 10c:}

\vspace{-0.7cm}
\begin{center}
\begin{tabular}{|c|c|c|c|c|c|}\hline$1$&$1$&&$0$&$0$&\\+&+&$-$&+&$-$&+ \\\hline\end{tabular}
\end{center}
\begin{center}
\begin{tabular}{|c|c|c|c|}
\hline\small{$1$}&\small{$1$}&&\multirow{2}{*}{\small{$-vz_j^{-n_Q}(-v+\zn)$}}\\+&+&$-$&\\\hline
\end{tabular}
\hspace{0.5cm}
\begin{tabular}{|c|c|c|c|}
\hline\small{$0$}&\small{0}&&\multirow{2}{*}{\small{$-v(1-v)z_i^{-n_Q}\zn$}}\\+&$-$&$-$&\\\hline\small{$0$}&\small{0}&&\multirow{2}{*}{\small{$vz_j^{-n_Q}(1-\zn)$}}\\$-$&+&+&\\\hline
\end{tabular}
\end{center}

\textbf{Case 11:} $(+,-,-,+,+,+)$. There are no admissible configurations in this case.\\

\textbf{Case 12:}

\vspace{-0.7cm}
\begin{center}
\begin{tabular}{|c|c|c|c|c|c|}\hline$1$&$0$&&$0$&$0$&\\+&$-$&$-$&$-$&$-$&+ \\\hline\end{tabular}
\end{center}
\begin{center}
\begin{tabular}{|c|c|c|c|}
\hline\small{$1$}&\small{$0$}&&\multirow{2}{*}{\small{$v(1-\zn)$}}\\+&$-$&+&\\\hline\small{$0$}&\small{$1$}&&\multirow{2}{*}{\small{$1-v$}}\\$-$&+&$-$&\\\hline
\end{tabular}
\hspace{0.5cm}
\begin{tabular}{|c|c|c|c|}
\hline\small{$0$}&\small{0}&&\multirow{2}{*}{\small{$1-v\zn$}}\\$-$&$-$&$-$&\\\hline\end{tabular}
\end{center}

\textbf{Case 13:} $(-,+,-,+,+,+)$. There are no admissible configurations in this case.\\

\textbf{Case 14:}

\vspace{-0.7cm}
\begin{center}
\begin{tabular}{|c|c|c|c|c|c|}\hline$0$&$1$&&$0$&$0$&\\$-$&+&$-$&$-$&$-$&+ \\\hline\end{tabular}
\end{center}
\begin{center}
\begin{tabular}{|c|c|c|c|}
\hline\small{$1$}&\small{$0$}&&\multirow{2}{*}{\small{$(1-v)\zn$}}\\+&$-$&+&\\\hline\small{$0$}&\small{$1$}&&\multirow{2}{*}{\small{$1-\zn$}}\\$-$&+&$-$&\\\hline
\end{tabular}
\hspace{0.5cm}
\begin{tabular}{|c|c|c|c|}
\hline\small{$0$}&\small{0}&&\multirow{2}{*}{\small{$1-v\zn$}}\\$-$&$-$&+&\\\hline\end{tabular}
\end{center}

\textbf{Case 15:} $(-,-,-,-,+,+)$. There are no admissible configurations in this case.\\

\textbf{Case 16:} $(-,-,-,+,-,+)$. There are no admissible configurations in this case.\\

\textbf{Case 17:} $(+,+,+,-,+,-)$. There are no admissible configurations in this case.\\

\textbf{Case 18:} $(+,+,+,+,-,-)$. There are no admissible configurations in this case.\\

\textbf{Case 19a:} $k\nequiv 0\pmod{n_Q}$.

\vspace{-0.7cm}
\begin{center}
\begin{tabular}{|c|c|c|c|c|c|}\hline$k+1$&$0$&&$0$&$k$&\\+&$-$&+&+&+&$-$ \\\hline\end{tabular}
\end{center}
\begin{center}
\begin{tabular}{|c|c|c|c|}
\hline\small{$0$}&\small{$k+1$}&&\multirow{2}{*}{\small{$g_Q(k)(1-v)^2z_j^{-n_Q}$}}\\$-$&+&$-$&\\\hline
\end{tabular}
\hspace{0.5cm}
\begin{tabular}{|c|c|c|c|}
\hline\small{$0$}&\small{$k$}&&\multirow{2}{*}{\small{$g_Q(k)(1-v)^2z_i^{-n_Q}\zn$}}\\$+$&$+$&$-$&\\\hline\end{tabular}
\end{center}

\textbf{Case 19b:} $k\nequiv 0\pmod{n_Q}$.

\vspace{-0.7cm}
\begin{center}
\begin{tabular}{|c|c|c|c|c|c|}\hline$k+1$&$0$&&$k$&$0$&\\+&$-$&+&+&+&$-$ \\\hline\end{tabular}
\end{center}
\begin{center}
\begin{tabular}{|c|c|c|c|}
\hline\small{$k+1$}&\small{$0$}&&\multirow{2}{*}{\small{$v(1-v)z_i^{-n_Q}(1-\zn)$}}\\$+$&$-$&$+$&\\\hline
\end{tabular}
\hspace{0.5cm}
\begin{tabular}{|c|c|c|c|}
\hline\small{$0$}&\small{$k$}&&\multirow{2}{*}{\small{$g_Q(k)g_Q(-k)(1-v)z_i^{-n_Q}(1-\zn)$}}\\$+$&$+$&$-$&\\\hline\end{tabular}
\end{center}

\textbf{Case 19c:} 

\vspace{-0.7cm}
\begin{center}
\begin{tabular}{|c|c|c|c|c|c|}\hline$1$&$0$&&$0$&$0$&\\+&$-$&+&+&+&$-$ \\\hline\end{tabular}
\end{center}
\begin{center}
\begin{tabular}{|c|c|c|c|}
\hline\small{$0$}&\small{$1$}&&\multirow{2}{*}{\small{$-v(1-v)^2(z_iz_j)^{-n_Q}$}}\\$-$&+&$-$&\\\hline\small{$1$}&\small{$0$}&&\multirow{2}{*}{\small{$v(1-v)(z_iz_j)^{-n_Q}(1-\zn)$}}\\$+$&$-$&$+$&\\\hline
\end{tabular}
\hspace{0.5cm}
\begin{tabular}{|c|c|c|c|}
\hline\small{$0$}&\small{$0$}&&\multirow{2}{*}{\small{$-v(1-v)(z_iz_j)^{-n_Q}(-v+\zn)$}}\\$+$&$+$&$-$&\\\hline\end{tabular}
\end{center}

\textbf{Case 20:} $(+,-,+,-,-,-)$. There are no admissible configurations in this case.\\

\textbf{Case 21a:} $k\nequiv0\pmod{n_Q}$.

\vspace{-0.7cm}
\begin{center}
\begin{tabular}{|c|c|c|c|c|c|}\hline$0$&$k+1$&&$k$&$0$&\\$-$&$+$&+&+&+&$-$ \\\hline\end{tabular}
\end{center}
\begin{center}
\begin{tabular}{|c|c|c|c|}
\hline\small{$k+1$}&\small{$0$}&&\multirow{2}{*}{\small{$(1-v)^2z_i^{-n_Q}\zn$}}\\+&$-$&+&\\\hline
\end{tabular}
\hspace{0.5cm}
\begin{tabular}{|c|c|c|c|}
\hline\small{$k$}&\small{$0$}&&\multirow{2}{*}{\small{$(1-v)^2z_j^{-n_Q}$}}\\$+$&$+$&$+$&\\\hline\end{tabular}
\end{center}

Note that $k<0$ here in the $R$-matrix on the right, because we treat $+0$ as $+n_Q$.

\textbf{Case 21b:} $k\nequiv 0\pmod{n_Q}$.

\vspace{-0.7cm}
\begin{center}
\begin{tabular}{|c|c|c|c|c|c|}\hline$0$&$k+1$&&$0$&$k$&\\$-$&$+$&+&+&+&$-$ \\\hline\end{tabular}
\end{center}
\begin{center}
\begin{tabular}{|c|c|c|c|}
\hline\small{$0$}&\small{$k+1$}&&\multirow{2}{*}{\small{$g_Q(k)(1-v)z_j^{-n_Q}(1-\zn)$}}\\$-$&$+$&$-$&\\\hline
\end{tabular}
\hspace{0.5cm}
\begin{tabular}{|c|c|c|c|}
\hline\small{$k$}&\small{$0$}&&\multirow{2}{*}{\small{$g_Q(k)(1-v)z_j^{-n_Q}(1-\zn)$}}\\$+$&$+$&$+$&\\\hline\end{tabular}
\end{center}

\textbf{Case 21c:} 

\vspace{-0.7cm}
\begin{center}
\begin{tabular}{|c|c|c|c|c|c|}\hline$0$&$1$&&$0$&$0$&\\$-$&$+$&+&+&+&$-$ \\\hline\end{tabular}
\end{center}
\begin{center}
\begin{tabular}{|c|c|c|c|}
\hline\small{$0$}&\small{$1$}&&\multirow{2}{*}{\small{$-v(1-v)(z_iz_j)^{-n_Q}(1-\zn)$}}\\$-$&$+$&$-$&\\\hline\small{$1$}&\small{$0$}&&\multirow{2}{*}{\small{$(1-v)^2(z_iz_j)^{-n_Q}\zn$}}\\$+$&$-$&$+$&\\\hline
\end{tabular}
\hspace{0.5cm}
\begin{tabular}{|c|c|c|c|}
\hline\small{$0$}&\small{$0$}&&\multirow{2}{*}{\small{$(1-v)(z_iz_j)^{-n_Q}(-v+\zn)$}}\\$+$&$+$&$+$&\\\hline\end{tabular}
\end{center}

\textbf{Case 22:} $(-,+,+,-,-,-)$. There are no admissible configurations in this case.\\

\textbf{Case 23:}

\vspace{-0.7cm}
\begin{center}
\begin{tabular}{|c|c|c|c|c|c|}\hline$0$&$0$&&$0$&$0$&\\$-$&$-$&+&$-$&+&$-$ \\\hline\end{tabular}
\end{center}
\begin{center}
\begin{tabular}{|c|c|c|c|}
\hline\small{$0$}&\small{$0$}&&\multirow{2}{*}{\small{$(1-v)z_i^{-n_Q}(1-v\zn)$}}\\$-$&$-$&$+$&\\\hline
\end{tabular}
\hspace{0.5cm}
\begin{tabular}{|c|c|c|c|}
\hline\small{$0$}&\small{$0$}&&\multirow{2}{*}{\small{$(1-v)^2z_j^{-n_Q}$}}\\$-$&$+$&$+$&\\\hline\small{$0$}&\small{$0$}&&\multirow{2}{*}{\small{$(1-v)z_i^{-n_Q}(1-\zn)$}}\\$+$&$-$&$-$&\\\hline
\end{tabular}
\end{center}

\textbf{Case 24:}

\vspace{-0.7cm}
\begin{center}
\begin{tabular}{|c|c|c|c|c|c|}\hline$0$&$0$&&$0$&$0$&\\$-$&$-$&+&$+$&$-$&$-$ \\\hline\end{tabular}
\end{center}
\begin{center}
\begin{tabular}{|c|c|c|c|}
\hline\small{$0$}&\small{$0$}&&\multirow{2}{*}{\small{$(1-v)z_j^{-n_Q}(1-v\zn)$}}\\$-$&$-$&$-$&\\\hline
\end{tabular}
\hspace{0.5cm}
\begin{tabular}{|c|c|c|c|}
\hline\small{$0$}&\small{$0$}&&\multirow{2}{*}{\small{$v(1-v)z_j^{-n_Q}(1-\zn)$}}\\$-$&$+$&$+$&\\\hline\small{$0$}&\small{$0$}&&\multirow{2}{*}{\small{$(1-v)^2z_i^{-n_Q}\zn$}}\\$+$&$-$&$-$&\\\hline
\end{tabular}
\end{center}

\textbf{Case 25a:}

\vspace{-0.7cm}
\begin{center}
\begin{tabular}{|c|c|c|c|c|c|}\hline$k+1$&$k+1$&&$k$&$k$&\\$+$&$+$&$-$&$+$&+&$-$ \\\hline\end{tabular}
\end{center}
\begin{center}
\begin{tabular}{|c|c|c|c|}
\hline\small{$k+1$}&\small{$k+1$}&&\multirow{2}{*}{\small{$g_Q(k)^2(z_iz_j)^{-n_Q\d(k)}(-v+\zn)$}}\\$+$&$+$&$-$&\\\hline
\end{tabular}
\hspace{0.5cm}
\begin{tabular}{|c|c|c|c|}
\hline\small{$k$}&\small{$k$}&&\multirow{2}{*}{\small{$g_Q(k)^2(z_iz_j)^{-n_Q\d(k)}(-v+\zn)$}}\\$+$&$+$&$-$&\\\hline
\end{tabular}
\end{center}

\textbf{Case 25b:} $k\nequiv \ell\pmod{n_Q}$.

\vspace{-0.7cm}
\begin{center}
\begin{tabular}{|c|c|c|c|c|c|}\hline$k+1$&$\ell+1$&&$\ell$&$k$&\\$+$&$+$&$-$&$+$&+&$-$ \\\hline\end{tabular}
\end{center}
\begin{center}
\begin{tabular}{|c|c|c|c|}
\hline\small{$\ell+1$}&\small{$k+1$}&&\multirow{2}{*}{\small{$g_Q(k)g_Q(\ell)(1-v)z_i^{-n_Q\d(k)}z_j^{-n_Q\d(\ell)}\cdot\begin{cases}\zn &\ell+1>k+1\\1 &\ell+1<k+1 \end{cases}$}}\\$+$&$+$&$-$&\\\hline
\end{tabular}
\hspace{0.5cm}
\begin{tabular}{|c|c|c|c|}
\hline\small{$\ell$}&\small{$k$}&&\multirow{2}{*}{\small{$g_Q(k)g_Q(\ell)(1-v)z_j^{-n_Q\d(k)}z_i^{-n_Q\d(\ell)}\cdot\begin{cases}\zn &\ell>k\\1 &\ell<k \end{cases}$}}\\$+$&$+$&$-$&\\\hline
\end{tabular}
\end{center}

\textbf{Case 25c:} $k\nequiv \ell\pmod{n_Q}$.

\vspace{-0.7cm}
\begin{center}
\begin{tabular}{|c|c|c|c|c|c|}\hline$k+1$&$\ell+1$&&$k$&$\ell$&\\$+$&$+$&$-$&$+$&+&$-$ \\\hline\end{tabular}
\end{center}
\begin{center}
\begin{tabular}{|c|c|c|c|}
\hline\small{$k+1$}&\small{$\ell+1$}&&\multirow{2}{*}{\small{$g_Q(k)g_Q(\ell)g_Q(\ell-k)z_j^{-n_Q\d(k)}z_i^{-n_Q\d(\ell)}(1-\zn)$}}\\$+$&$+$&$-$&\\\hline
\end{tabular}
\hspace{0.5cm}
\begin{tabular}{|c|c|c|c|}
\hline\small{$\ell$}&\small{$k$}&&\multirow{2}{*}{\small{$g_Q(k)g_Q(\ell)g_Q(\ell-k)z_j^{-n_Q\d(k)}z_i^{-n_Q\d(\ell)}(1-\zn)$}}\\$+$&$+$&$-$&\\\hline
\end{tabular}
\end{center}

\textbf{Case 26:} $(+,+,-,-,-,-)$. There are no admissible configurations in this case.\\

\textbf{Case 27a:} $k\nequiv 0\pmod{n_Q}$.

\vspace{-0.7cm}
\begin{center}
\begin{tabular}{|c|c|c|c|c|c|}\hline$k+1$&$0$&&$0$&$k$&\\$+$&$-$&$-$&$-$&+&$-$ \\\hline\end{tabular}
\end{center}
\begin{center}
\begin{tabular}{|c|c|c|c|}
\hline\small{$0$}&\small{$k+1$}&&\multirow{2}{*}{\small{$g_Q(k)(1-v)$}}\\$-$&$+$&$-$&\\\hline
\end{tabular}
\hspace{0.5cm}
\begin{tabular}{|c|c|c|c|}
\hline\small{$0$}&\small{$k$}&&\multirow{2}{*}{\small{$g_Q(k)(1-v)$}}\\$-$&$+$&$-$&\\\hline\end{tabular}
\end{center}

\textbf{Case 27b:}

\vspace{-0.7cm}
\begin{center}
\begin{tabular}{|c|c|c|c|c|c|}\hline$1$&$0$&&$0$&$0$&\\$+$&$-$&$-$&$-$&+&$-$ \\\hline\end{tabular}
\end{center}
\begin{center}
\begin{tabular}{|c|c|c|c|}
\hline\small{$0$}&\small{$1$}&&\multirow{2}{*}{\small{$-v(1-v)z_i^{-n_Q}$}}\\$-$&$+$&$-$&\\\hline\small{$1$}&\small{$0$}&&\multirow{2}{*}{\small{$v(1-v)z_i^{-n_Q}(1-\zn)$}}\\$+$&$-$&$+$&\\\hline
\end{tabular}
\hspace{0.5cm}
\begin{tabular}{|c|c|c|c|}
\hline\small{$0$}&\small{$0$}&&\multirow{2}{*}{\small{$-v(1-v)z_j^{-n_Q}$}}\\$-$&$+$&$-$&\\\hline\end{tabular}
\end{center}

\textbf{Case 28:}

\vspace{-0.7cm}
\begin{center}
\begin{tabular}{|c|c|c|c|c|c|}\hline$1$&$0$&&$0$&$0$&\\$+$&$-$&$-$&$+$&$-$&$-$ \\\hline\end{tabular}
\end{center}
\begin{center}
\begin{tabular}{|c|c|c|c|}
\hline\small{$1$}&\small{$0$}&&\multirow{2}{*}{\small{$-v^2z_j^{-n_Q}(1-\zn)$}}\\$+$&$-$&$-$&\\\hline
\end{tabular}
\hspace{0.5cm}
\begin{tabular}{|c|c|c|c|}
\hline\small{$0$}&\small{$0$}&&\multirow{2}{*}{\small{$-v^2z_j^{-n_Q}(1-\zn)$}}\\$-$&$+$&$-$&\\\hline\end{tabular}
\end{center}

\textbf{Case 29a:} $k\nequiv 0\pmod{n_Q}$.

\vspace{-0.7cm}
\begin{center}
\begin{tabular}{|c|c|c|c|c|c|}\hline$0$&$k+1$&&$0$&$k$&\\$-$&$+$&$-$&$-$&+&$-$ \\\hline\end{tabular}
\end{center}
\begin{center}
\begin{tabular}{|c|c|c|c|}
\hline\small{$0$}&\small{$k+1$}&&\multirow{2}{*}{\small{$g_Q(k)(1-\zn)$}}\\$-$&$+$&$-$&\\\hline
\end{tabular}
\hspace{0.5cm}
\begin{tabular}{|c|c|c|c|}
\hline\small{$k$}&\small{$0$}&&\multirow{2}{*}{\small{$g_Q(k)(1-\zn)$}}\\$+$&$-$&$-$&\\\hline\end{tabular}
\end{center}

\textbf{Case 29b:}

\vspace{-0.7cm}
\begin{center}
\begin{tabular}{|c|c|c|c|c|c|}\hline$0$&$1$&&$0$&$a$&\\$-$&$+$&$-$&$-$&+&$-$ \\\hline\end{tabular}
\end{center}
\begin{center}
\begin{tabular}{|c|c|c|c|}
\hline\small{$0$}&\small{$1$}&&\multirow{2}{*}{\small{$-vz_i^{-n_Q}(1-\zn)$}}\\$-$&$+$&$-$&\\\hline\small{$1$}&\small{$0$}&&\multirow{2}{*}{\small{$(1-v)^2z_i^{-n_Q}\zn$}}\\$+$&$-$&$+$&\\\hline
\end{tabular}
\hspace{0.5cm}
\begin{tabular}{|c|c|c|c|}
\hline\small{$0$}&\small{$0$}&&\multirow{2}{*}{\small{$-vz_i^{-n_Q}(1-\zn)$}}\\$+$&$-$&$-$&\\\hline\small{$0$}&\small{$0$}&&\multirow{2}{*}{\small{$(1-v)^2z_j^{-n_Q}$}}\\$-$&$+$&$+$&\\\hline
\end{tabular}
\end{center}

\textbf{Case 30a:} $k\nequiv 0\pmod{n_Q}$.

\vspace{-0.7cm}
\begin{center}
\begin{tabular}{|c|c|c|c|c|c|}\hline$0$&$k+1$&&$k$&$0$&\\$-$&$+$&$-$&$+$&$-$&$-$ \\\hline\end{tabular}
\end{center}
\begin{center}
\begin{tabular}{|c|c|c|c|}
\hline\small{$k+1$}&\small{$0$}&&\multirow{2}{*}{\small{$g_Q(k)(1-v)\zn$}}\\$+$&$-$&$-$&\\\hline
\end{tabular}
\hspace{0.5cm}
\begin{tabular}{|c|c|c|c|}
\hline\small{$k$}&\small{$0$}&&\multirow{2}{*}{\small{$g_Q(k)(1-v)\zn$}}\\$+$&$-$&$-$&\\\hline\end{tabular}
\end{center}

\textbf{Case 30b:}

\vspace{-0.7cm}
\begin{center}
\begin{tabular}{|c|c|c|c|c|c|}\hline$0$&$1$&&$0$&$0$&\\$-$&$+$&$-$&$+$&$-$&$-$ \\\hline\end{tabular}
\end{center}
\begin{center}
\begin{tabular}{|c|c|c|c|}
\hline\small{$1$}&\small{$0$}&&\multirow{2}{*}{\small{$-v(1-v)z_j^{-n_Q}\zn$}}\\$+$&$-$&$-$&\\\hline
\end{tabular}
\hspace{0.5cm}
\begin{tabular}{|c|c|c|c|}
\hline\small{$0$}&\small{$0$}&&\multirow{2}{*}{\small{$-v(1-v)z_i^{-n_Q}\zn$}}\\$+$&$-$&$-$&\\\hline\small{$0$}&\small{$0$}&&\multirow{2}{*}{\small{$v(1-v)z_j^{-n_Q}(1-\zn)$}}\\$-$&$+$&$+$&\\\hline
\end{tabular}
\end{center}

\textbf{Case 31:} $(-,-,-,+,+,-)$. There are no admissible configurations in this case.\\

\textbf{Case 32:}

\vspace{-0.7cm}
\begin{center}
\begin{tabular}{|c|c|c|c|c|c|}\hline$0$&$0$&&$0$&$0$&\\$-$&$-$&$-$&$-$&$-$&$-$ \\\hline\end{tabular}
\end{center}
\begin{center}
\begin{tabular}{|c|c|c|c|}
\hline\small{$0$}&\small{$0$}&&\multirow{2}{*}{\small{$1-v\zn$}}\\$-$&$-$&$-$&\\\hline
\end{tabular}
\hspace{0.5cm}
\begin{tabular}{|c|c|c|c|}
\hline\small{$0$}&\small{$0$}&&\multirow{2}{*}{\small{$1-v\zn$}}\\$-$&$-$&$-$&\\\hline\end{tabular}
\end{center}

\bibliography{mybib}
\bibliographystyle{siam}

\end{document}